\documentclass{amsart}
\usepackage{amssymb}

\theoremstyle{plain}
\newtheorem{theorem}{Theorem}[section]
\newtheorem{proposition}[theorem]{Proposition}
\newtheorem{lemma}[theorem]{Lemma}
\newtheorem{claim}[theorem]{Claim}

\theoremstyle{definition}

\newtheorem{tab}[theorem]{Table}

\begin{document}

\title[Exceptional points]{Optimal quality of~exceptional points for~the~Lebesgue density theorem}
\author{Ond\v{r}ej Kurka}
\thanks{This research was supported in part by the grant GA\v{C}R 201/09/0067 and in part by the grant SVV-2010-261316.}
\address{Department of Mathematical Analysis, Faculty of Mathematics and Physics,
Charles University, Sokolovsk\'a 83, 186 75 Prague 8, Czech Republic}
\email{kurka.ondrej@seznam.cz}
\begin{abstract}
In~spite of~the~Lebesgue density theorem, there is a~positive $ \delta $ such that, for~every non-trivial measurable set $ S \subset \mathbb{R} $, there is a~point at~which both the~lower densities of~$ S $ and of~$ \mathbb{R} \setminus S $ are at~least $ \delta $. The~problem of~determining the~supremum of~possible values of~this $ \delta $ was studied in~a~paper of~V.~I.~Kolyada, as well as in~some recent papers. We solve this problem in~the~present work.
\end{abstract}
\maketitle

\section{Introduction and main result}

Let $ S $ be a~Lebesgue measurable subset of~$ \mathbb{R} $. The~well-known Lebesgue density theorem, in~a~weakened form, states that almost every point is either a~density point of~$ S $ or a~density point of~$ \mathbb{R} \setminus S $. We say that a~point is an~\emph{exceptional point} for~$ S $ if it is neither a~density point of~$ S $ nor a~density point of~$ \mathbb{R} \setminus S $. The~Lebesgue theorem states, in~other words, that the~set of~exceptional points for~$ S $ is null. The~question whether this set can be empty arises. Of~course, there is no exceptional point for~$ S $ if either $ S $ or $ \mathbb{R} \setminus S $ is null.

V.~I.~Kolyada \cite{kolyada} showed that, if $ S $ and $ \mathbb{R} \setminus S $ are of~positive measure, then an~exceptional point exists for~$ S $. Moreover, it is possible to~find a~point at~which both the~lower densities of~$ S $ and of~$ \mathbb{R} \setminus S $ are at~least $ 1/4 $. It was not clear whether the~constant $ 1/4 $ is the~best possible. Maybe unexpectedly, A.~Szenes \cite{szenes} proved that this constant can be improved, but the~problem of~determining the~supremum of~possible values remained open. This problem is solved in~the~present paper.

Let $ 0 \leq \delta \leq 1/2 $ and let $ S \subset \mathbb{R} $ be a Lebesgue measurable set. We will denote the Lebesgue measure by $ \lambda $. We say that $ S $ is \emph{non-trivial} if $ \lambda S > 0 $ and $ \lambda (\mathbb{R} \setminus S) > 0 $. We say that $ s \in \mathbb{R} $ is a \emph{$ \delta $-exceptional point} for $ S $ if
$$ \delta \leq \liminf _{\omega \rightarrow 0+} \frac{\lambda (S \cap (s - \omega , s + \omega ))}{2\omega } \leq \limsup _{\omega \rightarrow 0+} \frac{\lambda (S \cap (s - \omega , s + \omega ))}{2\omega } \leq 1 - \delta . $$
We will be studying the statement
$$ \mathcal{H}(\delta ) : \quad \textrm{There is a $ \delta $-exceptional point for every non-trivial $ S $.} $$
The main problem is to find the constant
$$ \delta _{\mathcal{H}} = \sup \{ \delta : \mathcal{H}(\delta ) \} . $$

It was shown by V.~I.~Kolyada \cite{kolyada} that
$$ 1/4 \leq \delta _{\mathcal{H}} \leq (\sqrt{17} - 3)/4 = 0,2807... \; . $$
Later, A.~Szenes \cite{szenes} improved both bounds and proved that
$$ 0,2629... \leq \delta _{\mathcal{H}} \leq 0,2718... $$
where the exact lower bound is a root of $ 4x^{3} + 2x^{2} + 3x - 1 $ and the exact upper bound is a root of $ 8x^{3} + 4x^{2} + 2x - 1 $. He conjectured that $ \delta _{\mathcal{H}} $ in fact equals to his upper bound. M.~Cs\"ornyei, J.~Grahl and T.~C.~O'Neil \cite{cgo} reduced the upper bound, showing that Szenes's conjecture is false. They proved that
$$ \delta _{\mathcal{H}} \leq 0,2710... $$
where the exact upper bound is a root of $ 2x^{3} + 2x^{2} + 3x - 1 $.

In this paper, we find the exact value of $ \delta _{\mathcal{H}} $.

\begin{theorem}
$ \delta _{\mathcal{H}} $ is the only real root of the polynomial $ 8x^{3} + 8x^{2} + x - 1 $. 
\end{theorem}

So, $ \delta _{\mathcal{H}} = 0,268486... $ . To prove this, we just combine Propositions \ref{restatement}, \ref{ubound} and \ref{lbound} below.

A.~Szenes formulated, based on an idea of M.~Laczkovich, a discrete restatement of the problem. We will need this restatement, as well as the authors needed in the articles \cite{szenes} and \cite{cgo}. By a \emph{configuration} we call a set
$$ C = (-\infty , 0) \cup (a_{1}, b_{1}) \cup (a_{2}, b_{2}) \cup \dots \cup (a_{r}, b_{r}), $$
where $ 0 < a_{1} < b_{1} < a_{2} < b_{2} < \dots < a_{r} < b_{r} $. The points $ 0, a_{1}, b_{1}, a_{2}, b_{2}, \dots , a_{r}, b_{r} $ are called the \emph{endpoints} of $ C $. Generally, by an endpoint of a set we mean a point from its boundary. By $ \mathcal{K}(\delta ) $ we denote the statement
\smallskip
\begin{itemize}
\item[$ \mathcal{K}(\delta ) : $] $ \quad $ For every configuration $ C $, there is an endpoint $ c $ of $ C $ such that
$$ \delta < \frac{\lambda (C \cap (c - \omega , c + \omega ))}{2\omega } < 1 - \delta \quad \textrm{for all $ \omega > 0 $.} $$
\end{itemize}
Analogously, we put
$$ \delta _{\mathcal{K}} = \sup \{ \delta : \mathcal{K}(\delta ) \} . $$

The problem of finding $ \delta _{\mathcal{H}} $ can be restated as the problem of finding $ \delta _{\mathcal{K}} $.

\begin{proposition}[{\cite[Proposition 3]{szenes}}] \label{restatement}
We have $ \delta _{\mathcal{H}} = \delta _{\mathcal{K}} $.
\end{proposition}

Let us introduce some notation now. Throughout the work, $ I_{\omega }(p) $ means the open interval $ (p - \omega , p + \omega ) $ and $ \lambda (A | B) $ means $ \lambda(A \cap B)/\lambda (B) $. By
$$ \lambda (A | I_{\gamma }(c)) \unrhd 1 - \delta $$
we mean
$$ \lambda (A | I_{\gamma }(c)) \geq 1 - \delta \quad \textrm{and} \quad 0 < \varepsilon < \gamma \Rightarrow \lambda (A | I_{\varepsilon }(c)) < 1 - \delta . $$
If $ a \leq b $, then $ (a, b) $ denotes $ \{ x \in \mathbb{R} : a < x < b \} $ (so, if $ a < b $, then $ (a, b) $ denotes the open interval as usual, and if $ a = b $, then $ (a, b) $ is empty).

\section{Upper bound}

In this section, we give an upper bound on $ \delta _{\mathcal{K}} $. Our method is based on the method given by Szenes in \cite{szenes} and improved by Cs\"ornyei, Grahl and O'Neil in \cite{cgo}. Let us recall the idea of these methods.

Szenes constructed a configuration depending on a couple of parameters. This configuration consists of $ (-\infty , 0) $ and a number of intervals of the same length uniformly distributed in the interval $ (m, 1) $, where $ m $ is a parameter with $ 0 < m < 1 $.
$$
\unitlength=1.00mm
\linethickness{1pt}
\begin{picture}(100,11)(0,-6)
\put(0.00,0.00){\line(1,0){32.77}}
\put(43.30,0.00){\line(1,0){9.77}}
\put(58.38,0.00){\line(1,0){9.77}}
\put(73.46,0.00){\line(1,0){9.77}}
\put(88.54,0.00){\line(1,0){9.77}}
\put(32.77,-3.00){\makebox(0,0){$ 0 $}}
\put(43.30,-3.00){\makebox(0,0){$ m $}}
\put(98.31,-3.00){\makebox(0,0){$ 1 $}}
\end{picture}
$$
He showed that, for any $ \delta $ with $ 8\delta ^{3} + 4\delta ^{2} + 2\delta > 1 $, a suitable choice of parameters gives a counterexample to $ \mathcal{K}(\delta ) $. This provides the bound $ \delta _{\mathcal{K}} \leq 0,271844... $ .

Szenes conjectured that his construction is optimal. This conjecture was disproved by Cs\"ornyei, Grahl and O'Neil. They modified Szenes's configuration by inserting a small gap into the intervals.
$$
\unitlength=1.00mm
\linethickness{1pt}
\begin{picture}(100,11)(0,-6)
\put(0.00,0.00){\line(1,0){32.77}}
\put(42.96,0.00){\line(1,0){4.54}}
\put(48.27,0.00){\line(1,0){4.54}}
\put(58.12,0.00){\line(1,0){4.54}}
\put(63.44,0.00){\line(1,0){4.54}}
\put(73.29,0.00){\line(1,0){4.54}}
\put(78.61,0.00){\line(1,0){4.54}}
\put(88.46,0.00){\line(1,0){4.54}}
\put(93.77,0.00){\line(1,0){4.54}}
\put(32.77,-3.00){\makebox(0,0){$ 0 $}}
\put(42.96,-3.00){\makebox(0,0){$ m $}}
\put(98.31,-3.00){\makebox(0,0){$ 1 $}}
\end{picture}
$$
Their method gives a counterexample to $ \mathcal{K}(\delta ) $ for any  $ \delta $ with $ 2\delta ^{3} + 2\delta ^{2} + 3\delta > 1 $. This provides the bound $ \delta _{\mathcal{K}} \leq 0,271069... $ .

Our method is very similar to the method of Cs\"ornyei, Grahl and O'Neil. The difference is in that we insert a gap into every other interval only. One may expect that all these methods are too weak and it will be necessary to construct more and more intricate configurations to obtain better and better upper bounds on $ \delta _{\mathcal{K}} $. In spite of this expectation, we will see later that our method gives the best possible upper bound (Proposition \ref{lbound}).

\begin{proposition} \label{ubound}
We have $ \delta _{\mathcal{K}} \leq \zeta _{1} $ where $ \zeta _{1} $ is the only real root of the polynomial $ 8x^{3} + 8x^{2} + x - 1 $. 
\end{proposition}

The remainder of the section is devoted to the proof of this proposition.

We show that there is a counterexample to $ \mathcal{K}(\delta ) $ for any $ \delta $ with $ 8\delta ^{3} + 8\delta ^{2} + \delta > 1 $ (i.e., $ \delta > \zeta _{1} = 0,2684... $). We will use the assumption $ 8\delta ^{3} + 8\delta ^{2} + \delta > 1 $ in the form
\begin{equation} \label{in00}
\frac{(1 - \delta )(1 + 2\delta )}{1 + 3\delta + 4\delta ^{2}} < 2\delta .
\end{equation}
Since a counterexample to $ \mathcal{K}(\delta ) $ is a counterexample to $ \mathcal{K}(\delta ') $ for every $ \delta ' > \delta $, we may assume that $ \delta \leq \frac{1}{4}(\sqrt{5} - 1) = 0,3090... $ .

We define
$$ \alpha = \frac{1 + 2\delta - 4\delta ^{2}}{4(1 + 3\delta )}, \quad \beta = \frac{2\delta ^{2}}{1 + 3\delta }, \quad \varphi = \frac{1}{2(1 + 3\delta )}, \quad \psi = \frac{\delta }{1 + 3\delta }, $$
\begin{align*}
k_{1} = 0, & \quad l_{1} = \varphi , \\
k_{2} = \varphi + \psi , & \quad l_{2} = \varphi + \psi + \alpha , \\
\quad k_{3} = \varphi + \psi + \alpha + \beta , & \quad l_{3} = \varphi + \psi + \alpha + \beta + \alpha ,
\end{align*}
$$ S_{N} = \bigcup_{n=0}^{N-1} \big\{ n + \big[ (k_{1}, l_{1}) \cup (k_{2}, l_{2}) \cup (k_{3}, l_{3}) \big] \big\} \cup \big\{ N + (k_{1}, l_{1}) \big\} , \quad N \in \mathbb{N}. $$
Note that $ \alpha > 0 $ and that
$$ \varphi + \psi + \alpha + \beta + \alpha + \psi = 1, $$
so the distance of $ l_{3} $ to $ 1 $ is $ \psi $. Let us look how the set $ S_{N} $ looks like for $ N = 3 $.
$$
\unitlength=1.00mm
\linethickness{1pt}
\begin{picture}(100,13)(0,-6)
\put(0.00,0.00){\line(1,0){8.31}}
\put(12.77,0.00){\line(1,0){5.19}}
\put(20.35,0.00){\line(1,0){5.19}}
\put(4.15,3.00){\makebox(0,0){$ \varphi $}}
\put(10.54,3.00){\makebox(0,0){$ \psi $}}
\put(15.36,3.00){\makebox(0,0){$ \alpha $}}
\put(19.15,3.00){\makebox(0,0){$ \beta $}}
\put(22.95,3.00){\makebox(0,0){$ \alpha $}}
\put(27.77,3.00){\makebox(0,0){$ \psi $}}
\put(30.00,0.00){\line(1,0){8.31}}
\put(42.77,0.00){\line(1,0){5.19}}
\put(50.35,0.00){\line(1,0){5.19}}
\put(34.15,3.00){\makebox(0,0){$ \varphi $}}
\put(40.54,3.00){\makebox(0,0){$ \psi $}}
\put(45.36,3.00){\makebox(0,0){$ \alpha $}}
\put(49.15,3.00){\makebox(0,0){$ \beta $}}
\put(52.95,3.00){\makebox(0,0){$ \alpha $}}
\put(57.77,3.00){\makebox(0,0){$ \psi $}}
\put(60.00,0.00){\line(1,0){8.31}}
\put(72.77,0.00){\line(1,0){5.19}}
\put(80.35,0.00){\line(1,0){5.19}}
\put(64.15,3.00){\makebox(0,0){$ \varphi $}}
\put(70.54,3.00){\makebox(0,0){$ \psi $}}
\put(75.36,3.00){\makebox(0,0){$ \alpha $}}
\put(79.15,3.00){\makebox(0,0){$ \beta $}}
\put(82.95,3.00){\makebox(0,0){$ \alpha $}}
\put(87.77,3.00){\makebox(0,0){$ \psi $}}
\put(90.00,0.00){\line(1,0){8.31}}
\put(94.15,3.00){\makebox(0,0){$ \varphi $}}
\put(0.00,-3.00){\makebox(0,0){$ 0 $}}
\put(99.00,-3.00){\makebox(0,0){$ N \! + \! \varphi $}}
\end{picture}
$$
One \textquotedblleft period\textquotedblright {} of $ S_{N} $ consists of two smaller intervals of length $ \alpha $ and one larger interval of length $ \varphi $. The gap between two smaller intervals has length $ \beta $ and the gap between a smaller and the neighbouring larger interval has length $ \psi $.

\begin{claim} \label{uboundclaim}
Every endpoint $ p $ of $ S_{N} $ with $ 0 < p < N + \varphi $ has a radius $ \varepsilon > 0 $ such that $ \lambda(S_{N} | I_{\varepsilon }(p)) \geq 1 - \delta $.
\end{claim}

Suitable intervals for each of the three possible types of endpoints are shown in the picture.
$$
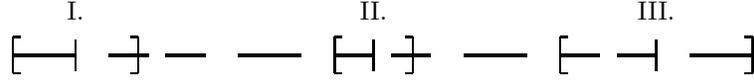
\unitlength=1.00mm
\begin{picture}(100,13)(0,-4)
\linethickness{0.4pt}
\put(0.00,-2.4){\line(0,1){4.8}}
\put(0.00,-2.4){\line(1,0){1}}
\put(0.00,2.4){\line(1,0){1}}
\put(16.62,-2.4){\line(0,1){4.8}}
\put(16.62,-2.4){\line(-1,0){1}}
\put(16.62,2.4){\line(-1,0){1}}
\put(42.77,-2.4){\line(0,1){4.8}}
\put(42.77,-2.4){\line(1,0){1}}
\put(42.77,2.4){\line(1,0){1}}
\put(53.15,-2.4){\line(0,1){4.8}}
\put(53.15,-2.4){\line(-1,0){1}}
\put(53.15,2.4){\line(-1,0){1}}
\put(72.77,-2.4){\line(0,1){4.8}}
\put(72.77,-2.4){\line(1,0){1}}
\put(72.77,2.4){\line(1,0){1}}
\put(98.31,-2.4){\line(0,1){4.8}}
\put(98.31,-2.4){\line(-1,0){1}}
\put(98.31,2.4){\line(-1,0){1}}
\linethickness{0.6pt}
\put(8.31,-2){\line(0,1){4}}
\put(47.96,-2){\line(0,1){4}}
\put(85.54,-2){\line(0,1){4}}
\linethickness{1pt}
\put(0.00,0.00){\line(1,0){8.31}}
\put(12.77,0.00){\line(1,0){5.19}}
\put(20.35,0.00){\line(1,0){5.19}}
\put(30.00,0.00){\line(1,0){8.31}}
\put(42.77,0.00){\line(1,0){5.19}}
\put(50.35,0.00){\line(1,0){5.19}}
\put(60.00,0.00){\line(1,0){8.31}}
\put(72.77,0.00){\line(1,0){5.19}}
\put(80.35,0.00){\line(1,0){5.19}}
\put(90.00,0.00){\line(1,0){8.31}}
\put(8.31,6.00){\makebox(0,0){I.}}
\put(47.96,6.00){\makebox(0,0){II.}}
\put(85.54,6.00){\makebox(0,0){III.}}
\end{picture}
$$

\begin{proof}
I. Let $ p = l_{1} + n $ or $ p = k_{1} + n $ for some $ n \in \mathbb{N} \cup \{ 0 \} $. We show that the radius
$$ \varepsilon = \varphi $$
works. One may easily check that $ \psi \leq \varphi \leq \psi + \alpha $ due to $ \frac{1}{4}(3 - \sqrt{5}) \leq \zeta _{1} < \delta $. So, the intersection $ I_{\varepsilon }(p) \cap S_{N} $ consists of one interval of length $ \varphi $ and one interval of length $ \varphi - \psi $. We obtain
$$ \lambda(S_{N} | I_{\varepsilon }(p)) = \frac{\varphi + (\varphi - \psi )}{2\varphi } = 1 - \frac{\psi }{2\varphi } = 1 - \delta . $$

II. Let $ p = l_{2} + n $ or $ p = k_{3} + n $ for some $ n \in \mathbb{N} \cup \{ 0 \} $. We show that the radius
$$ \varepsilon = \alpha $$
works. One may easily check that $ \beta /(2\alpha ) \leq \delta $ due to $ \delta \leq \frac{1}{4}(\sqrt{5} - 1) $. In particular, $ \beta \leq 2\delta \alpha \leq \alpha $. Hence,
$$ \lambda(S_{N} | I_{\varepsilon }(p)) = \frac{\alpha + (\alpha - \beta )}{2\alpha } = 1 - \frac{\beta }{2\alpha } \geq 1 - \delta . $$

III. Let $ p = l_{3} + n $ or $ p = k_{2} + n $ for some $ n \in \mathbb{N} \cup \{ 0 \} $. Then the radius
$$ \varepsilon = 2\alpha + \beta = \psi + \varphi $$
works, as
$$ \lambda(S_{N} | I_{\varepsilon }(p)) = \frac{2\alpha + \varphi }{2(\psi + \varphi )} = \frac{\frac{1}{2}(1 + 2\delta - 4\delta ^{2}) + \frac{1}{2} }{1 + 2\delta } = 1 - \delta . $$
\end{proof}

Now, we define
$$ m = \frac{4\delta ^{2}}{1 + 3\delta + 4\delta ^{2}}. $$
For every $ N \in \mathbb{N} $, let $ u_{N} : \mathbb{R} \rightarrow \mathbb{R} $ be the affine transformation which maps $ 0 $ onto $ m $ and $ N + \varphi $ onto $ 1 $. We define
$$ C_{N} = (-\infty , 0) \cup u_{N}(S_{N}), \quad N \in \mathbb{N}. $$
We see $ C_{3} $ on the picture.
$$
\unitlength=1.00mm
\linethickness{1pt}
\begin{picture}(100,11)(0,-6)
\put(0.00,0.00){\line(1,0){32.77}}
\put(41.79,0.00){\line(1,0){4.78}}
\put(49.14,0.00){\line(1,0){2.98}}
\put(53.49,0.00){\line(1,0){2.98}}
\put(59.04,0.00){\line(1,0){4.78}}
\put(66.38,0.00){\line(1,0){2.98}}
\put(70.74,0.00){\line(1,0){2.98}}
\put(76.29,0.00){\line(1,0){4.78}}
\put(83.63,0.00){\line(1,0){2.98}}
\put(87.99,0.00){\line(1,0){2.98}}
\put(93.53,0.00){\line(1,0){4.78}}
\put(32.77,-3.00){\makebox(0,0){$ 0 $}}
\put(41.79,-3.00){\makebox(0,0){$ m $}}
\put(98.31,-3.00){\makebox(0,0){$ 1 $}}
\end{picture}
$$

\begin{claim}
$ C_{N} $ is a counterexample to $ \mathcal{K}(\delta ) $ for a large enough $ N $.
\end{claim}

\begin{proof}
For $ N \in \mathbb{N} $, we have
\begin{eqnarray*}
\frac{1}{N + \varphi } \lambda S_{N} & = & \frac{N(\varphi + 2\alpha ) + \varphi }{N + \varphi } = \varphi + 2\alpha + \frac{\varphi (1 - \varphi - 2\alpha )}{N + \varphi } \\
 & \geq & \varphi + 2\alpha = \frac{(1 - \delta )(1 + 2\delta )}{1 + 3\delta },
\end{eqnarray*}
and so
$$ \lambda (u_{N}(S_{N})) = \frac{1 - m}{N + \varphi } \lambda S_{N} \geq \frac{1 + 3\delta }{1 + 3\delta + 4\delta ^{2}} \frac{(1 - \delta )(1 + 2\delta )}{1 + 3\delta } = \frac{(1 - \delta )(1 + 2\delta )}{1 + 3\delta + 4\delta ^{2}}. $$
Moreover,
\begin{eqnarray*}
\lim _{N \rightarrow \infty } \lambda (u_{N}(S_{N})) & = & \lim _{N \rightarrow \infty } \frac{1 - m}{N + \varphi } \lambda S_{N} = \lim _{N \rightarrow \infty } (1 - m) \frac{N(\varphi + 2\alpha ) + \varphi }{N + \varphi } \\
 & = & (1 - m)(\varphi + 2\alpha ) = \frac{1 + 3\delta }{1 + 3\delta + 4\delta ^{2}} \frac{(1 - \delta )(1 + 2\delta )}{1 + 3\delta } \\
 & = & \frac{(1 - \delta )(1 + 2\delta )}{1 + 3\delta + 4\delta ^{2}}.
\end{eqnarray*}
It follows from (\ref{in00}) that, for a large enough $ N $,
$$ \lambda (u_{N}(S_{N})) \leq 2\delta . $$
We show that $ C_{N} $ is a counterexample to $ \mathcal{K}(\delta ) $ for such an $ N $.

By Claim \ref{uboundclaim}, every endpoint $ p $ of $ C_{N} $ with $ m < p < 1 $ has a radius $ \varepsilon > 0 $ such that $ \lambda(C_{N} | I_{\varepsilon }(p)) \geq 1 - \delta $. It remains to find a radius $ \varepsilon > 0 $ such that $ \lambda(C_{N} | I_{\varepsilon }(p)) \notin (\delta , 1 - \delta ) $ for $ p = 0 $, $ p = m $ and $ p = 1 $. Note that it is easy to show that $ m \leq 1/2 $. The radius $ 1 - m $ works for $ m $, as
\begin{eqnarray*}
\lambda (C_{N} | I_{1 - m}(m)) & = & \frac{1}{2(1 - m)} \Big( \lambda (m - (1 - m), 0) + \lambda (u_{N}(S_{N})) \Big) \\
 & \geq & \frac{1}{2(1 - m)} \Big( - m + (1 - m) + \frac{(1 - \delta )(1 + 2\delta )}{1 + 3\delta + 4\delta ^{2}} \Big) \\
 & = & 1 - \frac{1}{2(1 - m)} \Big( 1 - \frac{(1 - \delta )(1 + 2\delta )}{1 + 3\delta + 4\delta ^{2}} \Big) \\
 & = & 1 - \delta ,
\end{eqnarray*}
and the same radius works for $ 0 $ because clearly $ \lambda (C_{N} | I_{1 - m}(0)) \geq \lambda (C_{N} | I_{1 - m}(m)) $. Finally, the radius $ 1 $ works for $ 1 $, as
$$ \lambda (C_{N} | I_{1}(1)) = \frac{1}{2} \lambda (u_{N}(S_{N})) \leq \frac{1}{2} 2\delta = \delta . $$
\end{proof}

\section{Informal notes concerning lower bound}

The proof of the optimal lower bound on $ \delta _{\mathcal{K}} $, presented in Sections 4--7, is quite long and technical. We decided to write these notes with hope to make the proof more accessible. The notes may perhaps interest also a reader who will not read the proof but wants to know its main ideas.

We point out that the sections of the proof are ordered in the way that the auxiliary results are proved first. It does not mean that the proofs of the auxiliary results must be read first. It is possible that the reader will rather read Section 5 earlier than Section 4 and Section 7 earlier than Section 6.

It is natural to ask if there are some simplifications of the proof. It is possible that a smart idea can simplify a part of the proof or make a part unnecessary, respectively. Nevertheless, we do not see any way how to come to such an idea, so we can not give a hint. At the same time, we are quite sceptical of finding a simple or short proof. Such a proof, as well as a new and completely different proof, would be very surprising.

The proof consists of two parts, let us call them (P1) and (P2). In the part (P1), we make an inspection of Szenes's methods from \cite{szenes}. Assuming that $ \mathcal{K}(\delta ) $ does not hold for a $ \delta $, we construct so-called $ \delta $-good set $ G $ with an upper bound on its measure (Lemma \ref{goodset}). In the part (P2), we find the best possible lower bound on the measure of a $ \delta $-good set (Proposition \ref{goodsetbound}). In the end, confronting the bounds on the measure of $ G $ leads to a lower bound on $ \delta $ (proof of Proposition \ref{lbound}). The same bound is true for $ \delta _{\mathcal{K}} $, as $ \delta > \delta _{\mathcal{K}} $ can be chosen arbitrarily.

In the matter of the omnipresent number $ \delta $, imagine $ \delta $ located close to $ \delta _{\mathcal{K}} $. As $ \delta $ increases, the things are running out of control. For this reason, we often meet the strange assumptions $ \delta < \zeta _{i} $ where $ \zeta _{i} $ are defined in Table \ref{roots}.

Let us briefly outline the two parts of the proof.

(P1) Let $ \delta $ be such that $ \mathcal{K}(\delta ) $ does not hold. There is a counterexample to $ \mathcal{K}(\delta ) $, i.e., a configuration
$$ C = (-\infty , 0) \cup (a_{1}, b_{1}) \cup \dots \cup (a_{r}, b_{r} = 1) $$
such that, for every its endpoint $ c $, there is $ \omega > 0 $ with $ \lambda (C | I_{\omega }(c)) \notin (\delta , 1 - \delta ) $. Let us moreover assume that $ C $ is a counterexample to $ \mathcal{K}(\delta ) $ with the least possible number of intervals in it.

Szenes \cite{szenes} found [after an eventual changing of $ C $ with its \textquotedblleft inverse\textquotedblright {} configuration $ 1 - (\mathbb{R} \setminus (C \cup \partial C)) $] an interval $ J_{\mathcal{B}} \subset (0, 1) $ and points $ v_{\mathcal{B}}, v_{\mathcal{W}} \in \overline{J_{\mathcal{B}}} $ such that the set
$$ D = C \cap J_{\mathcal{B}} $$
has properties (among other ones)
\begin{itemize}
\item[{(a)}] for every endpoint $ c $ of $ D $ with $ v_{\mathcal{B}} < c < v_{\mathcal{W}} $, there is a radius $ \omega > 0 $ such that $ I_{\omega }(c) \subset J_{\mathcal{B}} $ and $ \lambda (D | I_{\omega }(c)) \notin (\delta , 1 - \delta ) $,
\item[{(b)}] $ J_{\mathcal{B}} $ is covered by a finite number of intervals $ I $ with $ \lambda (D | I) \geq 1 - \delta $,
\item[{(c)}] $ \lambda D \leq \frac{4\delta ^{2}}{1 - 2\delta } (v_{\mathcal{W}} - v_{\mathcal{B}}) $.
\end{itemize}
It should be mentioned that, by \cite[Lemma 13]{szenes}, property (b) implies
$$ \lambda (D | J_{\mathcal{B}}) \geq \frac{1 - \delta }{1 + \delta }, \leqno (*) $$
so, using (c), we can compute
$$ \frac{4\delta ^{2}}{1 - 2\delta } (v_{\mathcal{W}} - v_{\mathcal{B}}) \geq \lambda D \geq \frac{1 - \delta }{1 + \delta } \lambda (J_{\mathcal{B}}) \geq \frac{1 - \delta }{1 + \delta } (v_{\mathcal{W}} - v_{\mathcal{B}}). $$
Since $ \delta > \delta _{\mathcal{K}} $ could be chosen arbitrarily, we obtain
$$ \frac{4\delta _{\mathcal{K}}^{2}}{1 - 2\delta _{\mathcal{K}}} \geq \frac{1 - \delta _{\mathcal{K}}}{1 + \delta _{\mathcal{K}}}, $$ 
which is nothing else than the Szenes bound $ \delta _{\mathcal{K}} \geq 0,2629... $ \cite[Theorem 5]{szenes}.

Szenes himself suggested to improve the lower bound on $ \delta _{\mathcal{K}} $ by improving the inequality ($ * $). We follow this idea in a manner. Imagine that we are able to prove
$$ \lambda D \geq \frac{(1 - \delta )(1 + 2\delta )}{1 + 3\delta } (v_{\mathcal{W}} - v_{\mathcal{B}}). \leqno (**) $$
Analogously as above, we write
$$ \frac{4\delta ^{2}}{1 - 2\delta } (v_{\mathcal{W}} - v_{\mathcal{B}}) \geq \lambda D \geq \frac{(1 - \delta )(1 + 2\delta )}{1 + 3\delta } (v_{\mathcal{W}} - v_{\mathcal{B}}) $$
and obtain
$$ \frac{4\delta _{\mathcal{K}}^{2}}{1 - 2\delta _{\mathcal{K}}} \geq \frac{(1 - \delta _{\mathcal{K}})(1 + 2\delta _{\mathcal{K}})}{1 + 3\delta _{\mathcal{K}}}. $$
This is nothing else than the desired bound from Proposition \ref{lbound}. So, proving ($ ** $) would confirm that the construction in Section 2 was optimal. Note that this is not the first time we meet the quantity $ (1 - \delta )(1 + 2\delta )/(1 + 3\delta ) $. It equals to the measure of one \textquotedblleft period\textquotedblright {} of the set $ S_{N} $ from Section 2.

We choose a bit different way from proving ($ ** $). The set $ D $ is a complicated object with a complicated history of construction. Moreover, property (a), which should be one of its fundamental properties, deals only with the endpoints $ c $ with $ v_{\mathcal{B}} < c < v_{\mathcal{W}} $, and not with all the endpoints from the interior of $ J_{\mathcal{B}} $. Therefore, in several steps, we modify the set $ D $ and obtain a relatively easily definable set where property (a) is \textquotedblleft repaired\textquotedblright {} and the other important properties of $ D $ are preserved. The result of our construction is so-called $ \delta $-good set obtained in Lemma \ref{goodset}. The inequality corresponding to ($ ** $) occurs in Proposition \ref{goodsetbound} which is the main result of the part (P2) of our proof, discussed below.

We do not assert that proving ($ ** $) is not the right way. Lemma \ref{lemmaxy} says exactly what has to be done. The only problem is that we are not able to prove property (C) from the lemma (it is possible that a simplification of the proof of Lemma \ref{keyprop} will enable to prove ($ ** $) and make the notion of a $ \delta $-good set unnecessary).

As we indicated, we would be satisfied with property (a) if $ (v_{\mathcal{B}}, v_{\mathcal{W}}) = J_{\mathcal{B}} $ (although it is a plausible hypothesis that always $ (v_{\mathcal{B}}, v_{\mathcal{W}}) = J_{\mathcal{B}} $ for the objects $ v_{\mathcal{B}}, v_{\mathcal{W}}, J_{\mathcal{B}} $ how Szenes defined them, we were not able to prove it). Motivated by this, we consider the following idea. Let $ E $ be given by
$$ E = (q, p) \cup (D \cap [p, p']) \cup (p', q') $$
for some $ q < p < p' < q' $ with $ v_{\mathcal{B}} < p $ and $ p' < v_{\mathcal{W}} $. One may hope that a suitable choice of the parameters will guarantee that, analogously as above,
\begin{itemize}
\item[{(a)}] for every endpoint $ c $ of $ E $ with $ q < c < q' $, there is a radius $ \omega > 0 $ such that $ I_{\omega }(c) \subset (q, q') $ and $ \lambda (E | I_{\omega }(c)) \notin (\delta , 1 - \delta ) $,
\item[{(b)}] $ (q, q') $ is covered by a finite number of intervals $ I $ with $ \lambda (E | I) \geq 1 - \delta $,
\item[{(c)}] $ \lambda E \leq \frac{4\delta ^{2}}{1 - 2\delta } (q' - q) $
\end{itemize}
(the properties correspond to the situation $ (v_{\mathcal{B}}, v_{\mathcal{W}}) = J_{\mathcal{B}} $).

In our proof, we introduce numbers $ q < p < p' < q' $ such that (a) is satisfied for $ E $. We were not able to prove (c) for these numbers. Nevertheless, we discovered that a more relaxed condition
\begin{itemize}
\item[{(c')}] $ \lambda (E | [\frac{1}{2}(q + p), \frac{1}{2}(p' + q')]) \leq \frac{4\delta^{2}}{1 - 2\delta} $
\end{itemize}
is satisfied and that this condition is sufficient for our purposes. We use a little trick how to deal with the excessive pieces $ (q, \frac{1}{2}(q + p)) $ and $ (\frac{1}{2}(p' + q'), q') $. Imagine that we place two mirrors to the points $ \frac{1}{2}(q + p) $ and $ \frac{1}{2}(p' + q') $. We obtain a periodic set with the measure of one period less or equal to $ \frac{4\delta^{2}}{1 - 2\delta} $. A \textquotedblleft global\textquotedblright {} analogue of property (a) holds. Such a periodic set is the result of the first part of our proof (see Lemma \ref{goodset}).

It should be pointed out that the minimality of $ C $ is a fundamental aspect of our proof. Recall that the configuration $ C $ was chosen so that
\begin{itemize}
\item[{(i)}] $ C $ is a counterexample to $ \mathcal{K}(\delta ) $, i.e., for every endpoint $ c $ of $ C $, there is $ \omega > 0 $ with $ \lambda (C | I_{\omega }(c)) \notin (\delta , 1 - \delta ) $,
\item[{(ii)}] $ C $ is a counterexample to $ \mathcal{K}(\delta ) $ with the minimal possible number of intervals in it, i.e., every configuration $ C' $ which consists of less intervals than $ C $ has an endpoint $ c' $ such that $ \lambda (C' | I_{\omega }(c')) \in (\delta , 1 - \delta ) $ for every $ \omega > 0 $.
\end{itemize}
Notice that these conditions are coherent with the conditions (i) and (ii) from the definition of a $ \delta $-good set.

(P2) The key argument of the proof is Lemma \ref{lemmaxy} which is an abstract version of our lower bound on the measure of a $ \delta $-good set (Proposition \ref{goodsetbound}). It is possible to say that the difference between Szenes's and our method inheres in the difference between \cite[Lemma 13]{szenes} and Lemma \ref{lemmaxy}. Incidentally, these two methods correspond to the inequalities ($ * $) and ($ ** $) discussed above.

Let us make a little confrontation of the methods. Let $ I $ be an open interval and $ D \subset \mathbb{R} $ be a measurable set (typically, $ D $ is a finite union of intervals). Assume that $ I $ satisfies
$$ I = \bigcup _{k=1}^{n} I_{k} $$
where $ I_{k} $ are open intervals with
$$ \lambda (D | I_{k}) \geq 1 - \delta . $$
Then \cite[Lemma 13]{szenes} says that
$$ \lambda (D | I) \geq \frac{1 - \delta }{1 + \delta } \quad \textrm{(cf. $ (*) $)}. \leqno (+) $$
Without loss of generality, we will assume that the intervals are ordered from left to right so that $ I_{k} \cap I_{k+1} \neq \emptyset $ but $ \overline{I_{k}} \cap \overline{I_{k+2}} = \emptyset $. In other words,
$$ v_{0} < v_{1} < u_{2} < v_{2} < \dots < u_{k} < v_{k} < \dots < u_{n-1} < v_{n-1} < u_{n} < u_{n+1} $$
where $ u_{k}, v_{k} $ denote the endpoints of the intervals in the manner that
$$ I_{k} = (v_{k-1}, u_{k+1}). $$
We can imagine that the relative measures $ \lambda (D | (u_{k}, v_{k})) $ and $ \lambda ((\mathbb{R} \setminus D) | (v_{k}, u_{k+1})) $ are small.

We would like to improve the inequality $ (+) $. Realize that it can not be improved for a general $ D $. It is sufficient to consider the set $ D $ consisting of a large number of uniformly distributed intervals of length $ 1 - \delta $ such that the gap between two neighbouring intervals has length $ 2\delta $. Nevertheless, such an example is far away from the objects we work with.

We deal with a more sophisticated system of intervals than the system $ I_{k} $ above (see Claim \ref{clb}). We construct points
$$ v_{0} < u_{1} < v_{1} < u_{2} < v_{2} < \dots < u_{k} < v_{k} < \dots < u_{n-1} < v_{n-1} < u_{n} < v_{n} < u_{n+1} $$
and, for $ k = 1, 2, \dots , n $, two intervals $ I_{\alpha _{k}}(a_{k}), I_{\beta _{k}}(b_{k}) $ with
$$ I_{\alpha _{k}}(a_{k}) \cup I_{\beta _{k}}(b_{k}) = (v_{k-1}, u_{k+1}). $$
We introduce here only the two most important properties of the system (for the complete list, see Claim \ref{clb}). Similarly as above, the intervals satisfy
$$ \lambda (D | I_{\alpha _{k}}(a_{k})) \geq 1 - \delta , \quad \lambda (D | I_{\beta _{k}}(b_{k})) \geq 1 - \delta . $$
At the same time, the differences of the centers fulfill
$$ b_{k} - a_{k} \geq v_{k} - u_{k} - \frac{2}{1 - \delta } \lambda (D \cap (u_{k}, v_{k})). $$
In particular, the intervals $ I_{\alpha _{k}}(a_{k}) $ and $ I_{\beta _{k}}(b_{k}) $ do not coincide in the expected case that $ \lambda (D | (u_{k}, v_{k})) < (1 - \delta )/2 $.

The computation presented in Claims \ref{clc}--\ref{clxy} shows that the existence of such a system of intervals implies the desired inequality
$$ \lambda D \geq \frac{(1 - \delta )(1 + 2\delta )}{1 + 3\delta } \lambda (u_{1}, v_{n}) \quad \textrm{(cf. $ (**) $)}. \leqno (++) $$
In Lemma \ref{lemmaxy}, we introduce a property (C) from which the existence of the system can be derived. We note that, unfortunately, the proof that a $ \delta $-good set has this property (Lemma~\ref{keyprop}) is not an easy task.

\section{The Szenes method and its consequences}

The remainder of the paper is devoted to giving a lower bound on $ \delta _{\mathcal{K}} $. Although the lower bound given by Szenes in \cite{szenes} is not the best possible, we found his methods strongly useful and his results became a starting point of our proof. Let us recall some notation from \cite[Section 4]{szenes} first.

Let $ \delta \in (0, 1/2) $ be such that $ \mathcal{K}(\delta ) $ does not hold. Let
$$ C = (-\infty , 0) \cup (a_{1}, b_{1}) \cup \dots \cup (a_{r}, b_{r} = 1), $$
where $ 0 < a_{1} < b_{1} < a_{2} < b_{2} < \dots < a_{r} < b_{r} = 1 $, be a configuration which is a counterexample to $ \mathcal{K}(\delta ) $ with the least possible number of intervals in it.

For every endpoint $ p $ of $ C $, we define $ \omega (p) $ as the greatest radius such that $ \lambda (C | I_{\omega (p)}(p)) \notin (\delta , 1 - \delta ) $. We put
$$ \mathcal{B} = \{ p : \lambda (C | I_{\omega (p)}(p)) \geq 1 - \delta \} , \quad \mathcal{W} = \{ p : \lambda (C | I_{\omega (p)}(p)) \leq \delta \} , $$
$$ v_{\mathcal{B}} = \max \{ p \in \mathcal{B} : p \leq 1/2, \omega (p) \geq p \} , \quad v_{\mathcal{W}} = \min \{ p \in \mathcal{W} : p \geq 1/2, \omega (p) \geq 1 - p \} , $$
$$ \rho = \lambda (C \cap (0, 1)). $$

We resume some results from \cite{szenes} in the following proposition (see \cite[Corollaries 8 and 9]{szenes}).

\begin{proposition}[Szenes] \label{propszenes}
We have
$$ 1 - \rho \leq 2\delta (1 - v_{\mathcal{B}}), \quad \rho \leq 2\delta v_{\mathcal{W}} \quad \textrm{and} \quad v_{\mathcal{W}} - v_{\mathcal{B}} \geq \frac{1}{2\delta } - 1. $$
\end{proposition}

We apply these results in the proof of the following lemma.

\begin{lemma} \label{lemmaa}
Let $ 0 < \delta < 1/2 $. If $ \mathcal{K}(\delta ) $ does not hold, then there are numbers $ a, b, 0 \leq a < b \leq 1, $ and a set
$$ D = (r_{1}, s_{1}) \cup (r_{2}, s_{2}) \cup \dots \cup (r_{n}, s_{n} = 1), $$
where $ 0 < r_{1} < s_{1} < r_{2} < s_{2} < \dots < r_{n} < s_{n} = 1 $, such that the following conditions are satisfied:

{\rm (a)} for every endpoint $ p $ of $ D $ with $ a < p < b $, there is a radius $ \mu > 0 $ such that $ I_{\mu }(p) \subset (0, 1) $ and $ \lambda (D | I_{\mu }(p)) \notin (\delta , 1 - \delta ) $,

{\rm (b)} $ D $ consists of less intervals than any counterexample to $ \mathcal{K}(\delta ) $,

{\rm (c)} $ \lambda D \leq \frac{4\delta ^{2}}{1 - 2\delta } (b - a) $ and $ \lambda ((0, 1) \setminus D) \leq \frac{4\delta ^{2}}{1 - 2\delta } (b - a) $.
\end{lemma}

\begin{proof}
We work with the notation from \cite{szenes} mentioned above. Let us verify that the choice
$$ D = C \cap (0, 1), \quad a = v_{\mathcal{B}}, \quad b = v_{\mathcal{W}} $$
works. We verify (a) only for $ p \in \mathcal{B} $, the proof for $ p \in \mathcal{W} $ is similar. So, let $ p \in \mathcal{B} $ and $ a < p < b $. If $ p \leq 1/2 $, then, by $ p > a = v_{\mathcal{B}} $ and the definition of $ v_{\mathcal{B}} $, we have $ \omega (p) < p $, and thus $ I_{\omega (p)}(p) \subset (0, 1) $. So, we can take $ \mu = \omega (p) $. If $ p > 1/2 $, then we consider two cases. In the case $ 1 \notin I_{\omega (p)}(p) $, we can take $ \mu = \omega (p) $ (as $ I_{\omega (p)}(p) \subset (0, 1) $). In the case $ 1 \in I_{\omega (p)}(p) $, we can take $ \mu = 1 - p $ (as $ I_{1 - p}(p) \subset (0, 1) $ and $ \lambda (C | (I_{\omega (p)}(p) \setminus I_{1 - p}(p))) \leq 1/2 \leq 1 - \delta \leq \lambda (C | I_{\omega (p)}(p)) $).

Further, (b) is clear, as $ D $ consists of less intervals than $ C $ and $ C $ is a counterexample to $ \mathcal{K}(\delta ) $ with the least possible number of intervals in it. Using Proposition \ref{propszenes}, we can compute
$$ \lambda D = \rho \leq 2\delta v_{\mathcal{W}} \leq 2\delta , $$
$$ \lambda ((0, 1) \setminus D) = 1 - \rho \leq 2\delta (1 - v_{\mathcal{B}}) \leq 2\delta $$
and
$$ 2\delta = \frac{4\delta ^{2}}{1 - 2\delta } \Big( \frac{1}{2\delta } - 1 \Big) \leq \frac{4\delta ^{2}}{1 - 2\delta } (v_{\mathcal{W}} - v_{\mathcal{B}}) = \frac{4\delta ^{2}}{1 - 2\delta } (b - a). $$
This gives (c).
\end{proof}

We proceed to a more thoroughful look on the set from Lemma \ref{lemmaa}.

\begin{lemma} \label{lemmab}
Let $ 0 < \delta < \frac{1}{4}(\sqrt{5} - 1) $. If $ \mathcal{K}(\delta ) $ does not hold, then there are numbers $ a, b $, intervals $ I_{\varepsilon (a')}(a'), I_{\varepsilon (b')}(b') $ and a set
$$ D = (r_{1}, s_{1}) \cup (r_{2}, s_{2}) \cup \dots \cup (r_{n}, s_{n}), $$
where $ r_{1} < s_{1} < r_{2} < s_{2} < \dots < r_{n} < s_{n} $, such that the following conditions are satisfied:

{\rm (i)} $ r_{1} \leq a < b \leq s_{n} $ and $ a', b' $ are endpoints of $ D $ with $ a < a' < b $, $ a < b' < b $,

{\rm (ii)} $ D $ consists of less intervals than any counterexample to $ \mathcal{K}(\delta ) $,

{\rm (iii)} for every endpoint $ p $ of $ D $ with $ a < p < b $, there is a radius $ \omega > 0 $ such that $ I_{\omega }(p) \subset (r_{1}, s_{n}) $ and $ \lambda (D | I_{\omega }(p)) \notin (\delta , 1 - \delta ) $,

{\rm (iv)} $ I_{\varepsilon (a')}(a') \subset (r_{1}, s_{n}) $ and $ I_{\varepsilon (b')}(b') \subset (r_{1}, s_{n}) $,

{\rm (v)} $ \lambda (D | I_{\varepsilon (a')}(a')) \geq 1 - \delta $ and $ \lambda (D | I_{\varepsilon (a')}(a')) \geq \lambda (D | I_{\varepsilon }(a')) $ whenever $ 0 < \varepsilon < \varepsilon (a') $, $ \lambda (D | I_{\varepsilon (b')}(b')) \geq 1 - \delta $ and $ \lambda (D | I_{\varepsilon (b')}(b')) \geq \lambda (D | I_{\varepsilon }(b')) $ whenever $ 0 < \varepsilon < \varepsilon (b') $,

{\rm (vi)} $ I_{\varepsilon (a')}(a') \cap (r_{1}, s_{1}) \neq \emptyset \neq I_{\varepsilon (b')}(b') \cap (r_{n}, s_{n}) $,

{\rm (vii)} $ \lambda D \leq \frac{4\delta ^{2}}{1 - 2\delta } (b - a) $.
\end{lemma}

To prove the lemma, we need some notation and a claim first. Let $ a, b $ and $ D $ be as in Lemma \ref{lemmaa}. By (a), there is, for every endpoint $ p $ of $ D $ with $ a < p < b $, a radius $ \mu (p) > 0 $ such that $ I_{\mu (p)}(p) \subset (0, 1) $ and $ \lambda (D | I_{\mu (p)}(p)) \notin (\delta , 1 - \delta ) $. We will assume that this radius is chosen so that $ \lambda (D | I_{\mu (p)}(p)) $ is maximal possible when $ \lambda (D | I_{\mu (p)}(p)) \geq 1 - \delta $ and minimal possible when $ \lambda (D | I_{\mu (p)}(p)) \leq \delta $. We put
$$ \mathcal{I}_{\mathcal{B}} = \bigcup \big\{ I_{\mu (p)}(p) : \lambda (D | I_{\mu (p)}(p)) \geq 1 - \delta \big\} , $$
$$ \mathcal{I}_{\mathcal{W}} = \bigcup \big\{ I_{\mu (p)}(p) : \lambda (D | I_{\mu (p)}(p)) \leq \delta \big\} . $$

\begin{claim} \label{claimszenes}
If $ \mathcal{I}_{\mathcal{B}} $ and $ \mathcal{I}_{\mathcal{W}} $ are non-empty, then one of the two following possibilities takes place:
$$ \inf \mathcal{I}_{\mathcal{B}} < \inf \mathcal{I}_{\mathcal{W}} \quad \textrm{and} \quad \sup \mathcal{I}_{\mathcal{W}} < \sup \mathcal{I}_{\mathcal{B}}, $$
$$ \inf \mathcal{I}_{\mathcal{W}} < \inf \mathcal{I}_{\mathcal{B}} \quad \textrm{and} \quad \sup \mathcal{I}_{\mathcal{B}} < \sup \mathcal{I}_{\mathcal{W}}. $$
\end{claim}

\begin{proof} (cf. proof of \cite[Proposition 12]{szenes})
We need to show that the cases
$$ \inf \mathcal{I}_{\mathcal{B}} \leq \inf \mathcal{I}_{\mathcal{W}} \quad \textrm{and} \quad \sup \mathcal{I}_{\mathcal{B}} \leq \sup \mathcal{I}_{\mathcal{W}} $$
and
$$ \inf \mathcal{I}_{\mathcal{W}} \leq \inf \mathcal{I}_{\mathcal{B}} \quad \textrm{and} \quad \sup \mathcal{I}_{\mathcal{W}} \leq \sup \mathcal{I}_{\mathcal{B}} $$
are not possible. We give the proof for the first case only, the proof for the second case is similar. Assume that $ \inf \mathcal{I}_{\mathcal{B}} \leq \inf \mathcal{I}_{\mathcal{W}} $ and $ \sup \mathcal{I}_{\mathcal{B}} \leq \sup \mathcal{I}_{\mathcal{W}} $ . Denote
$$ v = \inf \mathcal{I}_{\mathcal{B}}, \quad w = \sup \mathcal{I}_{\mathcal{W}}. $$
There are $ a', b' $, endpoints of $ D $ with $ a < a' < b, a < b' < b $, such that
$$ v = \inf I_{\mu (a')}(a') = a' - \mu (a'), \quad \lambda (D | I_{\mu (a')}(a')) \geq 1 - \delta , $$
$$ w = \sup I_{\mu (b')}(b') = b' + \mu (b'), \quad \lambda (D | I_{\mu (b')}(b')) \leq \delta . $$
Let us realize that $ (v, v + \varepsilon ) \subset D $ for a small enough $ \varepsilon > 0 $. In the opposite case, we have $ (v, v + \varepsilon ) \cap D = \emptyset $ for a small enough $ \varepsilon > 0 $. But it is easy to check that, in such a case, $ \lambda (D | I_{\mu (a') - \varepsilon }(a')) > \lambda (D | I_{\mu (a')}(a')) $, which is not possible due to the choice of $ \mu (a') $. Similarly, it can be shown that $ (w - \varepsilon , w) \cap D = \emptyset $ for a small enough $ \varepsilon > 0 $. For this reason, $ v $ and $ w $ are not endpoints of the set $ C $ defined by
$$ C = (-\infty , v] \cup D \setminus [w, \infty ). $$
Thus, every endpoint of $ C $ is an endpoint of $ D $ at the same time.

We claim that every endpoint $ p $ of $ C $ has a radius $ \omega > 0 $ such that $ \lambda (C | I_{\omega }(p)) \notin (\delta , 1 - \delta ) $. If $ a < p < b $, then $ \omega = \mu (p) $ works, as $ I_{\mu (p)}(p) \subset (v, w) $ by our assumption (and thus $ C \cap I_{\mu (p)}(p) = D \cap I_{\mu (p)}(p) $). If $ p \leq a $, then $ p \leq a' $, and we can take $ \omega = \mu (a') $ (we have $ \lambda (C | I_{\mu (a')}(p)) \geq \lambda (C | I_{\mu (a')}(a')) $, as $ (p - \mu (a'), v) \subset C $). Similarly, if $ p \geq b $, then $ p \geq b' $, and we can take $ \omega = \mu (b') $ (we have $ \lambda (C | I_{\mu (b')}(p)) \leq \lambda (C | I_{\mu (b')}(b')) $, as $ (w, p + \mu (b')) \cap C = \emptyset $).

So, after an affine transformation, $ C $ will be a counterexample to $ \mathcal{K}(\delta ) $. This contradicts property (b) of $ D $, as $ C $ does not consist of more intervals than $ D $ (recall that $ (v, v + \varepsilon ) \subset D $ for a small enough $ \varepsilon > 0 $).
\end{proof}

\begin{proof}[Proof of Lemma \ref{lemmab}]
Let us show first that an endpoint $ p $ of $ D $ with $ a < p < b $ exists. Assume that $ D $ has no such an endpoint. Then $ (a, b) \subset D $ or $ (a, b) \subset (0, 1) \setminus D $. If $ (a, b) \subset D $, then, by (c),
$$ b - a = \lambda (a, b) \leq \lambda D \leq \frac{4\delta ^{2}}{1 - 2\delta } (b - a). $$
If $ (a, b) \subset (0, 1) \setminus D $, then, by (c),
$$ b - a = \lambda (a, b) \leq \lambda ((0, 1) \setminus D) \leq \frac{4\delta ^{2}}{1 - 2\delta } (b - a). $$
We obtain $ 1 \leq \frac{4\delta ^{2}}{1 - 2\delta } $, which means that $ \delta \geq \frac{1}{4}(\sqrt{5} - 1) $. This contradicts the assumption $ \delta < \frac{1}{4}(\sqrt{5} - 1) $.

So, at least one of the sets $ \mathcal{I}_{\mathcal{B}} $, $ \mathcal{I}_{\mathcal{W}} $ is not empty. Notice that, if we take
$$ D' = 1 - [(0, 1) \setminus (D \cup \partial D)] = (1 - r_{n},1 - s_{n-1}) \cup \dots \cup (1 - r_{2}, 1 - s_{1}) \cup (1 - r_{1}, 1), $$ 
$$ a' = 1 - b, \quad b' = 1 - a, \quad \mu '(1 - p) = \mu (p), $$
then we obtain objects with the same properties and $ \mathcal{I}_{\mathcal{B}}' = 1 - \mathcal{I}_{\mathcal{W}}, \mathcal{I}_{\mathcal{W}}' = 1 - \mathcal{I}_{\mathcal{B}} $. Therefore, without loss of generality,
\begin{itemize}
\item when one of the sets $ \mathcal{I}_{\mathcal{B}} $, $ \mathcal{I}_{\mathcal{W}} $ is empty, then we may assume that $ \mathcal{I}_{\mathcal{B}} \neq \emptyset = \mathcal{I}_{\mathcal{W}} $,
\item when both the sets $ \mathcal{I}_{\mathcal{B}} $, $ \mathcal{I}_{\mathcal{W}} $ are non-empty, then we may assume, due to Claim \ref{claimszenes}, that
$$ \inf \mathcal{I}_{\mathcal{B}} < \inf \mathcal{I}_{\mathcal{W}} \quad \textrm{and} \quad \sup \mathcal{I}_{\mathcal{W}} < \sup \mathcal{I}_{\mathcal{B}}. $$
\end{itemize}
In other words, for every endpoint $ p $ of $ D $ with $ a < p < b $,
$$ I_{\mu (p)}(p) \subset (v, w) $$
where
$$ v = \inf \mathcal{I}_{\mathcal{B}}, \quad w = \sup \mathcal{I}_{\mathcal{B}}. $$
We may also assume that
$$ D \subset (v, w) \cup (a, b). $$
Indeed, if we take $ D \cap ((v, w) \cup (a, b)) $ instead of $ D $, then no of the properties can be disrupted except the property $ \lambda ((0, 1) \setminus D) \leq \frac{4\delta ^{2}}{1 - 2\delta } (b - a) $ which interests us no more (by \textquotedblleft properties\textquotedblright {} we mean properties (a)--(c) from Lemma \ref{lemmaa}, endpoints $ p \in (a, b) $ and the relative measures $ \lambda (D | I_{\varepsilon }(p)) $ for $ 0 < \varepsilon \leq \mu (p) $).

There are $ a', b' $, endpoints of $ D $ with $ a < a' < b, a < b' < b $, such that
$$ v = \inf I_{\mu (a')}(a') = a' - \mu (a'), \quad \lambda (D | I_{\mu (a')}(a')) \geq 1 - \delta , $$
$$ w = \sup I_{\mu (b')}(b') = b' + \mu (b'), \quad \lambda (D | I_{\mu (b')}(b')) \geq 1 - \delta . $$
It is easy to verify that $ (v, v + \varepsilon ) \subset D $ and $ (w - \varepsilon , w) \subset D $ for a small enough $ \varepsilon > 0 $ (we delt with exactly the same thing in the proof of Claim \ref{claimszenes}). Even, we have $ (\min \{ v, a \}, v + \varepsilon ) \subset D $ and $ (w - \varepsilon , \max \{ w, b\} ) \subset D $. Indeed, to show that, e.g., $ (\min \{ v, a \}, v + \varepsilon ) \subset D $, it is sufficient to realize that $ D $ has no endpoint $ p $ such that $ a < p \leq v $ (such an endpoint would satisfy $ v \leq p - \mu (p) < p \leq v $). Therefore, if we express $ D $ in the form
$$ D = (r_{1}, s_{1}) \cup (r_{2}, s_{2}) \cup \dots \cup (r_{n}, s_{n}), $$
then
$$ r_{1} = \min \{ v, a \} , \quad v < s_{1}, $$
$$ s_{n} = \max \{ w, b \} , \quad r_{n} < w. $$

We define $ \varepsilon (a') = \mu (a'), \varepsilon (b') = \mu (b') $. The verification of the conditions (i)--(vii) is straightforward now.
\end{proof}

\begin{lemma} \label{lemmabwlog}
Let $ 0 < \delta < \frac{1}{4}(\sqrt{5} - 1) $. If $ \mathcal{K}(\delta ) $ does not hold, then there are numbers $ a, b $, intervals $ I_{\varepsilon (a')}(a'), I_{\varepsilon (b')}(b') $ and a set
$$ D = (r_{1}, s_{1}) \cup (r_{2}, s_{2}) \cup \dots \cup (r_{n}, s_{n}) $$
satisfying all the properties from Lemma \ref{lemmab} and, moreover,

{\rm (viii)} $ a' - a > \frac{4\delta - 1}{1 - 2\delta }\varepsilon (a') $,

{\rm (ix)} $ b - b' > \frac{4\delta - 1}{1 - 2\delta }\varepsilon (b') $.
\end{lemma}

\begin{proof}
1) Lemma \ref{lemmab} gives objects satisfying (i)--(vii), and we construct now objects satisfying (i)--(viii). The condition (viii) is automatically satisfied when $ \delta < 1/4 $, so suppose that $ \delta \geq 1/4 $. Let (i)--(vii) be satisfied for
$$ D_{0} = (r, s_{1}) \cup (r_{2}, s_{2}) \cup \dots \cup (r_{n}, s_{n}) $$
and $ a_{0}, b_{0}, I_{\varepsilon (a'_{0})}(a'_{0}), I_{\varepsilon (b'_{0})}(b'_{0}) $. We need to construct $ D, a, b, I_{\varepsilon (a')}(a'), I_{\varepsilon (b')}(b') $ so that (i)--(viii) are satisfied.

We may assume that (viii) is not satisfied, i.e., that
$$ a'_{0} - a_{0} \leq \frac{4\delta - 1}{1 - 2\delta } \varepsilon (a'_{0}). $$
We define
$$ D = \big( (r - (a'_{0} - r), r] \cup D_{0} \big) = (r - (a'_{0} - r), s_{1}) \cup (r_{2}, s_{2}) \cup \dots \cup (r_{n}, s_{n}), $$
$$ a = r - (a'_{0} - r), \quad b = b_{0}, \quad I_{\varepsilon (a')}(a') = I_{\varepsilon (a'_{0})}(a'_{0}), \quad I_{\varepsilon (b')}(b') = I_{\varepsilon (b'_{0})}(b'_{0}). $$
The conditions (i), (ii), (iv)--(vi) for $ D, a $ etc. are immediate consequences of the conditions (i), (ii), (iv)--(vi) for $ D_{0}, a_{0} $ etc. and to prove that (iii) holds, it remains to show that, for every endpoint $ p $ of $ D $ with $ a < p \leq a_{0} $, there is a radius $ \omega > 0 $ such that $ I_{\omega }(p) \subset (r - (a'_{0} - r), s_{n}) $ and $ \lambda (D | I_{\omega }(p)) \notin (\delta , 1 - \delta ) $. Let $ p $ be such an endpoint. Then $ p \geq r $, as $ D $ has no endpoints in $ (r - (a'_{0} - r), r) $. The choice $ \omega = \varepsilon (a'_{0}) $ works. Indeed, since $ a'_{0} - \varepsilon (a'_{0}) \geq r $, we have $ p - \varepsilon (a'_{0}) = p - a'_{0} + a'_{0} - \varepsilon (a'_{0}) \geq r - a'_{0} + r = r - (a'_{0} - r) $, and so $ I_{\varepsilon (a'_{0})}(p) \subset (r - (a'_{0} - r), s_{n}) $. Using (vi) for $ D_{0}, a_{0} $ etc., we obtain $ (p - \varepsilon (a'_{0}), a'_{0} - \varepsilon (a'_{0})) \subset (r - (a'_{0} - r), s_{1}) \subset D $, and so $ \lambda (D | I_{\varepsilon (a'_{0})}(p)) \geq \lambda (D | I_{\varepsilon (a'_{0})}(a'_{0})) \geq 1 - \delta $.

Since $ \delta < \frac{1}{4}(\sqrt{5} - 1) \leq 1/3 $, we obtain (viii) from
$$ a' - a \geq r - a = r - (r - (a'_{0} - r)) = a'_{0} - r \geq \varepsilon (a'_{0}) = \varepsilon (a') > \frac{4\delta - 1}{1 - 2\delta }\varepsilon (a'). $$
It remains to prove (vii). By our assumption, we have
$$ a'_{0} - a_{0} \leq \frac{4\delta - 1}{1 - 2\delta } \varepsilon (a'_{0}) \leq \frac{4\delta - 1}{1 - 2\delta } (a'_{0} - r), $$
and so, since (vii) is satisfied for $ D_{0}, a_{0} $ etc., we have, using $ 1/4 \leq \delta < \frac{1}{4}(\sqrt{5} - 1) \leq 1/\sqrt{8} $,
\begin{eqnarray*}
\lambda D & = & \lambda (r - (a'_{0} - r), r] + \lambda D_{0} \\
 & \leq & a'_{0} - r + \frac{4\delta ^{2}}{1 - 2\delta } (b_{0} - a_{0})
\end{eqnarray*}
%divided because of pagebreak
\begin{eqnarray*} 
 & = & a'_{0} - r + \frac{4\delta ^{2}}{1 - 2\delta } (a'_{0} - a_{0}) + \frac{4\delta ^{2}}{1 - 2\delta } (a - a'_{0}) + \frac{4\delta ^{2}}{1 - 2\delta } (b - a) \\
 & = & a'_{0} - r + \frac{4\delta ^{2}}{1 - 2\delta } (a'_{0} - a_{0}) - \frac{4\delta ^{2}}{1 - 2\delta } 2(a'_{0} - r) + \frac{4\delta ^{2}}{1 - 2\delta } (b - a) \\
 & = & - \frac{(4\delta - 1)(1 + 2\delta )}{1 - 2\delta } (a'_{0} - r) + \frac{4\delta ^{2}}{1 - 2\delta } (a'_{0} - a_{0}) + \frac{4\delta ^{2}}{1 - 2\delta } (b - a) \\
 & \leq & - (1 + 2\delta )(a'_{0} - a_{0}) + \frac{4\delta ^{2}}{1 - 2\delta } (a'_{0} - a_{0}) + \frac{4\delta ^{2}}{1 - 2\delta } (b - a) \\
 & = & - \frac{1 - 8\delta ^{2}}{1 - 2\delta } (a'_{0} - a_{0}) + \frac{4\delta ^{2}}{1 - 2\delta } (b - a) \\
 & \leq & \frac{4\delta ^{2}}{1 - 2\delta } (b - a).
\end{eqnarray*}

2) We have constructed objects satisfying (i)--(viii), and we construct finally objects satisfying (i)--(ix). The condition (ix) is automatically satisfied when $ \delta < 1/4 $, so suppose that $ \delta \geq 1/4 $. Let (i)--(viii) be satisfied for
$$ D_{0} = (r_{1}, s_{1}) \cup (r_{2}, s_{2}) \cup \dots \cup (r_{n-1}, s_{n-1}) \cup (r_{n}, s) $$
and $ a_{0}, b_{0}, I_{\varepsilon (a'_{0})}(a'_{0}), I_{\varepsilon (b'_{0})}(b'_{0}) $. We need to construct $ D, a, b, I_{\varepsilon (a')}(a'), I_{\varepsilon (b')}(b') $ so that (i)--(ix) are satisfied.

We may assume that (ix) is not satisfied, i.e., that
$$ b_{0} - b'_{0} \leq \frac{4\delta - 1}{1 - 2\delta } \varepsilon (b'_{0}). $$
We define
$$ D = \big( D_{0} \cup [s, s + (s - b'_{0})) \big) = (r_{1}, s_{1}) \cup (r_{2}, s_{2}) \cup \dots \cup (r_{n-1}, s_{n-1}) \cup (r_{n}, s + (s - b'_{0})), $$
$$ a = a_{0}, \quad b = s + (s - b'_{0}), \quad I_{\varepsilon (a')}(a') = I_{\varepsilon (a'_{0})}(a'_{0}), \quad I_{\varepsilon (b')}(b') = I_{\varepsilon (b'_{0})}(b'_{0}). $$
The condition (viii) is an immediate consequence of the condition (viii) for $ D_{0}, a_{0} $ etc. and the conditions (i)--(vii) and (ix) can be proved in the same way as the conditions (i)--(viii) in the previous part.
\end{proof}

The following simple lemma will be useful in many situations.

\begin{lemma} \label{lemmac}
Let $ A \subset \mathbb{R} $ be measurable and let $ I_{\gamma }(c) $ be an interval such that
$$ \lambda (A | I_{\gamma }(c)) \geq 1 - \delta , \quad 0 < \varepsilon < \gamma \Rightarrow \lambda (A | I_{\varepsilon }(c)) \leq \lambda (A | I_{\gamma }(c)). $$
(This is fulfilled in particular when $ \lambda (A | I_{\gamma }(c)) \unrhd 1 - \delta $).

Then
$$ c - \gamma \leq s < c + \gamma \quad \Rightarrow \quad \lambda (A | (s, c + \gamma )) \geq 1 - 2\delta , $$
$$ c - \gamma < t \leq c + \gamma \quad \Rightarrow \quad \lambda (A | (c - \gamma , t)) \geq 1 - 2\delta . $$
\end{lemma}

\begin{proof}
It is enough to prove the first implication only, the second one can be proved in the same way. Let $ c - \gamma \leq s < c + \gamma $. We consider two cases.

1) Let $ s > c $. Since $ \lambda (A | I_{s - c}(c)) \leq \lambda (A | I_{\gamma }(c)) $, we have $ \lambda (A | (I_{\gamma }(c) \setminus I_{s - c}(c))) \geq \lambda (A | I_{\gamma }(c)) \geq 1 - \delta $. Thus,
\begin{eqnarray*}
\lambda (A \cap (s, c + \gamma )) & = & \lambda (A \cap (I_{\gamma }(c) \setminus I_{s - c}(c))) - \lambda (A \cap (c - \gamma, c - (s - c))) \\
 & \geq & (1 - \delta ) \lambda (I_{\gamma }(c) \setminus I_{c - s}(c)) - \lambda (c - \gamma, c - (s - c)) \\
 & = & (1 - 2\delta ) \lambda (s, c + \gamma ).
\end{eqnarray*}

2) Let $ s \leq c $. We compute
\begin{eqnarray*}
\lambda (A \cap (s, c + \gamma )) & = & \lambda (A \cap I_{\gamma }(c)) - \lambda (A \cap (c - \gamma , s)) \\
 & \geq & (1 - \delta ) \lambda I_{\gamma }(c) - \lambda (c - \gamma , s) \\
 & = & (1 - 2\delta )\gamma + c - s \\
 & = & (1 - 2\delta ) \lambda (s, c + \gamma ) + 2\delta (c - s) \\
 & \geq & (1 - 2\delta ) \lambda (s, c + \gamma ).
\end{eqnarray*}
\end{proof}

\begin{tab} \label{roots}
In the table, we introduce constants $ \zeta _{1}, \zeta _{2}, \dots , \zeta _{7} $. We will refer to this table quite often in order to use inequalities which are implied by the assumption that $ \delta < \zeta _{i} $. The introduced inequalities should be easily verified, so we skip the verification of them.
\end{tab}

\medskip

\centerline{
\begin{tabular}{|c|c|c|c|}
\hline
    $ \zeta _{i} $ &      equals to &                          root of &    $ 0 < \delta < \zeta _{i} $ implies that \\
\hline
\hline
    $ \zeta _{1} $ &    0,268486... &      $ 8x^{3} + 8x^{2} + x - 1 $ &    $ \frac{4\delta ^{2}}{1 - 2\delta } < \frac{(1 - \delta)(1 + 2\delta)}{1 + 3\delta} $ \\
\hline
    $ \zeta _{2} $ &    0,268700... &         $ 8x^{3} - 16x^{2} + 1 $ &    $ \frac{2}{1 - 2\delta } - \frac{(1 - 2\delta )^{2}}{1 - 3\delta - 2\delta ^{2}} > 0 $ \\
\hline
    $ \zeta _{3} $ &    0,270690... &             $ 4x^{2} + 10x - 3 $ &    $ 1 < \frac{1 - 2\delta }{4\delta ^{2}} \big( 1 - \frac{4}{3}\delta \big) $ \\
                   &                &                                  &    $ (-4\delta ^{2} + 6\delta - 1)(1 - \delta ) \geq 8\delta ^{3} + 4\delta - 1 $ \\
\hline
    $ \zeta _{4} $ &    0,273301... &     $ 4x^{3} - 6x^{2} + 5x - 1 $ &    $ 4\delta ^{3} < (1 - 2\delta )(1 - 3\delta ) $ \\
\hline
    $ \zeta _{5} $ &    0,275255... &             $ 4x^{2} - 12x + 3 $ &    $ 4\delta ^{2} - 12\delta + 3 > 0 $ \\
\hline
    $ \zeta _{6} $ &    0,277479... &     $ 8x^{3} - 8x^{2} - 2x + 1 $ &    $ 2\delta \geq \frac{2\delta }{1 - 2\delta } - \frac{1}{2} \frac{1}{1 - \delta } $ \\
\hline
    $ \zeta _{7} $ &    0,280776... &              $ 2x^{2} + 3x - 1 $ &    $ 1 - 3\delta - 2\delta ^{2} > 0 $ \\
                   &                &                                  &    $ \delta < \frac{1 - 2\delta }{1 + 2\delta }, \quad \frac{(1 - \delta )(1 - 2\delta )}{2\delta } > 2\delta $ \\
                   &                &                                  &    $ -14\delta ^{2} + 15\delta - 3 \leq 2(4\delta - 1)(1 - 2\delta ) $ \\
\hline
\end{tabular}  
}

\bigskip

We prove now a more transparent analogue of Lemma \ref{lemmab}.

\begin{lemma} \label{lemmad}
Let $ 0 < \delta < \zeta _{3} $ where $ \zeta _{3} $ is as in Table \ref{roots}. If there is a counterexample to $ \mathcal{K}(\delta ) $, then there is a set
$$ F = (u_{1}, v_{1}) \cup (u_{2}, v_{2}) \cup \dots \cup (u_{m}, v_{m}), $$
where $ u_{1} < v_{1} < u_{2} < v_{2} < \dots < u_{m} < v_{m} $, such that the following conditions are satisfied:

{\rm (I)} $ \frac{1}{2}(u_{1} + v_{1}) = 0 $, $ \frac{1}{2}(u_{m} + v_{m}) = 1 $,

{\rm (II)} $ F $ consists of less intervals than any counterexample to $ \mathcal{K}(\delta ) $,

{\rm (III)} for every endpoint $ c $ of $ F $ with $ u_{1} < c < v_{m} $, there is a radius $ \omega > 0 $ such that $ I_{\omega }(c) \subset (u_{1}, v_{m}) $ and $ \lambda (F | I_{\omega }(c)) \notin (\delta , 1 - \delta ) $,

{\rm (IV)} $ \lambda (F \cap [0, 1]) \leq \frac{4\delta^{2}}{1 - 2\delta} $.
\end{lemma}

The proof of the lemma is given in the form of a series of claims. Let
$$ D = (r_{1}, s_{1}) \cup (r_{2}, s_{2}) \cup \dots \cup (r_{n}, s_{n}) $$
and $ a, b, I_{\varepsilon (a')}(a'), I_{\varepsilon (b')}(b') $ be as in Lemma \ref{lemmabwlog}. We define
$$ p = \min \big\{ a' + \varepsilon (a') \big\} \cup \big\{ x \in [a' - \varepsilon (a'), a' + \varepsilon (a')) : \lambda (D | (x, a' + \varepsilon (a'))) \leq 2\delta \big\} , $$
$$ p' = \max \big\{ b' - \varepsilon (b') \big\} \cup \big\{ y \in (b' - \varepsilon (b'), b' + \varepsilon (b')] : \lambda (D | (b' - \varepsilon (b'), y)) \leq 2\delta \big\} , $$
$$ q = r_{1} + \frac{p - r_{1}}{2\delta } \big( 1 - \lambda (D | (r_{1}, p)) \big) , \quad q' = s_{n} - \frac{s_{n} - p'}{2\delta } \big( 1 - \lambda (D | (p', s_{n})) \big) , $$
$$ E = (q, p) \cup (D \cap [p, p']) \cup (p', q'). $$
Note that $ q < p < p' < q' $ by Claim \ref{cl02}(2). Note also that the resulting set $ F $ will be obtained as an affine transformation of $ E $.

\begin{claim} \label{cl01}
{\rm (1)} For every $ t $ with $ r_{1} < t \leq a' + \varepsilon (a') $, we have $ \lambda (D | (r_{1}, t)) \geq 1 - 2\delta $. Similarly, for every $ s $ with $ b' - \varepsilon (b') \leq s < s_{n} $, we have $ \lambda (D | (s, s_{n})) \geq 1 - 2\delta $.

{\rm (2)} For every $ x $ with $ r_{1} \leq x < p $, we have $ \lambda (D | (x, p)) > 2\delta $. Similarly, for every $ y $ with $ p' < y \leq s_{n} $, we have $ \lambda (D | (p', y)) > 2\delta $.

{\rm (3)} We have $ \lambda (D | (r_{1}, a' + \varepsilon (a'))) \geq 1 - \delta $ and $ \lambda (D | (b' - \varepsilon (b'), s_{n})) \geq 1 - \delta $.

{\rm (4)} We have $ \lambda (D | (r_{1}, p)) \geq 1 - \delta $ and $ \lambda (D | (p', s_{n})) \geq 1 - \delta $.
\end{claim}

\begin{proof}
In each part, it is sufficient, due to the symmetry, to prove the first half of the statement only. We realize first that, since $ (r_{1}, a' - \varepsilon (a')) \subset D $ by (iv) and (vi), it is sufficient to prove that:

(1') For every $ t $ with $ a' - \varepsilon (a') < t \leq a' + \varepsilon (a') $, we have $ \lambda (D | (a' - \varepsilon (a'), t)) \geq 1 - 2\delta $.

(2') For every $ x $ with $ a' - \varepsilon (a') \leq x < p $, we have $ \lambda (D | (x, p)) > 2\delta $.

(3') We have $ \lambda (D | (a' - \varepsilon (a'), a' + \varepsilon (a'))) \geq 1 - \delta $.

(4') We have $ \lambda (D | (a' - \varepsilon (a'), p)) \geq 1 - \delta $.

The statement (1') follows from (v) and from Lemma \ref{lemmac}. Let us show (2'). Assume that $ a' - \varepsilon (a') \leq x < p $. We have $ p = a' + \varepsilon (a') $ or $ \lambda (D | (p, a' + \varepsilon (a'))) \leq 2\delta $ from the definition of $ p $. If $ \lambda (D | (x, p)) \leq 2\delta $, then $ \lambda (D | (x, a' + \varepsilon (a'))) \leq 2\delta $, and so $ p \leq x $ by the definition of $ p $, which is not possible. Therefore, $ \lambda (D | (x, p)) > 2\delta $. The statement (3') follows from (v). To prove (4'), we use again that $ p = a' + \varepsilon (a') $ or $ \lambda (D | (p, a' + \varepsilon (a'))) \leq 2\delta $. Since $ 2\delta < 1 - \delta $, we obtain (4') from (3').
\end{proof}

\begin{claim} \label{cl02}
{\rm (1)} We have $ a < p \leq a' + \varepsilon (a') < b' - \varepsilon (b') \leq p' < b $.

{\rm (2)} We have $ q < p < p' < q' $.
\end{claim}

\begin{proof}
(1) The second and the fourth inequality are clear. Using (v), we compute
\begin{eqnarray*}
2(1 - \delta )\varepsilon (a') & = & (1 - \delta )\lambda I_{\varepsilon (a')}(a') \\
 & \leq & \lambda (D \cap I_{\varepsilon (a')}(a')) \\
 & = & \lambda (D \cap (a' - \varepsilon (a'), p)) + \lambda (D \cap (p, a' + \varepsilon (a'))) \\
 & \leq & p - (a' - \varepsilon (a')) + 2\delta (a' + \varepsilon (a') - p),
\end{eqnarray*}
and, using (viii), we obtain
$$ p \geq a' - \frac{4\delta - 1}{1 - 2\delta }\varepsilon (a') > a. $$
Similarly, we can show, using (v) and (ix), that $ p' < b $.

It remains to show that $ a' + \varepsilon (a') < b' - \varepsilon (b') $. Assume the opposite and put
$$ \alpha = \lambda (r_{1}, b' - \varepsilon (b')), \quad \alpha _{D} = \lambda (D \cap (r_{1}, b' - \varepsilon (b'))), $$
$$ \beta = \lambda (b' - \varepsilon (b'), a' + \varepsilon (a')), \quad \beta _{D} = \lambda (D \cap (b' - \varepsilon (b'), a' + \varepsilon (a'))), $$
$$ \gamma = \lambda (a' + \varepsilon (a'), s_{n}), \quad \gamma _{D} = \lambda (D \cap (a' + \varepsilon (a'), s_{n})). $$
We obtain from the parts (1) and (3) of Claim \ref{cl01} that
$$ \alpha _{D} \geq (1 - 2\delta )\alpha , \quad \gamma _{D} \geq (1 - 2\delta )\gamma , $$
$$ \alpha _{D} + \beta _{D} \geq (1 - \delta )(\alpha + \beta ), \quad \beta _{D} + \gamma _{D} \geq (1 - \delta )(\beta + \gamma ). $$
We compute
$$
\begin{aligned}
3(\alpha _{D} + \beta _{D} & + \gamma _{D}) + \beta _{D} = 2(\alpha _{D} + \beta _{D}) + 2(\beta _{D} + \gamma _{D}) + \alpha _{D} + \gamma _{D} \\
 & \geq 2(1 - \delta )(\alpha + \beta ) + 2(1 - \delta )(\beta + \gamma ) + (1 - 2\delta )\alpha + (1 - 2\delta )\gamma \\
 & = (3 - 4\delta )(\alpha + \beta + \gamma ) + \beta \\
 & \geq (3 - 4\delta )(\alpha + \beta + \gamma ) + \beta _{D}
\end{aligned}
$$
and obtain
$$ \frac{\lambda D}{\lambda (r_{1}, s_{n})} \geq 1 - \frac{4}{3}\delta . $$
Hence, using (vii),
$$ 1 \geq \frac{b - a}{s_{n} - r_{1}} \geq \frac{1 - 2\delta }{4\delta ^{2}} \frac{\lambda D}{\lambda (r_{1}, s_{n})} \geq \frac{1 - 2\delta }{4\delta ^{2}} \Big( 1 - \frac{4}{3}\delta \Big) , $$
which is not possible due to the assumption $ \delta < \zeta _{3} $.

(2) We obtain from Claim \ref{cl01}(4) that $ \lambda (D | (r_{1}, p)) \geq 1 - \delta > 1 - 2\delta $. Thus,
$$ q = r_{1} + \frac{p - r_{1}}{2\delta } \big( 1 - \lambda (D | (r_{1}, p)) \big) < r_{1} + \frac{p - r_{1}}{2\delta } 2\delta = p. $$
Similarly, it can be shown that $ p' < q ' $. Finally, we know from the previous part that $ p < p' $.
\end{proof}

\begin{claim} \label{cl03}
For sufficiently small $ \varepsilon > 0 $ and $ \varepsilon ' > 0 $, we have $ (p - \varepsilon , p) \subset D $ and $ (p', p' + \varepsilon ') \subset D $.
\end{claim}

\begin{proof}
Due to the symmetry, it is sufficient to find an $ \varepsilon > 0 $ only. By Claim \ref{cl01}(2), we have $ \lambda (D | (x, p)) > 2\delta $ whenever $ r_{1} \leq x < p $. In particular, $ D \cap (x, p) \neq \emptyset $ whenever $ r_{1} \leq x < p $. Since $ D $ is a finite union of intervals, there is an $ \varepsilon > 0 $ such that $ (p - \varepsilon , p) \subset D $.
\end{proof}

\begin{claim} \label{cl04}
$ E $ consists of less intervals than any counterexample to $ \mathcal{K}(\delta ) $.
\end{claim}

\begin{proof}
It follows from the definition of $ E $ and from Claim \ref{cl03} that $ E $ does not consist of more intervals than $ D $. Now, it is sufficient to use (ii).
\end{proof}

\begin{claim} \label{cl05}
For every endpoint $ c $ of $ E $ with $ q < c < q' $, there is a radius $ \varrho > 0 $ such that $ I_{\varrho }(c) \subset (q, q') $ and $ \lambda (E | I_{\varrho }(c)) \notin (\delta , 1 - \delta ) $.
\end{claim}

\begin{proof}
Realize first that $ c $ is also an endpoint of $ D $. It is clear in the case that $ c \in (p, p') $. If $ c = p $ and $ c $ is not an endpoint of $ D $, then, by Claim \ref{cl03}, it lies in the interior of $ D $, and it also lies in the interior of $ E $, which is not possible. Similarly, the case that $ c = p' $ and $ c $ is not an endpoint of $ D $ is not possible. Now, as $ c $ is an endpoint of $ D $ and $ c \in [p, p'] \subset (a, b) $ by Claim \ref{cl02}(1), we obtain from (iii) a radius $ \omega > 0 $ such that $ I_{\omega }(c) \subset (r_{1}, s_{n}) $ and $ \lambda (D | I_{\omega }(c)) \notin (\delta , 1 - \delta ) $. We consider two possibilities.

(a) Let $ \lambda (D | I_{\omega }(c)) \leq \delta $. Let $ \varrho > 0 $ be a radius with $ I_{\varrho }(c) \subset (r_{1}, s_{n}) $ such that $ \lambda (D | I_{\varrho }(c)) $ is minimal (i.e., $ \lambda (D | I_{\varrho }(c)) \leq \lambda (D | I_{\varepsilon }(c)) $ whenever $ I_{\varepsilon }(c) \subset (r_{1}, s_{n}) $). Let us show that $ I_{\varrho }(c) \subset (p, p') $. Assume that, e.g., $ p \in I_{\varrho }(c) $. Then, by Claim \ref{cl01}(2), we have $ \lambda (D | (c - \varrho , p)) > 2\delta $. On the other hand, by Lemma \ref{lemmac} applied on $ \mathbb{R} \setminus D $, we have $ \lambda (D | (c - \varrho , p)) \leq 2\delta $, which is a contradiction. So, $ I_{\varrho }(c) \subset (p, p') $ indeed, and we obtain $ \lambda (E | I_{\varrho }(c)) = \lambda (D | I_{\varrho }(c)) \leq \lambda (D | I_{\omega }(c)) \leq \delta $.

(b) Let $ \lambda (D | I_{\omega }(c)) \geq 1 - \delta $. We will assume that the distance of $ c $ and $ q $ is not greater than the distance of $ c $ and $ q' $ (the other case is symmetric). We may suppose that $ c - \omega \leq q $ because $ \lambda (E | I_{\omega }(c)) \geq \lambda (D | I_{\omega }(c)) \geq 1 - \delta $ in the case that $ c - \omega > q $. Let us show that
$$ \varrho = c - q $$
works. Note that $ 0 < \varrho \leq \omega $ and $ c + \varrho \leq q' $. Using Claim \ref{cl01}(1), we compute
$$
\begin{aligned}
 & \hspace{-0.4cm} \lambda (E \cap I_{\varrho}(c)) - (1 - \delta ) \lambda I_{\varrho}(c) \\
 = & \lambda (q, p) + \lambda (E \cap (p, c + \varrho)) - (1 - \delta ) 2\varrho \\
 \geq & \lambda (q, p) + \lambda (D \cap (p, c + \varrho)) - (1 - \delta ) 2\varrho \\
 = & p - q - (1 - \delta ) 2\varrho + \lambda (D \cap I_{\omega }(c)) - \lambda (D \cap (c - \omega , p)) - \lambda (D \cap (c + \varrho, c + \omega )) \\
 \geq & p - q - (1 - \delta ) 2\varrho + (1 - \delta ) 2\omega - \lambda (D \cap (r_{1}, p)) + \lambda (D \cap (r_{1}, c - \omega )) - (\omega - \varrho) \\
 \geq & p - q - (1 - \delta ) 2\varrho + (1 - \delta ) 2\omega - \lambda (D \cap (r_{1}, p)) + (1 - 2\delta )(c - \omega - r_{1}) - (\omega - \varrho) \\
 = & p - q - (1 - 2\delta )\varrho + (1 - 2\delta )\omega - \lambda (D \cap (r_{1}, p)) + (1 - 2\delta )(c - \omega - r_{1}) \\
 = & p - q - (1 - 2\delta )(c - q) + (1 - 2\delta )\omega - \lambda (D \cap (r_{1}, p)) + (1 - 2\delta )(c - \omega - r_{1}) \\
 = & p - 2\delta q - (1 - 2\delta )r_{1} - \lambda (D \cap (r_{1}, p)) \\
 = & 0.
\end{aligned}
$$
\end{proof}

\begin{claim} \label{cl06}
We have $ \lambda (E | [\frac{1}{2}(q + p), \frac{1}{2}(p' + q')]) \leq \frac{4\delta^{2}}{1 - 2\delta} $.
\end{claim}

\begin{proof}
We put
$$ t = \frac{1}{2}\big( q + p \big) , \quad t' = \frac{1}{2}\big( p' + q' \big) , $$
$$ W = \lambda (r_{1}, p) + \lambda (p', s_{n}), \quad W_{D} = \lambda (D \cap (r_{1}, p)) + \lambda (D \cap (p', s_{n})). $$
Compute
\begin{eqnarray*}
2(t' - t) - 2(s_{n} - r_{1}) & = & p' + q' - q - p - 2(s_{n} - r_{1}) \\
 & = & - (s_{n} - p') - (s_{n} - q') - (q - r_{1}) - (p - r_{1}) \\
 & = & - (s_{n} - p') - \frac{s_{n} - p'}{2\delta } \big( 1 - \lambda (D | (p', s_{n})) \big) \\
 & & - (p - r_{1}) - \frac{p - r_{1}}{2\delta } \big( 1 - \lambda (D | (r_{1}, p)) \big) \\
 & = & - \Big( 1 + \frac{1}{2\delta }\Big) W + \frac{1}{2\delta }W_{D}
\end{eqnarray*}
and
\begin{eqnarray*}
\lambda (E \cap [t, t']) - \lambda D & = & \lambda [t, p) + \lambda (D \cap [p, p']) + \lambda (p', t'] - \lambda D \\
 & = & p - t + t' - p' - \lambda (D \cap (r_{1}, p)) - \lambda (D \cap (p', s_{n})) \\
 & = & W - (s_{n} - r_{1}) + t' - t - W_{D} \\
 & = & W + \frac{1}{2} \Big[ - \Big( 1 + \frac{1}{2\delta }\Big) W + \frac{1}{2\delta }W_{D} \Big] - W_{D}.
\end{eqnarray*}
We have, using (vii),
$$ \lambda D \leq \frac{4\delta ^{2}}{1 - 2\delta } (b - a) \leq \frac{4\delta ^{2}}{1 - 2\delta } (s_{n} - r_{1}), $$
so, to prove the claim, it is sufficient to prove that
$$ \lambda (E \cap [t, t']) - \lambda D \leq \frac{4\delta ^{2}}{1 - 2\delta } (t' - t) - \frac{4\delta ^{2}}{1 - 2\delta } (s_{n} - r_{1}), $$
which can be rewritten as
$$ W + \frac{1}{2} \Big[ - \Big( 1 + \frac{1}{2\delta }\Big) W + \frac{1}{2\delta }W_{D} \Big] - W_{D} \leq \frac{4\delta ^{2}}{1 - 2\delta } \frac{1}{2} \Big[ - \Big( 1 + \frac{1}{2\delta }\Big) W + \frac{1}{2\delta }W_{D} \Big] . $$
A calculation shows that the desired inequality is equivalent to
$$ (-4\delta ^{2} + 6\delta - 1) W_{D} \geq (8\delta ^{3} + 4\delta - 1) W. $$
Note that, since we assume that $ \mathcal{K}(\delta ) $ holds, we have $ \delta \geq \delta _{\mathcal{K}} \geq 1/4 $ (see \cite{szenes}). It follows that $ -4\delta ^{2} + 6\delta - 1 \geq 0 $. By Claim \ref{cl01}(4),
$$ W_{D} \geq (1 - \delta )W, $$
and it is sufficient to compute, using the assumption $ \delta < \zeta _{3} $,
$$ (-4\delta ^{2} + 6\delta - 1) W_{D} \geq (-4\delta ^{2} + 6\delta - 1)(1 - \delta ) W \geq (8\delta ^{3} + 4\delta - 1) W. $$
\end{proof}

Now, let $ o $ be the center of the lowest and $ o' $ be the center of the highest connected component of $ E $. Let $ \varphi $ be the affine transformation which maps $ o $ onto $ 0 $ and $ o' $ onto $ 1 $. We set
$$ F = \varphi (E). $$
The condition (I) is clear and the conditions (II)--(IV) follow from Claims \ref{cl04}--\ref{cl06} (note that $ \lambda (E | [o, o']) \leq \lambda (E | [\frac{1}{2}(q + p), \frac{1}{2}(p' + q')]) $ since $ (\frac{1}{2}(q + p), o) \subset E $ and $ (o', \frac{1}{2}(p' + q')) \subset E $). This completes the proof of Lemma \ref{lemmad}.

\section{The concept of a $ \delta $-good set}

In the previous section, we worked with a configuration which was a counterexample to $ \mathcal{K}(\delta ) $ with the least possible number of intervals in it. We constructed new objects from this configuration (Lemma \ref{lemmab} and Lemma \ref{lemmad}). One more step remains to construct an object which will play a key role in the proof of our lower bound on $ \delta _{\mathcal{H}} $.

Let $ 0 < \delta < 1/2 $. Let $ G $ be a set given by
$$ G = [0, \nu _{1}) \cup (\mu _{2}, \nu _{2}) \cup \dots \cup (\mu _{r-1}, \nu _{r-1}) \cup (\mu _{r}, 1], $$
where $ 0 < \nu _{1} < \mu _{2} < \nu _{2} < \dots < \mu _{r-1} < \nu _{r-1} < \mu _{r} < 1 $. We say that $ G $ is \emph{$ \delta $-good} if the following conditions are satisfied for $ H $ defined as $ G + \mathbb{Z} $:
\begin{itemize}
\item[{(i)}] for every endpoint $ p $ of $ H $, there is a radius $ \omega > 0 $ such that $ \lambda (H | I_{\omega }(p)) \notin (\delta , 1 - \delta ) $,
\item[{(ii)}] if $ a < b $, then each of the sets
$$ ((-\infty , a) \cup H) \setminus [b, \infty ) , \quad (H \cup (b, \infty )) \setminus (-\infty , a], $$
$$ H \setminus [b, \infty ) , \quad H \setminus (-\infty , a], $$
denoted by $ C $, has an endpoint $ p $ such that $ \lambda (C | I_{\omega }(p)) \in (\delta , 1 - \delta ) $ for every $ \omega > 0 $.
\end{itemize}

\begin{lemma} \label{goodset}
Let $ 0 < \delta < \zeta _{3} $ where $ \zeta _{3} $ is as in Table \ref{roots}. If there is a counterexample to $ \mathcal{K}(\delta ) $, then there is a $ \delta $-good set
$$ G = [0, \nu _{1}) \cup (\mu _{2}, \nu _{2}) \cup \dots \cup (\mu _{r-1}, \nu _{r-1}) \cup (\mu _{r}, 1] $$
such that
$$ \lambda G \leq \frac{4\delta^{2}}{1 - 2\delta}. $$
\end{lemma}

We prove a claim first.

\begin{claim} \label{goodsetclaim}
Let $ F $ be as in Lemma \ref{lemmad}. Then, for every $ s \in (u_{1}, v_{m}) $, we have
$$ \lambda (F | (u_{1}, s)) > 2\delta \quad \textrm{and} \quad \lambda (F | (s, v_{m})) > 2\delta . $$
\end{claim}

\begin{proof}
Due to the symmetry, it is enough to prove the first inequality only. Assume that $ s \in (u_{1}, v_{m}) $ and $ \lambda (F | (u_{1}, s)) \leq 2\delta $. Then
$$ C = F \cup [v_{m}, \infty ) = (u_{1}, v_{1}) \cup (u_{2}, v_{2}) \cup \dots \cup (u_{m-1}, v_{m-1}) \cup (u_{m}, \infty ) $$
will be, after an affine transformation, a counterexample to $ \mathcal{K}(\delta ) $. Indeed, there is an appropriate radius for every endpoint of $ C $ which is greater than $ u_{1} $ (we can take the radius (III) gives). For $ u_{1} $, we can take the radius $ \omega = s - u_{1} $ because $ \lambda (C | I_{\omega }(u_{1})) = \lambda (C | (u_{1}, s))/2 = \lambda (F | (u_{1}, s))/2 \leq \delta $. This is a contradiction with (II), as $ C $ does not consist of more intervals than $ F $.
\end{proof}

\begin{proof}[Proof of Lemma \ref{goodset}]
Let
$$ F = (u_{1}, v_{1}) \cup (u_{2}, v_{2}) \cup \dots \cup (u_{m}, v_{m}) $$
be the set which Lemma \ref{lemmad} gives for $ \delta $. We show that the choice
$$ G = (1/2)(F \cap [0, 1]) \cup \big( 1 - (1/2)(F \cap [0, 1]) \big) $$
works. It follows from (IV) that
$$ \lambda G = \lambda (F \cap [0, 1]) \leq \frac{4\delta ^{2}}{1 - 2\delta }. $$
We define $ H = G + \mathbb{Z} $. Note that we obtain from (I) the following geometric interpretation: $ H $ consists of affine transformations of $ F $ such that the last interval of the transformation $ (1/2)F + z $ coincides with the first interval of the transformation $ 1 - (1/2)F + z $ and the last interval of the transformation $ 1 - (1/2)F + z $ coincides with the first interval of the transformation $ (1/2)F + (z + 1) $.

The condition (i) on a $ \delta $-good set follows from (III) and from the mentioned geometric interpretation. To prove (ii), we need an observation first. We define
$$ L = \{ (1/2)v_{m}, 1 - (1/2)u_{1} \} + \mathbb{Z}. $$
Let us show that
$$ \lambda (H | (s, l)) > 2\delta , \quad l \in L, s < l. $$
If $ l' \leq s < l $ where $ l' $ is the predecessor of $ l $ in $ L $, then we have $ \lambda (H | (s, l)) > 2\delta $ from Claim \ref{goodsetclaim}. In particular, $ \lambda (H | (l', l)) > 2\delta $. If $ l \in L $ and $ s < l $ are general, then we write first $ l_{k+1} \leq s < l_{k} < \dots < l_{1} = l $ where $ l_{i+1} $ is the predecessor of $ l_{i} $ in $ L $ for $ i = 1, 2, \dots , k $. Then we obtain $ \lambda (H | (s, l)) > 2\delta $ from $ \lambda (H | (s, l_{k})) > 2\delta $ and from $ \lambda (H | (l_{i+1}, l_{i})) > 2\delta , i = 1, 2, \dots , k - 1 $.

Let us prove (ii) now. It is sufficient, due to $ H = -H $, to consider the cases $ ((-\infty , a) \cup H) \setminus [b, \infty ) $ and $ H \setminus [b, \infty ) $ only. Actually, this will be one case for us since we will allow $ a = -\infty $. So, let $ -\infty \leq a < b < \infty $ and let
$$ C = ((-\infty , a) \cup H) \setminus [b, \infty ). $$
What we need is to show that there is an endpoint $ p $ of $ C $ such that $ \lambda (C | I_{\omega }(p)) \in (\delta , 1 - \delta ) $ for every $ \omega > 0 $. Suppose the opposite, i.e., that every endpoint $ p $ of $ C $ has a radius $ \omega (p) > 0 $ such that $ \lambda (C | I_{\omega (p)}(p)) \notin (\delta , 1 - \delta ) $. We put
$$ a' = \max L \cap (-\infty , b). $$
We have $ a < a' $. Indeed, if $ a \geq a' $, then $ C $ will be, after an affine transformation (and adding an isolated point of the complement if necessary), a counterexample to $ \mathcal{K}(\delta ) $ which does not consist of more intervals than $ F $, which contradicts (II). We are going to show that every endpoint $ p $ of
$$ C' = ((-\infty , a') \cup H) \setminus [b, \infty ) $$
has a radius $ \varrho (p) > 0 $ such that $ \lambda (C' | I_{\varrho (p)}(p)) \notin (\delta , 1 - \delta ) $. Firstly, every endpoint $ p $ of $ C' $ is also an endpoint of $ C $, and thus $ p $ has an appropriate $ \omega (p) $. Secondly, $ C \subset C' $, so we can choose $ \varrho (p) = \omega (p) $ if $ \lambda (C | I_{\omega (p)}(p)) \geq 1 - \delta $. It remains to find an appropriate $ \varrho (p) $ for endpoints with $ \lambda (C | I_{\omega (p)}(p)) \leq \delta $. We choose $ \varrho (p) $ to be a radius for which $ \lambda (C | I_{\varrho (p)}(p)) $ is minimal. Let us show that $ I_{\varrho (p)}(p) \subset (a', \infty ) $. Assume that $ a' \in I_{\varrho (p)}(p) $. Then, as $ a' \in L $, we have $ \lambda (C | (p - \varrho (p), a')) \geq \lambda (H | (p - \varrho (p), a')) > 2\delta $. On the other hand, by Lemma \ref{lemmac} applied on $ \mathbb{R} \setminus C $, we have $ \lambda (C | (p - \varrho (p), a')) \leq 2\delta $, which is not possible. Hence $ I_{\varrho (p)}(p) \subset (a', \infty ) $ indeed, and we obtain $ \lambda (C' | I_{\varrho (p)}(p)) = \lambda (C | I_{\varrho (p)}(p)) \leq \lambda (C | I_{\omega (p)}(p)) \leq \delta $.

So, after an affine transformation, $ C' $ will be a counterexample to $ \mathcal{K}(\delta ) $ which does not consist of more intervals than $ F $. Again, this contradicts (II), and (ii) is proved.
\end{proof}

Now, we prove some useful properties of a $ \delta $-good set.

\begin{lemma} \label{lemmae}
Let
$$ G = [0, \nu _{1}) \cup (\mu _{2}, \nu _{2}) \cup \dots \cup (\mu _{r-1}, \nu _{r-1}) \cup (\mu _{r}, 1] $$
be a $ \delta $-good set and let $ H $ denote $ G + \mathbb{Z} $.

{\rm (1)} If $ s, t \in \mathbb{R} $ are points which do not belong to the interior of $ H $, then there are intervals $ I_{\alpha }(a) $ and $ I_{\beta }(b) $ such that
$$ a \leq s, \quad s \in I_{\alpha }(a), \quad b \geq t, \quad t \in I_{\beta }(b), $$
$$ \lambda (H | I_{\alpha }(a)) \unrhd 1 - \delta , \quad \lambda (H | I_{\beta }(b)) \unrhd 1 - \delta . $$

{\rm (2)} If $ p \leq q $ are two endpoints of $ H $, then there are intervals $ I_{\alpha }(a) $ and $ I_{\beta }(b) $ such that
$$ p \leq a \leq q, \quad \quad 0 < \varepsilon < \alpha \Rightarrow \lambda (H | I_{\varepsilon }(a)) \in (\delta , 1 - \delta ), $$
$$ p \leq b \leq q, \quad \quad 0 < \varepsilon < \beta \Rightarrow \lambda (H | I_{\varepsilon }(b)) \in (\delta , 1 - \delta ), $$
$$ [\lambda (H | I_{\alpha }(a)) \geq 1 - \delta \; \textrm{and} \; p \in I_{\alpha }(a)] \quad \textrm{or} \quad [\lambda (H | I_{\alpha }(a)) \leq \delta \; \textrm{and} \; q \in I_{\alpha }(a)], $$
$$ [\lambda (H | I_{\beta }(b)) \geq 1 - \delta \; \textrm{and} \; q \in I_{\beta }(b)] \quad \textrm{or} \quad [\lambda (H | I_{\beta }(b)) \leq \delta \; \textrm{and} \; p \in I_{\beta }(b)]. $$

{\rm (3)} $ \mathbb{R} $ is covered by a locally finite system of intervals $ I $ with $ \lambda (H | I) \geq 1 - \delta $.
\end{lemma}

\begin{proof}
(1) We prove the existence of an $ I_{\alpha }(a) $ only because the existence of an $ I_{\beta }(b) $ can be proved in the same way. Consider the set
$$ C = H \setminus [s, \infty ). $$
By (ii), there is an endpoint $ a $ of $ C $ such that, for every $ \omega > 0 $, we have $ \lambda (C | I_{\omega }(a)) \in (\delta , 1 - \delta ) $. Note that $ a $ is also an endpoint of $ H $, as $ s $ does not belong to the interior of $ H $. On the other hand, by (i), there is an $ \alpha > 0 $ such that $ \lambda (H | I_{\alpha }(a)) \notin (\delta , 1 - \delta ) $. Since $ \lambda (H | I_{\alpha }(a)) \geq \lambda (C | I_{\alpha }(a)) > \delta $, we have $ \lambda (H | I_{\alpha }(a)) \geq 1 - \delta $. Let us take the minimal $ \alpha > 0 $ with this property that $ \lambda (H | I_{\alpha }(a)) \geq 1 - \delta $, so we have automatically
$$ \lambda (H | I_{\alpha }(a)) \unrhd 1 - \delta . $$
Such a minimal $ \alpha > 0 $ exists, as $ a $ is an endpoint of $ H $. Further, we have $ s \in I_{\alpha }(a) $. Indeed, if $ s \notin I_{\alpha }(a) $, then $ \lambda (H | I_{\alpha }(a)) = \lambda (C | I_{\alpha }(a)) < 1 - \delta $, which is not possible.

(2) We prove the existence of an $ I_{\alpha }(a) $ only because the existence of an $ I_{\beta }(b) $ can be proved in the same way. If $ p = q $, then we put $ a = p = q $. By (i), there is a radius $ \alpha > 0 $ such that $ \lambda (H | I_{\alpha }(a)) \notin (\delta , 1 - \delta ) $. It is sufficient to take the minimal such an $ \alpha > 0 $. Such a minimal $ \alpha $ exists, as $ a $ is an endpoint of $ H $.

So, we may assume that $ p < q $. Consider the set
$$ C = (H \cup (q, \infty )) \setminus (-\infty , p]. $$
By (ii), there is an endpoint $ a $ of $ C $ such that $ \lambda (C | I_{\omega }(a)) \in (\delta , 1 - \delta ) $ for every $ \omega > 0 $. Necessarily $ p \leq a \leq q $. Moreover, $ a $ is also an endpoint of $ H $ ($ H $ and $ C $ have the same endpoints in $ (p, q) $ and $ p, q $ are endpoints of $ H $). By (i), there is, on the other hand, a radius $ \alpha > 0 $ such that $ \lambda (H | I_{\alpha }(a)) \notin (\delta , 1 - \delta ) $. Let $ \alpha $ be the minimal such a radius, so we have automatically
$$ 0 < \varepsilon < \alpha \Rightarrow \lambda (H | I_{\varepsilon }(a)) \in (\delta , 1 - \delta ). $$
Such a minimal $ \alpha > 0 $ exists, as $ a $ is an endpoint of $ H $.

We consider two cases, $ \lambda (H | I_{\alpha }(a)) \geq 1 - \delta $ and $ \lambda (H | I_{\alpha }(a)) \leq \delta $.

\begin{itemize}
\item Let $ \lambda (H | I_{\alpha }(a)) \geq 1 - \delta $. If $ I_{\alpha }(a) \subset (p, \infty ) $, then $ 1 - \delta \leq \lambda (H | I_{\alpha }(a)) = \lambda ((H \setminus (-\infty , p]) | I_{\alpha }(a)) \leq \lambda (C | I_{\alpha }(a)) < 1 - \delta $, which is not possible. Hence, $ p \in I_{\alpha }(a) $.
\item Let $ \lambda (H | I_{\alpha }(a)) \leq \delta $. If $ I_{\alpha }(a) \subset (-\infty , q) $, then $ \delta \geq \lambda (H | I_{\alpha }(a)) = \lambda ((H \cup (q, \infty )) | I_{\alpha }(a)) \geq \lambda (C | I_{\alpha }(a)) > \delta $, which is not possible. Hence, $ q \in I_{\alpha }(a) $.
\end{itemize}

(3) We show that every interval $ (x, s) $ with $ (x, s) \subset \mathbb{R} \setminus H $ is covered by an interval $ I_{\alpha }(a) $ with $ \lambda (H | I_{\alpha }(a)) \geq 1 - \delta $. The part (1) gives an interval $ I_{\alpha }(a) $ such that $ a \leq s, s \in I_{\alpha }(a) $ and $ \lambda (H | I_{\alpha }(a)) \unrhd 1 - \delta $. It remains to show that $ x \in I_{\alpha }(a) $. If $ x \notin I_{\alpha }(a) $, then $ (a - \alpha , s) \subset \mathbb{R} \setminus H $, and so $ \lambda (H | (a - \alpha , s)) = 0 $. But this is not possible because, by Lemma \ref{lemmac}, we have $ \lambda (H | (a - \alpha , s)) \geq 1 - 2\delta $.
\end{proof}

\section{A Property of $ \delta $-good sets}

We introduced in Lemma \ref{lemmae} some properties of $ \delta $-good sets. These properties were obtained quite easily and naturally from the definition of a $ \delta $-good set. On the other hand, the proof of the following property is not easy. Although the property appears naturally in Lemma \ref{lemmaxy}, its proof is very technical. As the proof proceeds, we deal with more and more curious cases which demand more and more computations.

\begin{lemma} \label{keyprop}
Let $ 0 < \delta < \zeta _{2} $ where $ \zeta _{2} $ is as in Table \ref{roots} and let
$$ G = [0, \nu _{1}) \cup (\mu _{2}, \nu _{2}) \cup \dots \cup (\mu _{r-1}, \nu _{r-1}) \cup (\mu _{r}, 1] $$
be a $ \delta $-good set. Let
$$ H = G + \mathbb{Z}. $$
If $ w < v $ and $ \lambda (H | (w, v)) \leq (1 - \delta )/2 $, then there is an interval $ I_{\alpha }(a) \supset (w, v) $ such that
$$ \lambda (H | I_{\alpha }(a)) \unrhd 1 - \delta $$
and, for every $ u \in [w, v) $, there is an interval $ I_{\beta }(b) \supset (u, v) $ such that
$$ \lambda (H | I_{\beta }(b)) \unrhd 1 - \delta $$
and
$$ b - a \geq v - u - \frac{2}{1 - \delta } \lambda (H \cap (u, v)). $$
\end{lemma}

The proof of the lemma is provided in two steps. In the first part of the proof, we prove the lemma for the intervals with an additional property (Claim \ref{cl5}). In the second part, we are dealing with an interval which does not have this property. This part is finished by Claim \ref{cl12}. 

For a measurable $ A \subset \mathbb{R} $ and a couple $ p < q $, consider the function
$$ f_{A, p, q}(x) = (1 - 2\delta )x - \lambda (A \cap (p, x)), \quad x \in [p, q], $$
and let $ s(A, p, q) $ be the greatest point of minimum and $ t(A, p, q) $ be the least point of maximum of $ f_{A, p, q} $, so $ s(A, p, q) $ and $ t(A, p, q) $ have the property that
$$ p \leq x < s(A, p, q) \quad \Rightarrow \quad \lambda (A | (x, s(A, p, q))) \geq 1 - 2\delta , $$
$$ s(A, p, q) < x \leq q \quad \Rightarrow \quad \lambda (A | (s(A, p, q), x)) < 1 - 2\delta , $$
$$ p \leq x < t(A, p, q) \quad \Rightarrow \quad \lambda (A | (x, t(A, p, q))) < 1 - 2\delta , $$
$$ t(A, p, q) < x \leq q \quad \Rightarrow \quad \lambda (A | (t(A, p, q), x)) \geq 1 - 2\delta . $$

\begin{claim} \label{cl1}
Let $ A \subset \mathbb{R} $ be measurable and let $ I_{\gamma }(c) $ be an interval such that
$$ \lambda (A | I_{\gamma }(c)) \unrhd 1 - \delta . $$
For a couple $ p < q $, we have:
\begin{itemize}
\item If $ s(A, p, q) \in I_{\gamma }(c) $, then $ q \in I_{\gamma }(c) $.
\item If $ t(A, p, q) \in I_{\gamma }(c) $, then $ p \in I_{\gamma }(c) $.
\end{itemize}
\end{claim}

\begin{proof}
We prove the first assertion only because the second one can be proved in the same way. Assume that $ c - \gamma < s(A, p, q) < c + \gamma \leq q $. We want to show that this situation is impossible. By a property of $ s(A, p, q) $, we have $ \lambda (A | (s(A, p, q), c + \gamma )) < 1 - 2\delta $. On the other hand, by Lemma \ref{lemmac}, we have $ \lambda (A | (s(A, p, q), c + \gamma )) \geq 1 - 2\delta $.
\end{proof}

\begin{claim} \label{cl2}
Let $ w < v $ be such that $ \lambda (H | (w, v)) < 1 - 2\delta $. Then there are intervals $ I_{\alpha }(a) $ and $ I_{\beta }(b) $ such that
$$ a \leq s(H, w, v), \quad v \in I_{\alpha }(a), \quad b \geq t(H, w, v), \quad w \in I_{\beta }(b), $$
$$ \lambda (H | I_{\alpha }(a)) \unrhd 1 - \delta , \quad \lambda (H | I_{\beta }(b)) \unrhd 1 - \delta . $$
\end{claim}

\begin{proof}
We prove the existence of an $ I_{\alpha }(a) $ only because the existence of an $ I_{\beta }(b) $ can be proved in the same way. We have $ s(H, w, v) < v $. Indeed, if $ s(H, w, v) = v $, then $ 1 - 2\delta > \lambda (H | (w, v)) = \lambda (H | (w, s(H, w, v))) \geq 1 - 2\delta $, which is not possible. For every $ x $ with $ s(H, w, v) < x \leq v $, we have $ \lambda (H | (s(H, w, v), x)) < 1 - 2\delta $, and so $ (s(H, w, v), x) \not\subset H $. Consequently, $ s(H, w, v) $ does not belong to the interior of $ H $.

By Lemma \ref{lemmae}(1), there is an interval $ I_{\alpha }(a) $ such that $ a \leq s(H, w, v) $, $ s(H, w, v) \in I_{\alpha }(a) $ and $ \lambda (H | I_{\alpha }(a)) \unrhd 1 - \delta $. Since $ s(H, w, v) \in I_{\alpha }(a) $, we have $ v \in I_{\alpha }(a) $ by Claim \ref{cl1}.
\end{proof}

\begin{claim} \label{cl3}
Let $ p \leq t < s \leq q $ be such that $ \lambda (H | (p, s)) \geq 1 - 2\delta $ and $ \lambda (H | (t, q)) \geq 1 - 2\delta $.

{\rm (1)} We have $ \lambda (H | (p, q)) \geq (1 - 2\delta )/(1 + 2\delta ) $.

{\rm (2)} If $ \lambda (H | (p, q)) \leq (1 - \delta )/2 $, then $ \lambda (H | (t, s)) \geq (1 - \delta )(1 - 2\delta )/2\delta $.
\end{claim}

\begin{proof}
Let us put
$$ k = \lambda (p, t), \quad k^{H} = \lambda (H \cap (p, t)), $$
$$ l = \lambda (t, s), \quad l^{H} = \lambda (H \cap (t, s)), $$
$$ m = \lambda (s, q), \quad m^{H} = \lambda (H \cap (s, q)), $$
$$ L = \lambda (p, q), \quad L^{H} = \lambda (H \cap (p, q)). $$
We obtain inequalities
$$ k^{H} + l^{H} \geq (1 - 2\delta )(k + l), \quad l^{H} + m^{H} \geq (1 - 2\delta )(l + m), $$
and compute
$$ L^{H} + l^{H} = k^{H} + 2l^{H} + m^{H} \geq (1 - 2\delta )(k + 2l + m) = (1 - 2\delta )L + (1 - 2\delta )l, $$
$$ 2L^{H} \geq L^{H} + l^{H} \geq (1 - 2\delta )L + (1 - 2\delta )l, $$
$$ L^{H} + l \geq L^{H} + l^{H} \geq (1 - 2\delta )L + (1 - 2\delta )l. $$
So,
$$ l^{H} \geq (1 - 2\delta )L + (1 - 2\delta )l - L^{H}, $$
$$ l \leq \frac{2}{1 - 2\delta }L^{H} - L, $$
$$ l \geq \frac{1 - 2\delta }{2\delta }L - \frac{1}{2\delta }L^{H}. $$

(1) From
$$ \frac{1 - 2\delta }{2\delta }L - \frac{1}{2\delta }L^{H} \leq l \leq \frac{2}{1 - 2\delta }L^{H} - L, $$
we obtain by a straightforward computation that
$$ L^{H} \geq \frac{1 - 2\delta }{1 + 2\delta }L. $$

(2) If $ L^{H} \leq \frac{1 - \delta }{2} L $, then we compute
\begin{eqnarray*}
l^{H} - \frac{(1 - \delta )(1 - 2\delta )}{2\delta } l & \geq & (1 - 2\delta )L + (1 - 2\delta )l - L^{H} - \frac{(1 - \delta )(1 - 2\delta )}{2\delta } l \\
 & = & (1 - 2\delta )L - L^{H} - \frac{(1 - 3\delta )(1 - 2\delta )}{2\delta } l \\
 & \geq & (1 - 2\delta )L - L^{H} - \frac{(1 - 3\delta )(1 - 2\delta )}{2\delta } \Big( \frac{2}{1 - 2\delta }L^{H} - L \Big) \\
 & = & \frac{(1 - 2\delta )}{\delta } \frac{(1 - \delta )}{2} L - \frac{(1 - 2\delta )}{\delta } L^{H} \\
 & \geq & 0.
\end{eqnarray*}
\end{proof}

\begin{claim} \label{cl4}
Let $ w < v $ be such that $ \lambda (H | (w, v)) \leq (1 - \delta )/2 $ and $ t(H, w, v) < s(H, w, v) $. Let $ t(H, w, v) \leq c \leq s(H, w, v) $ and let $ \gamma > 0 $ be a radius such that
$$ \lambda (H | I_{\gamma }(c)) \unrhd 1 - \delta . $$
If one of the points $ s(H, w, v), t(H, w, v) $ is an element of $ I_{\gamma }(c) $, then $ [w, v] \subset I_{\gamma }(c) $.
\end{claim}

\begin{proof}
We may assume, due to the symmetry, that $ t(H, w, v) \in I_{\gamma }(c) $. By Claim \ref{cl1}, we have also $ w \in I_{\gamma }(c) $. It remains to show that $ v \in I_{\gamma }(c) $.

Assume the opposite, i.e., that $ c + \gamma \leq v $. We put
$$ k = \lambda (w, t(H, w, v)), \quad k^{H} = \lambda (H \cap (w, t(H, w, v))), $$
$$ l = \lambda (t(H, w, v), c + \gamma ), \quad l^{H} = \lambda (H \cap (t(H, w, v), c + \gamma )), $$
$$ m = \lambda (c + \gamma , v), \quad m^{H} = \lambda (H \cap (c + \gamma , v)), $$
$$ L = \lambda (w, v), \quad L^{H} = \lambda (H \cap (w, v)). $$
We have $ l \geq \gamma \geq k $, as $ l = c + \gamma - t(H, w, v) \geq \gamma = c - (c - \gamma ) \geq t(H, w, v) - w = k $. Also, $ c - \gamma \geq t(H, w, v) - l $, as $ c - \gamma = c + \gamma - 2\gamma \geq c + \gamma - 2l = t(H, w, v) - l $. Let us realize that
$$ k^{H} + l^{H} \geq k + (1 - 2\delta )l, \quad l^{H} + m^{H} \geq (1 - 2\delta )(l + m), $$
$$ L^{H} \leq \frac{1 - \delta }{2}L. $$
The second inequality is a property of $ t(H, w, v) $ and the third one is an assumption. The first inequality can be obtained from the computation
\begin{eqnarray*}
k^{H} + l^{H} + l - k & = & \lambda (H \cap (w, c + \gamma )) + \lambda (t(H, w, v) - l, w) \\
 & = & \lambda (((-\infty , w) \cup H) \cap (t(H, w, v) - l, c + \gamma )) \\
 & = & 2l \lambda (((-\infty , w) \cup H) | (t(H, w, v) - l, c + \gamma )) \\
 & \geq & 2l \lambda (((-\infty , w) \cup H) | I_{\gamma }(c)) \\
 & \geq & 2l \lambda (H | I_{\gamma }(c)) \\
 & \geq & (1 - \delta ) 2l.
\end{eqnarray*}

We obtain
$$ (1 - 2\delta )l^{H} + L^{H} = (1 - 2\delta )l^{H} + k^{H} + l^{H} + m^{H} \geq (1 - 2\delta )(k^{H} + l^{H}) + l^{H} + m^{H} $$
$$ \geq (1 - 2\delta )k + (1 - 2\delta )^{2}l + (1 - 2\delta )(l + m) = (1 - 2\delta )^{2}l + (1 - 2\delta )L. $$
The inequalities
$$ (1 - \delta )^{2} L = 2(1 - \delta ) \frac{1 - \delta }{2} L \geq 2(1 - \delta ) L^{H} \geq (1 - 2\delta )l^{H} + L^{H} \geq (1 - 2\delta )^{2}l + (1 - 2\delta )L $$
lead quickly to
$$ \delta ^{2} L \geq (1 - 2\delta )^{2}l. $$
On the other hand, the inequalities
$$ (1 - 2\delta )l + \frac{1 - \delta }{2} L \geq (1 - 2\delta )l^{H} + L^{H} \geq (1 - 2\delta )^{2}l + (1 - 2\delta )L $$
lead quickly to
$$ 2\delta (1 - 2\delta )l \geq \frac{1 - 3\delta }{2} L. $$
Hence,
$$ \frac{\delta ^{2}}{1 - 2\delta } L \geq (1 - 2\delta )l \geq \frac{1 - 3\delta }{4\delta } L, $$
and so
$$ 4\delta ^{3} \geq (1 - 2\delta )(1 - 3\delta ). $$
This contradicts our assumption $ \delta < \zeta _{2} $, as $ \zeta _{2} \leq \zeta _{4} $ where $ \zeta _{4} $ is as in Table \ref{roots}.
\end{proof}

\begin{claim} \label{cl5}
For a couple $ w < v $ with $ \lambda (H | (w, v)) \leq (1 - \delta )/2 $, let at least one of the following conditions be satisfied.

{\rm (a)} $ s(H, w, v) \leq t(H, w, v) $,

{\rm (b)} $ \lambda (H | (p, s(H, w, v))) < 1 - 2\delta $ for some $ p < s(H, w, v) $ or $ \lambda (H | (t(H, w, v), q)) < 1 - 2\delta $ for some $ q > t(H, w, v) $,

{\rm (c)} $ t(H, w, v) < s(H, w, v) $ and $ (t(H, w, v), s(H, w, v)) $ is covered by an interval $ (p, q) $ with $ \lambda (H | (p, q)) \leq \delta $.

Then there are intervals $ I_{\alpha }(a) \supset [w, v] $ and $ I_{\beta }(b) \supset [w, v] $ such that
$$ \lambda (H | I_{\alpha }(a)) \unrhd 1 - \delta , \quad \lambda (H | I_{\beta }(b)) \unrhd 1 - \delta $$
and
$$ b - a \geq v' - u - \frac{1}{1 - 2\delta } \lambda (H \cap (u, v')) $$
whenever $ w \leq u < v' \leq v $.
\end{claim}

\begin{proof}
(a) Assume that $ s(H, w, v) \leq t(H, w, v) $. We show that the intervals $ I_{\alpha }(a) $ and $ I_{\beta }(b) $ given by Claim \ref{cl2} work. If $ w \leq u < v' \leq v $, then
\begin{eqnarray*}
(1 - 2\delta )(b - a) & \geq & (1 - 2\delta )(t(H, w, v) - s(H, w, v)) \\
 & = & f_{H, w, v}(t(H, w, v)) - f_{H, w, v}(s(H, w, v)) \\
 & & \quad + \lambda (H \cap (s(H, w, v), t(H, w, v))) \\
 & \geq & f_{H, w, v}(t(H, w, v)) - f_{H, w, v}(s(H, w, v)) \\
 & \geq & f_{H, w, v}(v') - f_{H, w, v}(u) \\
 & = & (1 - 2\delta )(v' - u) - \lambda (H \cap (u, v')).
\end{eqnarray*}
It remains to show that $ w \in I_{\alpha }(a) $ and $ v \in I_{\beta }(b) $. We show only that $ w \in I_{\alpha }(a) $, the proof of $ v \in I_{\beta }(b) $ is the same. We have $ a \leq s(H, w, v) \leq t(H, w, v) \leq v $ and $ v \in I_{\alpha }(a) $. Hence, $ t(H, w, v) \in I_{\alpha }(a) $. Now, $ w \in I_{\alpha }(a) $ by Claim \ref{cl1}.

(b) We may assume, due to the symmetry, that $ \lambda (H | (p, s(H, w, v))) < 1 - 2\delta $ for some $ p < s(H, w, v) $. We have necessarily $ p < w $, as $ \lambda (H | (x, s(H, w, v))) \geq 1 - 2\delta $ when $ w \leq x < s(H, w, v) $. Let us realize that
$$ s(H, p, v) < w. $$
Suppose that $ s(H, p, v) \in [w, v] $. Then $ s(H, p, v) = s(H, w, v) $ (as $ f_{H, p, v} $ and $ f_{H, w, v} $ have the same points of minimum in $ [w, v] $). But this is not possible because, in such a case, $ \lambda (H | (p, s(H, w, v))) = \lambda (H | (p, s(H, p, v))) \geq 1 - 2\delta $ by a property of $ s(H, p, v) $. 

By Claim \ref{cl2} (applied on $ (p, v) $ and on $ (w, v) $), there are intervals $ I_{\alpha }(a) $ and $ I_{\beta }(b) $ such that
$$ a \leq s(H, p, v), \quad v \in I_{\alpha }(a), \quad b \geq t(H, w, v), \quad w \in I_{\beta }(b), $$
$$ \lambda (H | I_{\alpha }(a)) \unrhd 1 - \delta , \quad \lambda (H | I_{\beta }(b)) \unrhd 1 - \delta . $$
We have $ [w, v] \subset (s(H, p, v), v] \subset (a, v] \subset I_{\alpha }(a) $. We already have $ w \in I_{\beta }(b) $. If $ b \geq s(H, w, v) $, then $ s(H, w, v) \in I_{\beta }(b) $, and therefore $ v \in I_{\beta }(b) $ by Claim \ref{cl1}, so $ [w, v] \subset I_{\beta }(b) $. If $ t(H, w, v) \leq b < s(H, w, v) $, then $ [w, v] \subset I_{\beta }(b) $ by Claim \ref{cl4}.

For $ w \leq u < v' \leq v $, we obtain
\begin{eqnarray*}
(1 - 2\delta )(b - a) & \geq & (1 - 2\delta )(t(H, w, v) - s(H, p, v)) \\
 & = & f_{H, w, v}(t(H, w, v)) + \lambda (H \cap (w, t(H, w, v))) \\
 & & \quad - f_{H, p, v}(s(H, p, v)) - \lambda (H \cap (p, s(H, p, v))) \\
 & \geq & f_{H, w, v}(v') + \lambda (H \cap (w, t(H, w, v))) \\
 & & \quad - f_{H, p, v}(u) - \lambda (H \cap (p, s(H, p, v))) \\
 & = & (1 - 2\delta )v' - \lambda (H \cap (w, v')) + \lambda (H \cap (w, t(H, w, v))) \\
 & & \quad - (1 - 2\delta )u + \lambda (H \cap (p, u)) - \lambda (H \cap (p, s(H, p, v))) \\
 & = & (1 - 2\delta )(v' - u) - \lambda (H \cap (u, v')) \\
 & & \quad + \lambda (H \cap (s(H, p, v), t(H, w, v))) \\
 & \geq & (1 - 2\delta )(v' - u) - \lambda (H \cap (u, v')).
\end{eqnarray*}

(c) Since $ \delta < \zeta _{2} \leq \zeta _{7} $ where $ \zeta _{7} $ is as in Table \ref{roots}, we have $ \lambda (H | (p, q)) \leq \delta < (1 - 2\delta )/(1 + 2\delta ) $. We have $ \lambda (H | (p, s(H, w, v))) < 1 - 2\delta $ or $ \lambda (H | (t(H, w, v), q)) < 1 - 2\delta $ because, in the other case, $ \lambda (H | (p, q)) \geq (1 - 2\delta )/(1 + 2\delta ) $ by Claim \ref{cl3}(1). So (b) is satisfied.
\end{proof}

Claim \ref{cl5} says that Lemma \ref{keyprop} is proved for intervals for which one of the conditions (a)--(c) is satisfied. In what follows, we will work with a fixed couple $ w < v $ with $ \lambda (H | (w, v)) \leq (1 - \delta )/2 $ for which none of the conditions (a)--(c) is satisfied. We will denote
$$ s = s(H, w, v), \quad t = t(H, w, v), $$
$$ s' = s(\mathbb{R} \setminus H, t, s), \quad t' = t(\mathbb{R} \setminus H, t, s). $$

\begin{claim} \label{cl6}
If $ \delta < 1/4 $, then $ s' = t $ and $ t' = s $.
\end{claim}

\begin{proof}
If $ s' > t $, then, using the properties of $ s' $ and $ t $,
$$ 2\delta = 1 - (1 - 2\delta ) \geq 1 - \lambda ((\mathbb{R} \setminus H) | (t, s')) = \lambda (H | (t, s')) \geq 1 - 2\delta , $$
and so $ \delta \geq 1/4 $. Hence, $ s' = t $. In the same way, it can be shown that $ t' = s $.
\end{proof}

\begin{claim} \label{cl7}
The points $ s, t, s', t', s(\mathbb{R} \setminus H, t, t') $ and $ t(\mathbb{R} \setminus H, s', s) $ are endpoints of $ H $.
\end{claim}

\begin{proof}
1) We show first that $ s $ is an endpoint of $ H $. It is enough to show that $ s $ does not lie in the interior of $ H $ nor in the interior of $ \mathbb{R} \setminus H $. We have $ s < v $ (if $ s = v $, then $ 1 - 2\delta > (1 - \delta )/2 \geq \lambda (H | (w, v)) = \lambda (H | (w, s)) \geq 1 - 2\delta $). For every $ x $ with $ s < x \leq v $, we have $ \lambda (H | (s, x)) < 1 - 2\delta $, in particular, $ (s, x) \not\subset H $. Therefore, $ s $ does not lie in the interior of $ H $. Further, we have $ w < s $, as $ w \leq t < s $. For every $ x $ with $ w \leq x < s $, we have $ \lambda (H | (x, s)) \geq 1 - 2\delta $, in particular, $ (x, s) \not\subset \mathbb{R} \setminus H $. Therefore, $ s $ does not lie in the interior of $ \mathbb{R} \setminus H $.

Hence, $ s $ is an endpoint of $ H $ and it can be shown in the same way that $ t $ is an endpoint of $ H $, too.

2) The proof can be finished via the following observation: If $ p < q $ are endpoints of $ H $, then $ s(\mathbb{R} \setminus H, p, q) $ and $ t(\mathbb{R} \setminus H, p, q) $ are endpoints of $ H $. Due to the symmetry, it is sufficient to show only that $ s(\mathbb{R} \setminus H, p, q) $ is an endpoint of $ H $. We may assume that $ p < s(\mathbb{R} \setminus H, p, q) < q $. It is enough to show that $ s(\mathbb{R} \setminus H, p, q) $ does not lie in the interior of $ H $ nor in the interior of $ \mathbb{R} \setminus H $. For every $ x $ with $ s(\mathbb{R} \setminus H, p, q) < x \leq q $, we have $ \lambda ((\mathbb{R} \setminus H) | (s(\mathbb{R} \setminus H, p, q), x)) < 1 - 2\delta $, in particular, $ (s(\mathbb{R} \setminus H, p, q), x) \not\subset \mathbb{R} \setminus H $. Therefore, $ s(\mathbb{R} \setminus H, p, q) $ does not lie in the interior of $ \mathbb{R} \setminus H $. Further, for every $ x $ with $ p \leq x < s(\mathbb{R} \setminus H, p, q) $, we have $ \lambda ((\mathbb{R} \setminus H) | (x, s(\mathbb{R} \setminus H, p, q))) \geq 1 - 2\delta $, in particular, $ (x, s(\mathbb{R} \setminus H, p, q)) \not\subset H $. Therefore, $ s(\mathbb{R} \setminus H, p, q) $ does not lie in the interior of $ H $.
\end{proof}

\begin{claim} \label{cl8}
There is an interval $ I_{\beta _{1}}(b_{1}) \supset (s, v) $ such that
$$ \lambda (H | I_{\beta _{1}}(b_{1})) \unrhd 1 - \delta $$
and
$$ b_{1} - s \geq v - u - \frac{2}{1 - \delta } \lambda (H \cap (u, v)) $$
whenever $ w \leq u < v $.
\end{claim}

\begin{proof}
By Claim \ref{cl2}, there is an interval $ I_{\beta _{1}}(b_{1}) $ such that
$$ b_{1} \geq t(H, s, v), \quad s \in I_{\beta _{1}}(b_{1}), $$
$$ \lambda (H | I_{\beta _{1}}(b_{1})) \unrhd 1 - \delta . $$
By Claim \ref{cl1}, we have $ v \in I_{\beta _{1}}(b_{1}) $.

Now, if $ w \leq u < v $, then
\begin{eqnarray*}
(1 - 2\delta )(b_{1} - s) & \geq & (1 - 2\delta )(t(H, s, v) - s) \\
 & = & f_{H, s, v}(t(H, s, v)) + \lambda (H \cap (s, t(H, s, v))) \\
 & & \quad - f_{H, w, v}(s) - \lambda (H \cap (w, s)) \\
 & \geq & f_{H, s, v}(v) - f_{H, w, v}(u) - \lambda (H \cap (w, s)) \\
 & = & (1 - 2\delta )v - \lambda (H \cap (s, v)) \\
 & & \quad - (1 - 2\delta )u + \lambda (H \cap (w, u)) - \lambda (H \cap (w, s)) \\
 & = & (1 - 2\delta )(v - u) - \lambda (H \cap (u, v)),
\end{eqnarray*}
and so
$$ b_{1} - s \geq v - u - \frac{1}{1 - 2\delta }\lambda (H \cap (u, v)) \geq v - u - \frac{2}{1 - \delta }\lambda (H \cap (u, v)). $$
\end{proof}

\begin{claim} \label{cl9}
If $ s' \leq t' $, then there are intervals $ I_{\alpha }(a) $ and $ I_{\beta _{2}}(b_{2}) $ such that
$$ t \leq a \leq s', \quad (w, v) \subset I_{\alpha }(a), \quad t' \leq b_{2} \leq s, \quad (w, v) \subset I_{\beta _{2}}(b_{2}), $$
$$ \lambda (H | I_{\alpha }(a)) \unrhd 1 - \delta , \quad \lambda (H | I_{\beta _{2}}(b_{2})) \unrhd 1 - \delta . $$
\end{claim}

\begin{proof}
We prove the existence of an $ I_{\alpha }(a) $ only because the existence of an $ I_{\beta _{2}}(b_{2}) $ can be proved in the same way. Note that $ t $ and $ s' $ are endpoints of $ H $ by Claim \ref{cl7}. Let $ I_{\alpha }(a) $ be the interval which Lemma \ref{lemmae}(2) gives for $ p = t $ and $ q = s' $. We consider two cases.

(i) Let $ \lambda (H | I_{\alpha }(a)) \geq 1 - \delta $ and $ t \in I_{\alpha }(a) $. We obtain $ (w, v) \subset I_{\alpha }(a) $ from Claim \ref{cl4}, and so $ I_{\alpha }(a) $ has all the required properties.

(ii) Let $ \lambda (H | I_{\alpha }(a)) \leq \delta $ and $ s' \in I_{\alpha }(a) $. Actually, we show that this case is impossible. Using Claim \ref{cl1}, we obtain $ s \in I_{\alpha }(a) $. As $ a \leq s' \leq t' \leq s $, we have also $ t' \in I_{\alpha }(a) $. Using Claim \ref{cl1} again, we obtain $ t \in I_{\alpha }(a) $. So $ I_{\alpha }(a) $ covers $ (t, s) $, which contradicts our assumption that none of the conditions of Claim \ref{cl5} is satisfied.
\end{proof}

\begin{claim} \label{cl10}
If $ t' < s' $, then there are intervals $ I_{\alpha }(a) $ and $ I_{\beta _{2}}(b_{2}) $ such that
$$ t \leq a \leq s(\mathbb{R} \setminus H, t, t'), \quad (w, v) \subset I_{\alpha }(a), \quad t(\mathbb{R} \setminus H, s', s) \leq b_{2} \leq s, \quad (w, v) \subset I_{\beta _{2}}(b_{2}), $$
$$ \lambda (H | I_{\alpha }(a)) \unrhd 1 - \delta , \quad \lambda (H | I_{\beta _{2}}(b_{2})) \unrhd 1 - \delta . $$
\end{claim}

\begin{proof}
We prove the existence of an $ I_{\alpha }(a) $ only because the existence of an $ I_{\beta _{2}}(b_{2}) $ can be proved in the same way. Note that $ t $ and $ s(\mathbb{R} \setminus H, t, t') $ are endpoints of $ H $ by Claim \ref{cl7}. Let $ I_{\alpha }(a) $ be the interval which Lemma \ref{lemmae}(2) gives for $ p = t $ and $ q = s(\mathbb{R} \setminus H, t, t') $. We consider two cases.

(i) Let $ \lambda (H | I_{\alpha }(a)) \geq 1 - \delta $ and $ t \in I_{\alpha }(a) $. We obtain $ (w, v) \subset I_{\alpha }(a) $ from Claim \ref{cl4}, and so $ I_{\alpha }(a) $ has all the required properties.

(ii) Let $ \lambda (H | I_{\alpha }(a)) \leq \delta $ and $ s(\mathbb{R} \setminus H, t, t') \in I_{\alpha }(a) $. Actually, we show that this case is impossible. Using Claim \ref{cl1}, we obtain $ t' \in I_{\alpha }(a) $, and, using Claim \ref{cl1} once more, we obtain $ t \in I_{\alpha }(a) $. At the same time, we have $ a + \alpha \leq s $. Indeed, if $ s \in I_{\alpha }(a) $, then $ I_{\alpha }(a) $ covers $ (t, s) $, which contradicts our assumption that none of the conditions of Claim \ref{cl5} is satisfied.

Let us put
$$ k = \lambda (a - \alpha , t), \quad k^{H} = \lambda (H \cap (a - \alpha , t)), $$
$$ l = \lambda (t, t'), \quad l^{H} = \lambda (H \cap (t, t')), $$
$$ m = \lambda (t', a + \alpha ), \quad m^{H} = \lambda (H \cap (t', a + \alpha )), $$
$$ n = \lambda (a + \alpha , s), \quad n^{H} = \lambda (H \cap (a + \alpha , s)). $$
We obtain inequalities
\begin{eqnarray}
\label{in01} m^{H} + n^{H} & \leq & 2\delta (m + n), \\
\label{in02} l^{H} + m^{H} + n^{H} & \geq & \frac{(1 - \delta )(1 - 2\delta )}{2\delta } (l + m + n), \\
\label{in03} k^{H} + l^{H} + m^{H} & \leq & \delta (k + l + m), \\
\label{in04} k & \geq & l + m - \frac{2}{1 - 2\delta } (l - l^{H}), \\
\label{in05} l^{H} + m^{H} & \geq & (1 - 2\delta )(l + m), \\
\label{in06} k^{H} + l^{H} + m^{H} + n^{H} & \geq & (1 - 2\delta )(k + l + m + n).
\end{eqnarray}
The inequality (\ref{in01}) is a property of $ t' $ (we have $ \lambda ((\mathbb{R} \setminus H) | (t', s)) \geq 1 - 2\delta $). The inequality (\ref{in02}) can be obtained from Claim \ref{cl3}(2), as $ \lambda (H | (w, s)) \geq 1 - 2\delta $ and $ \lambda (H | (t, v)) \geq 1 - 2\delta $. Further, (\ref{in03}) is nothing else than $ \lambda (H | I_{\alpha }(a)) \leq \delta $. To show (\ref{in04}), we compute first that $ l - l^{H} = \lambda ((\mathbb{R} \setminus H) \cap (t, t')) \geq \lambda ((\mathbb{R} \setminus H) \cap (t, s(\mathbb{R} \setminus H, t, t'))) \geq (1 - 2\delta ) \lambda (t, s(\mathbb{R} \setminus H, t, t')) = (1 - 2\delta )(s(\mathbb{R} \setminus H, t, t') - t) \geq (1 - 2\delta )(a - t) $. Now, we obtain (\ref{in04}) by the computation $ l + m - k = t' - t + a + \alpha - t' - t + a - \alpha = 2(a - t) \leq \frac{2}{1 - 2\delta } (l - l^{H}) $. The inequality (\ref{in05}) follows from a property of $ t $. Finally, the inequality (\ref{in06}) follows from our assumption that none of the conditions of Claim \ref{cl5} is satisfied.

Note that $ \delta \geq 1/4 $ by Claim \ref{cl6} and our assumption $ t' < s' $. We obtain from (\ref{in03}), (\ref{in04}) and (\ref{in05}) that
\begin{align*}
\delta (k + l & + m) + (1 - 3\delta )k + \frac{4\delta - 1}{1 - 2\delta } (l^{H} + m^{H}) \\
 & \geq k^{H} + l^{H} + m^{H} + (1 - 3\delta )\Big( l + m - \frac{2}{1 - 2\delta } (l - l^{H}) \Big) + (4\delta - 1)(l + m),
\end{align*}
and so
\begin{equation} \label{in07}
(1 - 2\delta )k - k^{H} + \frac{2 - 6\delta }{1 - 2\delta } (l - l^{H}) - \frac{2 - 6\delta }{1 - 2\delta } l^{H} \geq \frac{2 - 6\delta }{1 - 2\delta } m^{H} \geq 0.
\end{equation}
Further, consider the coefficients
$$ A = \frac{(1 - 2\delta )(1 - 3\delta )}{1 - 3\delta - 2\delta ^{2}}, \quad B = \frac{2\delta (4\delta - 1)}{1 - 3\delta - 2\delta ^{2}}, $$
which have the properties (note that $ 1 - 3\delta - 2\delta ^{2} > 0 $, as $ \delta < \zeta _{2} \leq \zeta _{7} $)
$$ A \geq 0, \quad B \geq 0, \quad A - B = 1, \quad 2\delta A - B \frac{(1 - \delta )(1 - 2\delta )}{2\delta } = 1 - 2\delta . $$
We obtain from (\ref{in01}) and (\ref{in02}) that
$$ 2\delta A (m + n) + B(l^{H} + m^{H} + n^{H}) \geq A(m^{H} + n^{H}) + B \frac{(1 - \delta )(1 - 2\delta )}{2\delta } (l + m + n), $$
and so
\begin{align*}
\Big( 2\delta A - B \frac{(1 - \delta )(1 - 2\delta )}{2\delta } \Big) (l + m + n) & - (A - B)(l^{H} + m^{H} + n^{H}) \\
 & - 2\delta A(l - l^{H}) + (1 - 2\delta ) A l^{H} \geq 0,
\end{align*}
i.e.,
\begin{align*}
(1 - 2\delta ) (l + m + n) & - (l^{H} + m^{H} + n^{H}) \\
 & - 2\delta \frac{(1 - 2\delta )(1 - 3\delta )}{1 - 3\delta - 2\delta ^{2}} (l - l^{H}) + \frac{(1 - 2\delta )^{2}(1 - 3\delta )}{1 - 3\delta - 2\delta ^{2}} l^{H} \geq 0.
\end{align*}
Combining this with (\ref{in07}), we obtain
$$ (1 - 2\delta )(k + l + m + n) - (k^{H} + l^{H} + m^{H} + n^{H}) + \frac{2 - 6\delta }{1 - 2\delta } (l - l^{H}) - \frac{2 - 6\delta }{1 - 2\delta } l^{H} $$
$$ - 2\delta \frac{(1 - 2\delta )(1 - 3\delta )}{1 - 3\delta - 2\delta ^{2}} (l - l^{H}) + \frac{(1 - 2\delta )^{2}(1 - 3\delta )}{1 - 3\delta - 2\delta ^{2}} l^{H} \geq 0. $$
Using (\ref{in06}), we can write
$$ \frac{2 - 6\delta }{1 - 2\delta } (l - l^{H}) - \frac{2 - 6\delta }{1 - 2\delta } l^{H} - 2\delta \frac{(1 - 2\delta )(1 - 3\delta )}{1 - 3\delta - 2\delta ^{2}} (l - l^{H}) + \frac{(1 - 2\delta )^{2}(1 - 3\delta )}{1 - 3\delta - 2\delta ^{2}} l^{H} \geq 0, $$
and so
\begin{equation} \label{in08}
\Big[ \frac{2}{1 - 2\delta } - 2\delta \frac{1 - 2\delta }{1 - 3\delta - 2\delta ^{2}} \Big] (l - l^{H}) - \Big[ \frac{2}{1 - 2\delta } - \frac{(1 - 2\delta )^{2}}{1 - 3\delta - 2\delta ^{2}} \Big] l^{H} \geq 0.
\end{equation}
Note that $ m + n > 0 $ because $ t' < s' \leq s $. Using (\ref{in02}) and $ \delta < \zeta _{2} \leq \zeta _{7} $, we compute
\begin{eqnarray*}
l^{H} + m^{H} + n^{H} & \geq & \frac{(1 - \delta )(1 - 2\delta )}{2\delta } (l + m + n) \\
 & > & \frac{(1 - \delta )(1 - 2\delta )}{2\delta } l + 2\delta (m + n) \geq \frac{(1 - \delta )(1 - 2\delta )}{2\delta } l + m^{H} + n^{H},
\end{eqnarray*}
and so
$$ l^{H} > \frac{(1 - \delta )(1 - 2\delta )}{2\delta } l. $$
It follows that $ l^{H} > 0 $ and that
\begin{equation} \label{in09}
\Big[ \frac{2\delta }{(1 - \delta )(1 - 2\delta )} - 1 \Big] l^{H} > l - l^{H}.
\end{equation}
Now, we must have
\begin{equation} \label{in10}
\frac{2}{1 - 2\delta } - 2\delta \frac{1 - 2\delta }{1 - 3\delta - 2\delta ^{2}} < 0
\end{equation}
because, in the other case, using (\ref{in08}), (\ref{in09}) and $ \delta < \zeta _{2} \leq \zeta _{5} $,
\begin{eqnarray*}
0 & \leq & \Big[ \frac{2}{1 - 2\delta } - 2\delta \frac{1 - 2\delta }{1 - 3\delta - 2\delta ^{2}} \Big] (l - l^{H}) - \Big[ \frac{2}{1 - 2\delta } - \frac{(1 - 2\delta )^{2}}{1 - 3\delta - 2\delta ^{2}} \Big] l^{H} \\
 & \leq & \Big[ \frac{2}{1 - 2\delta } - 2\delta \frac{1 - 2\delta }{1 - 3\delta - 2\delta ^{2}} \Big] \Big[ \frac{2\delta }{(1 - \delta )(1 - 2\delta )} - 1 \Big] l^{H} \\
 & & \quad - \Big[ \frac{2}{1 - 2\delta } - \frac{(1 - 2\delta )^{2}}{1 - 3\delta - 2\delta ^{2}} \Big] l^{H} \\
 & = & - \frac{4\delta ^{2} - 12\delta + 3}{(1 - \delta )(1 - 2\delta )^{2}} l^{H} < 0.
\end{eqnarray*}
So, using (\ref{in08}), (\ref{in10}) and $ \delta < \zeta _{2} $, we can compute
\begin{eqnarray*}
0 & \leq & \Big[ \frac{2}{1 - 2\delta } - 2\delta \frac{1 - 2\delta }{1 - 3\delta - 2\delta ^{2}} \Big] (l - l^{H}) - \Big[ \frac{2}{1 - 2\delta } - \frac{(1 - 2\delta )^{2}}{1 - 3\delta - 2\delta ^{2}} \Big] l^{H} \\
 & \leq & 0 - \Big[ \frac{2}{1 - 2\delta } - \frac{(1 - 2\delta )^{2}}{1 - 3\delta - 2\delta ^{2}} \Big] l^{H} < 0,
\end{eqnarray*}
which is a contradiction.
\end{proof}

\begin{claim} \label{cl11}
There are intervals $ I_{\alpha }(a) \supset (w, v) $ and $ I_{\beta _{2}}(b_{2}) \supset (w, v) $ such that $ a \leq s $ and
$$ \lambda (H | I_{\alpha }(a)) \unrhd 1 - \delta , \quad \lambda (H | I_{\beta _{2}}(b_{2})) \unrhd 1 - \delta $$
and
$$ b_{2} - a \geq \frac{1}{1 - 2\delta } \big( \lambda (H \cap (z, s)) - 2\delta \lambda (z, s) \big) $$
for every $ z $ with $ t \leq z \leq s $.
\end{claim}

\begin{proof}
1) If $ s' \leq t' $, then we take the intervals $ I_{\alpha }(a) $ and $ I_{\beta _{2}}(b_{2}) $ from Claim \ref{cl9}. We have $ a \leq s' \leq s $. For a $ z $ with $ t \leq z \leq s $, we have
\begin{eqnarray*}
(1 - 2\delta )(b_{2} - a) & \geq & (1 - 2\delta )(t' - s') \\
& = & f_{\mathbb{R} \setminus H, t, s}(t') - f_{\mathbb{R} \setminus H, t, s}(s') + \lambda ((s', t') \setminus H) \\
 & \geq & f_{\mathbb{R} \setminus H, t, s}(t') - f_{\mathbb{R} \setminus H, t, s}(s') \\
 & \geq & f_{\mathbb{R} \setminus H, t, s}(s) - f_{\mathbb{R} \setminus H, t, s}(z) \\
 & = & (1 - 2\delta )(s - z) - \lambda ((z, s) \setminus H) \\
 & = & (1 - 2\delta )(s - z) - \lambda (z, s) + \lambda (H \cap (z, s)) \\
 & = & \lambda (H \cap (z, s)) - 2\delta \lambda (z, s).
\end{eqnarray*}

2) If $ t' < s' $, then we take the intervals $ I_{\alpha }(a) $ and $ I_{\beta _{2}}(b_{2}) $ from Claim \ref{cl10}. We have $ a \leq s(\mathbb{R} \setminus H, t, t') \leq t' < s' \leq s $. For a $ z $ with $ t \leq z \leq s $, we have
\begin{eqnarray*}
(1 - 2\delta )(b_{2} - a) & \geq & (1 - 2\delta )(t(\mathbb{R} \setminus H, s', s) - s') \\
 & = & f_{\mathbb{R} \setminus H, s', s}(t(\mathbb{R} \setminus H, s', s)) + \lambda ((s', t(\mathbb{R} \setminus H, s', s)) \setminus H) \\
 & & \quad - f_{\mathbb{R} \setminus H, t, s}(s') - \lambda ((t, s') \setminus H) \\
 & \geq & f_{\mathbb{R} \setminus H, s', s}(s) - f_{\mathbb{R} \setminus H, t, s}(z) - \lambda ((t, s') \setminus H) \\
 & = & (1 - 2\delta )s - \lambda ((s', s) \setminus H) \\
 & & \quad - (1 - 2\delta )z + \lambda ((t, z) \setminus H) - \lambda ((t, s') \setminus H) \\
 & = & (1 - 2\delta )(s - z) - \lambda ((z, s) \setminus H) \\
 & = & \lambda (H \cap (z, s)) - 2\delta \lambda (z, s).
\end{eqnarray*}
\end{proof}

The following claim completes the proof of Lemma \ref{keyprop}.

\begin{claim} \label{cl12}
There is an interval $ I_{\alpha }(a) \supset (w, v) $ such that
$$ \lambda (H | I_{\alpha }(a)) \unrhd 1 - \delta $$
and, for every $ u \in [w, v) $, there is an interval $ I_{\beta }(b) \supset (u, v) $ such that
$$ \lambda (H | I_{\beta }(b)) \unrhd 1 - \delta $$
and
$$ b - a \geq v - u - \frac{2}{1 - \delta } \lambda (H \cap (u, v)). $$
\end{claim}

\begin{proof}
Let $ I_{\beta _{1}}(b_{1}) $ be as in Claim \ref{cl8} and let $ I_{\alpha }(a) $ and $ I_{\beta _{2}}(b_{2}) $ be as in Claim \ref{cl11}. It is enough to show that, for every $ u \in [w, v) $, there is an $ i \in \{ 1, 2 \} $ such that $ (u, v) \subset I_{\beta _{i}}(b_{i}) $ and
$$ b_{i} - a \geq v - u - \frac{2}{1 - \delta } \lambda (H \cap (u, v)). $$
Since $ b_{1} - a \geq b_{1} - s $, we can choose $ i = 1 $ in the case that $ u \geq b_{1} - \beta _{1} $. It remains to show that we can choose $ i = 2 $ in the opposite case, i.e., we have to show that
$$ b_{2} - a \geq v - u - \frac{2}{1 - \delta } \lambda (H \cap (u, v)) $$
whenever $ w \leq u < b_{1} - \beta _{1} $.

So, let $ u $ be such that $ w \leq u < b_{1} - \beta _{1} $. Put
$$ k = \lambda (u, b_{1} - \beta _{1}), \quad k^{H} = \lambda (H \cap (u, b_{1} - \beta _{1})), $$
$$ l = \lambda (b_{1} - \beta _{1}, s), \quad l^{H} = \lambda (H \cap (b_{1} - \beta _{1}, s)), $$
$$ m = \lambda (s, v), \quad m^{H} = \lambda (H \cap (s, v)), $$
$$ L = \lambda (u, v), \quad L^{H} = \lambda (H \cap (u, v)). $$
We obtain inequalities
$$ k^{H} + l^{H} \geq (1 - 2\delta )(k + l), \quad l^{H} + m^{H} \geq (1 - 2\delta )(l + m), $$
$$ b_{2} - a \geq \frac{1}{1 - 2\delta } (l^{H} - 2\delta l). $$
The first inequality is a property of $ s $ and the second inequality follows from Lemma \ref{lemmac}. The third inequality is the inequality from Claim \ref{cl11}, we only need to check that $ t \leq b_{1} - \beta _{1} $. Since $ w \leq u < b_{1} - \beta _{1} $, we have $ w \notin I_{\beta _{1}}(b_{1}) $. Consequently, $ t \notin I_{\beta _{1}}(b_{1}) $ by Claim \ref{cl1}, and so $ t \leq b_{1} - \beta _{1} $.

We need to show that
$$ b_{2} - a \geq L - \frac{2}{1 - \delta }L^{H}. $$
We may assume that
$$ \frac{2}{1 - \delta }L^{H} \leq L $$
because, in the other case, applying the inequality from Claim \ref{cl11} on $ z = s $, we obtain $ b_{2} - a \geq 0 > L - \frac{2}{1 - \delta }L^{H} $. We compute
$$ L^{H} + l^{H} = k^{H} + 2l^{H} + m^{H} \geq (1 - 2\delta )(k + 2l + m) = (1 - 2\delta )L + (1 - 2\delta )l, $$
$$ 2L^{H} \geq L^{H} + l^{H} \geq (1 - 2\delta )L + (1 - 2\delta )l, $$
and so
$$ l^{H} - (1 - 2\delta )l \geq (1 - 2\delta )L - L^{H}, $$
$$ l \leq \frac{2}{1 - 2\delta }L^{H} - L. $$
If $ \delta \leq 1/4 $, then we compute
\begin{eqnarray*}
b_{2} - a & \geq & \frac{1}{1 - 2\delta } (l^{H} - 2\delta l) \geq \frac{1}{1 - 2\delta } (l^{H} - (1 - 2\delta )l) \\
 & \geq & \frac{1}{1 - 2\delta } ((1 - 2\delta )L - L^{H}) = L - \frac{1}{1 - 2\delta }L^{H} \geq L - \frac{2}{1 - \delta }L^{H}.
\end{eqnarray*}
If $ \delta > 1/4 $, then we compute
\begin{eqnarray*}
b_{2} - a & \geq & \frac{1}{1 - 2\delta } (l^{H} - 2\delta l) = \frac{1}{1 - 2\delta } \big( l^{H} - (1 - 2\delta )l - (4\delta - 1)l \big) \\
 & \geq & \frac{1}{1 - 2\delta } \Big[ (1 - 2\delta )L - L^{H} - (4\delta - 1)\Big( \frac{2}{1 - 2\delta }L^{H} - L \Big) \Big] \\
 & = & \frac{1}{1 - 2\delta } \Big( 2\delta L - \frac{6\delta - 1}{1 - 2\delta }L^{H} \Big) .
\end{eqnarray*}
It remains to show that
$$ \frac{1}{1 - 2\delta } \Big( 2\delta L - \frac{6\delta - 1}{1 - 2\delta }L^{H} \Big) \geq L - \frac{2}{1 - \delta }L^{H}, $$
i.e., that
$$ \frac{-14\delta ^{2} + 15\delta - 3}{(4\delta - 1)(1 - 2\delta )(1 - \delta )}L^{H} \leq L. $$
Since $ \delta < \zeta _{2} \leq \zeta _{7} $ where $ \zeta _{7} $ is as in Table \ref{roots}, the desired inequality can be obtained from
$$ \frac{-14\delta ^{2} + 15\delta - 3}{(4\delta - 1)(1 - 2\delta )(1 - \delta )}L^{H} \leq \frac{2}{1 - \delta }L^{H} \leq L. $$
\end{proof}

\section{Lower bound}

In this section, we prove the desired lower bound on $ \delta _{\mathcal{K}} $. Recall that Szenes \cite{szenes} proved that $ \delta _{\mathcal{K}} \geq 0,262978... $ where the exact value of the bound is a root of the polynomial $ 4x^{3} + 2x^{2} + 3x - 1 $. Broadly speaking, the reason why Szenes's bound is not the best possible is that the conclusion of \cite[Lemma 13]{szenes} is not as strong as we need. We prove an analogue of this lemma and, under an additional assumption (C), we obtain the \textquotedblleft right\textquotedblright {} conclusion.

\begin{lemma} \label{lemmaxy}
Let $ H \subset \mathbb{R} $ be a measurable set and let $ 0 < \delta < \zeta _{6} $ where $ \zeta _{6} $ is as in Table \ref{roots}. Let $ p < q $ be such that the interval $ (p, q) $ is covered by a finite number of intervals $ I $ with $ \lambda (H | I) \geq 1 - \delta $ and the following condition is satisfied.

$$
\textrm{\rm (C)}
\left\{ \parbox{11cm}{
If $ p \leq w < v \leq q $ and $ \lambda (H | (w, v)) \leq (1 - \delta )/2 $, then there is an interval $ I_{\alpha }(a) \supset (w, v) $ such that
$$ \lambda (H | I_{\alpha }(a)) \unrhd 1 - \delta  $$
and, for every $ u \in [w, v) $, there is an interval $ I_{\beta }(b) \supset (u, v) $ such that
$$ \lambda (H | I_{\beta }(b)) \unrhd 1 - \delta $$
and
$$ b - a \geq v - u - \frac{2}{1 - \delta } \lambda (H \cap (u, v)). $$
}
\right.
$$

Then
$$ \lambda H \geq \frac{(1 - \delta)(1 + 2\delta)}{1 + 3\delta} (q - p). $$
\end{lemma}

The proof of the lemma will be provided in several steps (Claims \ref{cla}--\ref{clxy}). If the measure $ \lambda H $ is infinite, then the inequality is clearly satisfied. We will assume that the measure $ \lambda H $ is finite.

\begin{claim} \label{cla}
If $ p \leq w < v \leq q $ and $ \lambda (H | (w, v)) \leq (1 - \delta )/2 $, then there is a point $ v' \leq w $ such that, for every $ u \in [w, v) $, there are intervals $ I_{\alpha }(a) $ and $ I_{\beta }(b) $ such that $ v' = a - \alpha $ and

{\rm (a)} $ (u, v) \subset I_{\alpha }(a), (u, v) \subset I_{\beta }(b), $

{\rm (b)} $ a - \alpha \leq b - \beta , a + \alpha \leq b + \beta , $

{\rm (c)} $ \lambda (H | I_{\alpha }(a)) \geq 1 - \delta , \lambda (H | I_{\beta }(b)) \geq 1 - \delta , $

{\rm (d)} $ b - a \geq v - u - \frac{2}{1 - \delta } \lambda (H \cap (u, v)), $

{\rm (e)} $ \lambda (H | (u, a + \alpha )) \geq 1 - 2\delta , \lambda (H | (b - \beta , v)) \geq 1 - 2\delta , $

{\rm (f)} if $ a - \alpha < b - \beta $, then $ \lambda (H | (a - \alpha , b - \beta )) \geq 1 - 2\delta $, if $ a + \alpha < b + \beta $, then $ \lambda (H | (a + \alpha , b + \beta )) \geq 1 - 2\delta . $
\end{claim}

\begin{proof}
Let $ I_{\alpha ^{0}}(a^{0}) \supset (w, v) $ be an interval which (C) gives for $ (w, v) $. Let $ v' $ be a point of the maximum of the function $ x \in (- \infty , a^{0} - \alpha ^{0}] \mapsto (1 - \delta )x + \lambda (H \cap (x, a^{0} - \alpha ^{0})) $, so $ v' $ has the property that
$$ x < v' \quad \Rightarrow \quad \lambda (H | (x, v')) \leq 1 - \delta , $$
$$ v' < x \leq a^{0} - \alpha ^{0} \quad \Rightarrow \quad \lambda (H | (v', x)) \geq 1 - \delta . $$
Let us show that this choice works. Let $ u \in [w, v) $. Then (C) gives some answer $ I_{\beta ^{0}}(b^{0}) \supset (u, v) $. We construct suitable intervals $ I_{\alpha }(a) $ and $ I_{\beta }(b) $ in two steps.

1) We put
$$ I_{\alpha ^{1}}(a^{1}) = (v', a^{0} + \alpha ^{0}), \quad I_{\beta ^{1}}(b^{1}) = I_{\beta ^{0}}(b^{0}) \setminus (-\infty , v']. $$
Let us realize that the following conditions hold.

{\rm (a')} $ (u, v) \subset I_{\alpha ^{1}}(a^{1}), (u, v) \subset I_{\beta ^{1}}(b^{1}), $

{\rm (b')} $ a^{1} - \alpha ^{1} \leq b^{1} - \beta ^{1}, $

{\rm (c')} $ \lambda (H | I_{\alpha ^{1}}(a^{1})) \geq 1 - \delta , \lambda (H | I_{\beta ^{1}}(b^{1})) \geq 1 - \delta , $

{\rm (d')} $ b^{1} - a^{1} \geq v - u - \frac{2}{1 - \delta } \lambda (H \cap (u, v)), $

{\rm (e')} $ \lambda (H | (b^{1} - \beta ^{1}, v)) \geq 1 - 2\delta , $

{\rm (f')} if $ a^{1} - \alpha ^{1} < b^{1} - \beta ^{1} $, then $ \lambda (H | (a^{1} - \alpha ^{1}, b^{1} - \beta ^{1})) \geq 1 - 2\delta . $

The condition (a') holds due to $ (u, v) \subset I_{\alpha ^{0}}(a^{0}) $, $ (u, v) \subset I_{\beta ^{0}}(b^{0}) $ and $ v' \leq a^{0} - \alpha ^{0} \leq u $. The condition (b') holds due to $ a^{1} - \alpha ^{1} = v' $ and $ (-\infty , v'] \cap I_{\beta ^{1}}(b^{1}) = \emptyset $. By the choice of $ v' $, we have $ v' < a^{0} - \alpha ^{0} \Rightarrow \lambda (H | (v', a^{0} - \alpha ^{0})) \geq 1 - \delta $ and $ b^{0} - \beta ^{0} < v' \Rightarrow \lambda (H | (b^{0} - \beta ^{0}, v')) \leq 1 - \delta $, and to prove (c'), it is enough to use $ \lambda (H | I_{\alpha ^{0}}(a^{0})) \geq 1 - \delta $ and $ \lambda (H | I_{\beta ^{0}}(b^{0})) \geq 1 - \delta $. Since clearly $ a^{1} \leq a^{0} $ and $ b^{1} \geq b^{0} $, we obtain
$$ b^{1} - a^{1} \geq b^{0} - a^{0} \geq v - u - \frac{2}{1 - \delta } \lambda (H \cap (u, v)), $$
which gives (d'). To show (e'), we consider two cases. If $ b^{1} - \beta ^{1} > v' $, then $ b^{1} - \beta ^{1} = b^{0} - \beta ^{0} $, and so, we can apply Lemma \ref{lemmac} on $ I_{\beta ^{0}}(b^{0}) $ to obtain $ \lambda (H | (b^{1} - \beta ^{1}, v)) = \lambda (H | (b^{0} - \beta ^{0}, v)) \geq 1 - 2\delta $. If $ b^{1} - \beta ^{1} = v' $, then we apply Lemma \ref{lemmac} on $ I_{\alpha ^{0}}(a^{0}) $ to obtain $ \lambda (H | (a^{0} - \alpha ^{0}, v)) \geq 1 - 2\delta $, and we use that $ v' < a^{0} - \alpha ^{0} \Rightarrow \lambda (H | (v', a^{0} - \alpha ^{0})) \geq 1 - \delta \geq 1 - 2\delta $ by the choice of $ v' $. Finally, to show (f'), assume that $ a^{1} - \alpha ^{1} < b^{1} - \beta ^{1} $. We consider two cases again. If $ b^{1} - \beta ^{1} \leq a^{0} - \alpha ^{0} $, then, by the choice of $ v' $, we have $ \lambda (H | (a^{1} - \alpha ^{1}, b^{1} - \beta ^{1})) = \lambda (H | (v', b^{1} - \beta ^{1})) \geq 1 - \delta \geq 1 - 2\delta $. If $ b^{1} - \beta ^{1} > a^{0} - \alpha ^{0} $, then, by the choice of $ v' $, we have $ v' < a^{0} - \alpha ^{0} \Rightarrow \lambda (H | (v', a^{0} - \alpha ^{0})) \geq 1 - \delta \geq 1 - 2\delta $, and, applying Lemma \ref{lemmac} on $ I_{\alpha ^{0}}(a^{0}) $, we obtain $ \lambda (H | (a^{0} - \alpha ^{0}, b^{1} - \beta ^{1})) \geq 1 - 2\delta $. Hence, $ \lambda (H | (a^{1} - \alpha ^{1}, b^{1} - \beta ^{1})) = \lambda (H | (v', b^{1} - \beta ^{1})) \geq 1 - 2\delta $.

2) If $ a^{1} + \alpha ^{1} \leq b^{1} + \beta ^{1} $, then we put
$$ I_{\alpha }(a) = I_{\alpha ^{1}}(a^{1}), \quad I_{\beta }(b) = I_{\beta ^{1}}(b^{1}). $$
If $ a^{1} + \alpha ^{1} > b^{1} + \beta ^{1} $ and $ \lambda (H | (b^{1} + \beta ^{1}, a^{1} + \alpha ^{1})) \leq 1 - \delta $, then we put
$$ I_{\alpha }(a) = (a^{1} - \alpha ^{1}, b^{1} + \beta ^{1}), \quad I_{\beta }(b) = I_{\beta ^{1}}(b^{1}). $$
If $ a^{1} + \alpha ^{1} > b^{1} + \beta ^{1} $ and $ \lambda (H | (b^{1} + \beta ^{1}, a^{1} + \alpha ^{1})) > 1 - \delta $, then we put
$$ I_{\alpha }(a) = I_{\alpha ^{1}}(a^{1}), \quad I_{\beta }(b) = (b^{1} - \beta ^{1}, a^{1} + \alpha ^{1}). $$

Let us check that the conditions (a)--(f) hold. The condition (a) easily follows from (a'). We get from (b') that $ a - \alpha = a^{1} - \alpha ^{1} \leq b^{1} - \beta ^{1} = b - \beta $. The second part of (b) is clear (if $ a^{1} + \alpha ^{1} \leq b^{1} + \beta ^{1} $, then $ a + \alpha = a^{1} + \alpha ^{1} \leq b^{1} + \beta ^{1} = b + \beta $, if $ a^{1} + \alpha ^{1} > b^{1} + \beta ^{1} $, then $ I_{\alpha }(a) $ and $ I_{\beta }(b) $ are chosen so that $ a + \alpha = b + \beta $). The condition (c) easily follows from (c') and from the definitions of $ I_{\alpha }(a) $ and $ I_{\beta }(b) $. We obtain (d) from (d') and from the obvious fact that $ a \leq a^{1} $ and $ b \geq b^{1} $. The second part of (e) is nothing else than (e'). To show the first part, we realize that $ a + \alpha \in \{ a^{1} + \alpha ^{1}, b^{1} + \beta ^{1} \} = \{ a^{0} + \alpha ^{0}, b^{0} + \beta ^{0} \} $ and use Lemma \ref{lemmac} to obtain $ \lambda (H | (u, a^{0} + \alpha ^{0})) \geq 1 - 2\delta $ and $ \lambda (H | (u, b^{0} + \beta ^{0})) \geq 1 - 2\delta $. The first part of (f) is nothing else than (f'). Let us show the second part of (f). Assume that $ a + \alpha < b + \beta $. We have $ a^{1} + \alpha ^{1} \leq b^{1} + \beta ^{1} $ (indeed, if $ a^{1} + \alpha ^{1} > b^{1} + \beta ^{1} $, then $ I_{\alpha }(a) $ and $ I_{\beta }(b) $ are chosen so that $ a + \alpha = b + \beta $). So we have $ I_{\alpha }(a) = I_{\alpha ^{1}}(a^{1}) $ and $ I_{\beta }(b) = I_{\beta ^{1}}(b^{1}) $. Now, we apply Lemma \ref{lemmac} on $ I_{\beta ^{0}}(b^{0}) $ to obtain $ \lambda (H | (a + \alpha , b + \beta )) = \lambda (H | (a^{1} + \alpha ^{1}, b^{1} + \beta ^{1})) = \lambda (H | (a^{0} + \alpha ^{0}, b^{0} + \beta ^{0})) \geq 1 - 2\delta $.

To finish the proof, we realize that $ a - \alpha = a^{1} - \alpha ^{1} = v' $.
\end{proof}

\begin{claim} \label{clb}
There exist $ n \in \mathbb{N} $, points $ u_{i}, v_{i}, 1 \leq i \leq n, $ and intervals $ I_{\alpha _{i}}(a_{i}), I_{\beta _{i}}(b_{i}) $, $ 1 \leq i \leq n, $ such that
\begin{itemize}
\item $ p = u_{1} < v_{1} \leq u_{2} < v_{2} \leq \dots \leq u_{n} < v_{n} = q, $
\item $ a_{i} - \alpha _{i} = v_{i-1}, 2 \leq i \leq n, $
\item $ b_{i} + \beta _{i} = u_{i+1}, 1 \leq i \leq n - 1, $
\end{itemize}
and, for every $ i $ with $ 1 \leq i \leq n $,

{\rm (a)} $ (u_{i}, v_{i}) \subset I_{\alpha _{i}}(a_{i}), (u_{i}, v_{i}) \subset I_{\beta _{i}}(b_{i}), $

{\rm (b)} $ a_{i} - \alpha _{i} \leq b_{i} - \beta _{i} , a_{i} + \alpha _{i} \leq b_{i} + \beta _{i}, $

{\rm (c)} $ \lambda (H | I_{\alpha _{i}}(a_{i})) \geq 1 - \delta , \lambda (H | I_{\beta _{i}}(b_{i})) \geq 1 - \delta , $

{\rm (d)} $ b_{i} - a_{i} \geq v_{i} - u_{i} - \frac{2}{1 - \delta } \lambda (H \cap (u_{i}, v_{i})), $

{\rm (e)} if $ I_{\alpha _{i}}(a_{i}) \neq I_{\beta _{i}}(b_{i}) $, then $ \lambda (H | (u_{i}, a_{i} + \alpha _{i})) \geq 1 - 2\delta , \lambda (H | (b_{i} - \beta _{i}, v_{i})) \geq 1 - 2\delta , $

{\rm (f)} if $ a_{i} - \alpha _{i} < b_{i} - \beta _{i} $, then $ \lambda (H | (a_{i} - \alpha _{i}, b_{i} - \beta _{i})) \geq 1 - 2\delta $, if $ a_{i} + \alpha _{i} < b_{i} + \beta _{i} $, then $ \lambda (H | (a_{i} + \alpha _{i}, b_{i} + \beta _{i})) \geq 1 - 2\delta . $
\end{claim}

\begin{proof}
We define recursively
$$ t_{1} = p $$
and, if $ t_{i} < q $,
$$ t_{i+1} = \sup \big\{ t > t_{i} : \textrm{$ \lambda (H | (s, t)) \geq 1 - \delta $ for some $ s \leq t_{i} $} \big\} . $$
Note that $ t_{i+1} $ is well defined due to the assumption that $ (p, q) $ is covered by a finite number of intervals $ I $ with $ \lambda (H | I) \geq 1 - \delta $ and the assumption that the measure $ \lambda H $ is finite. Let $ n $ be the number such that $ t_{n} < q \leq t_{n+1} $. We define recursively, for $ i = n, n - 1, \dots , 1, 0, $ points $ v_{i} \leq t_{i+1} $ with $ i \geq 1 \Rightarrow v_{i} \in (t_{i}, t_{i+1}] $, as follows.

(i.0) If $ i = n $, then put $ v_{n} = q \in (t_{n}, t_{n+1}] $.

(i.1) If $ 1 \leq i \leq n $ and there is some $ w $ with $ p \leq w \leq t_{i} $ such that $ \lambda (H | (w, v_{i})) \leq (1 - \delta )/2 $, then we define first $ w_{i} $ as the minimal $ w \in [p, t_{i}] $ with this property. Then we take $ v_{i-1} $ as a point which Claim \ref{cla} gives for $ w_{i} $ and $ v_{i} $. We have $ v_{i-1} \leq w_{i} \leq t_{i} $ in particular. Let $ I_{\alpha }(a) $ be an interval which Claim \ref{cla} gives (for $ u \in [w_{i}, v_{i}) $ chosen arbitrarily), so we have $ \lambda (H | I_{\alpha }(a)) \geq 1 - \delta $, $ v_{i-1} = a - \alpha $ and $ a + \alpha \geq v_{i} > t_{i} $. If $ i - 1 \geq 1 $, then $ v_{i-1} = a - \alpha > t_{i-1} $ (in the opposite case that $ a - \alpha \leq t_{i-1} $, we have $ a + \alpha \leq t_{i} $ by the definition of $ t_{i} $).

(i.2) If $ 1 \leq i \leq n $ and there is no $ w $ with $ p \leq w \leq t_{i} $ such that $ \lambda (H | (w, v_{i})) \leq (1 - \delta )/2 $, then we take $ v_{i-1} $ as a point such that $ v_{i-1} \leq t_{i} $ and $ \lambda (H | (v_{i-1}, t_{i+1})) \geq 1 - \delta $. Such a point exists due to the definition of $ t_{i+1} $. If $ i - 1 \geq 1 $, then $ v_{i-1} > t_{i-1} $ (in the opposite case that $ v_{i-1} \leq t_{i-1} $, we have $ t_{i+1} \leq t_{i} $ by the definition of $ t_{i} $, which is not possible).

Further, we define recursively, for $ i = 1, 2, \dots , n, $ points $ u_{i} \in [v_{i-1}, t_{i}] $ and intervals $ I_{\alpha _{i}}(a_{i}) \supset (u_{i}, v_{i}) $ and $ I_{\beta _{i}}(b_{i}) \supset (u_{i}, v_{i}) $ with $ \lambda (H | I_{\alpha _{i}}(a_{i})) \geq 1 - \delta , \lambda (H | I_{\beta _{i}}(b_{i})) \geq 1 - \delta $, as follows.

(ii.0) If $ i = 1 $, then put $ u_{1} = p = t_{1} $.

(ii.1) If $ 1 < i \leq n $, then put $ u_{i} = b_{i-1} + \beta _{i-1} $. We obtain $ u_{i} \leq t_{i} $ from the definition of $ t_{i} $ and from $ b_{i-1} - \beta _{i-1} \leq u_{i-1} \leq t_{i-1} $. At the same time, $ u_{i} = b_{i-1} + \beta _{i-1} \geq v_{i-1} $.

(iii.1) If $ 1 \leq i \leq n $ and $ \lambda (H | (u_{i}, v_{i})) \leq (1 - \delta )/2 $, then $ v _{i-1} $ was chosen by (i.1) and $ w_{i} \leq u_{i} \leq t_{i} < v_{i} $. We take a couple $ I_{\alpha _{i}}(a_{i}), I_{\beta _{i}}(b_{i}) $ which Claim \ref{cla} gives for $ u_{i} $. The conditions (a)--(f) for $ i $ follow from the conditions (a)--(f) in Claim \ref{cla}.

(iii.2) If $ 1 \leq i \leq n $ and $ \lambda (H | (u_{i}, v_{i})) > (1 - \delta )/2 $, then we realize first that there is an interval $ I_{\alpha }(a) $ such that $ \lambda (H | I_{\alpha }(a)) \geq 1 - \delta $, $ v_{i-1} = a - \alpha $ and $ v_{i} \leq a + \alpha $. If $ v_{i-1} $ was chosen by (i.1), then we can take an interval $ I_{\alpha }(a) $ which Claim \ref{cla} gives (for $ u \in [w_{i}, v_{i}) $ chosen arbitrarily). If $ v_{i-1} $ was chosen by (i.2), then the choice $ I_{\alpha }(a) = (v_{i-1}, t_{i+1}) $ works. Now, we put $ I_{\alpha _{i}}(a_{i}) = I_{\beta _{i}}(b_{i}) = I_{\alpha }(a) $. The condition (a) holds due to $ a - \alpha = v_{i-1} \leq u_{i} \leq t_{i} < v_{i} \leq a + \alpha $. The condition (c) is clear and the conditions (b), (e), (f) are immediate consequences of $ I_{\alpha _{i}}(a_{i}) = I_{\beta _{i}}(b_{i}) $. The condition (d) follows from $ b_{i} - a_{i} = 0 $ and from $ \lambda (H | (u_{i}, v_{i})) > (1 - \delta )/2 $.

So the objects are constructed, and we easily check the required properties now. We have $ p = u_{1} \leq t_{1} < v_{1} \leq u_{2} \leq t_{2} < v_{2} \leq \dots \leq u_{n} \leq t_{n} < v_{n} = q $. The condition $ a_{i} - \alpha _{i} = v_{i-1}, 2 \leq i \leq n, $ follows from the choice of $ I_{\alpha _{i}}(a_{i}) $. The condition $ b_{i} + \beta _{i} = u_{i+1}, 1 \leq i \leq n - 1, $ follows from the choice of $ u_{i+1} $. Finally, the conditions (a)--(f) were already discussed in (iii.1) and (iii.2).
\end{proof}

Let us fix such a system of points and intervals Claim \ref{clb} gives and define
$$ X = \{ i \in \mathbb{N} : 1 \leq i \leq n - 1, b_{i+1} - \beta _{i+1} \leq a_{i} + \alpha _{i} \} , $$
$$ Y = \{ 1, 2, \dots , n - 1 \} \setminus X = \{ i \in \mathbb{N} : 1 \leq i \leq n - 1, b_{i+1} - \beta _{i+1} > a_{i} + \alpha _{i} \} . $$
Further, for every $ i $ with $ 1 \leq i \leq n $, define
$$ \kappa _{i} = b_{i} - a_{i} - (v_{i} - u_{i}) + \frac{2}{1 - \delta } \lambda (H \cap (u_{i}, v_{i})), $$
$$ \chi _{i}^{-} = \frac{1}{1 - \delta } \lambda (H \cap (a_{i} - \alpha _{i}, b_{i} - \beta _{i})) - \lambda (a_{i} - \alpha _{i}, b_{i} - \beta _{i}), $$
$$ \chi _{i}^{+} = \frac{1}{1 - \delta } \lambda (H \cap (a_{i} + \alpha _{i}, b_{i} + \beta _{i})) - \lambda (a_{i} + \alpha _{i}, b_{i} + \beta _{i}), $$
$$ \psi _{i}^{-} = \lambda ((b_{i} - \beta _{i}, u_{i}) \setminus H), $$
$$ \psi _{i}^{+} = \lambda ((v_{i}, a_{i} + \alpha _{i}) \setminus H). $$

\begin{claim} \label{clc}
If $ \omega _{i}, 1 \leq i \leq n, $ are numbers with $ 0 \leq \omega _{i} \leq 2 $, then
\begin{eqnarray*}
\Big( \frac{1}{2\delta } + \frac{2}{1 - \delta } \Big) 2\lambda H & \geq & \Big( \frac{1}{2\delta } + 1 \Big) 2(q - p) \\
 & & + \sum_{i=1}^{n} \kappa _{i} + \sum_{i=1}^{n-1} \Big[ \frac{1}{2\delta }(\psi _{i}^{+} + \psi _{i+1}^{-}) + \omega _{i} \chi _{i}^{+} + (2 - \omega _{i+1}) \chi _{i+1}^{-} \Big] .
\end{eqnarray*}
\end{claim}

\begin{proof}
For $ 1 \leq i \leq n $, let us put
$$ k_{i} = \lambda (a_{i} - \alpha _{i}, b_{i} - \beta _{i}), \quad k_{i}^{H} = \lambda (H \cap (a_{i} - \alpha _{i}, b_{i} - \beta _{i})), $$
$$ l_{i} = \lambda (b_{i} - \beta _{i}, u_{i}), \quad l_{i}^{H} = \lambda (H \cap (b_{i} - \beta _{i}, u_{i})), $$
$$ m_{i} = \lambda (u_{i}, v_{i}), \quad m_{i}^{H} = \lambda (H \cap (u_{i}, v_{i})), $$
$$ n_{i} = \lambda (v_{i}, a_{i} + \alpha _{i}), \quad n_{i}^{H} = \lambda (H \cap (v_{i}, a_{i} + \alpha _{i})), $$
$$ o_{i} = \lambda (a_{i} + \alpha _{i}, b_{i} + \beta _{i}), \quad o_{i}^{H} = \lambda (H \cap (a_{i} + \alpha _{i}, b_{i} + \beta _{i})). $$
We have
$$ \kappa _{i} = \frac{1}{2}(k_{i} + o_{i}) - m_{i} + \frac{2}{1 - \delta } m_{i}^{H}, $$
$$ \chi _{i}^{-} = \frac{1}{1 - \delta } k_{i}^{H} - k_{i}, \quad \chi _{i}^{+} = \frac{1}{1 - \delta } o_{i}^{H} - o_{i}, $$
$$ \psi _{i}^{-} = l_{i} - l_{i}^{H}, \quad \psi _{i}^{+} = n_{i} - n_{i}^{H}. $$
Using (c), we obtain
$$
\begin{aligned}
 & \hspace{-1cm} \Big( \frac{1}{2\delta } + \frac{\omega _{i}}{1 - \delta } \Big) \lambda (H \cap I_{\alpha _{i}}(a_{i})) + \Big( \frac{1}{2\delta } + \frac{2 - \omega _{i}}{1 - \delta } \Big) \lambda (H \cap I_{\beta _{i}}(b_{i})) \\
 \geq \; & \Big( \frac{1}{2\delta } + \frac{\omega _{i}}{1 - \delta } \Big) (1 - \delta )\lambda I_{\alpha _{i}}(a_{i}) + \Big( \frac{1}{2\delta } + \frac{2 - \omega _{i}}{1 - \delta } \Big) (1 - \delta )\lambda I_{\beta _{i}}(b_{i}) \\
 = \; & \Big( \frac{1}{2\delta } - \frac{1}{2} + \omega _{i} \Big) \lambda I_{\alpha _{i}}(a_{i}) + \Big( \frac{1}{2\delta } - \frac{1}{2} + 2 - \omega _{i} \Big) \lambda I_{\beta _{i}}(b_{i}).
\end{aligned}
$$
It follows that
$$
\begin{aligned}
 & \hspace{-1cm} \Big( \frac{1}{2\delta } + \frac{\omega _{i}}{1 - \delta } \Big) (k_{i}^{H} + l_{i}^{H} + m_{i}^{H} + n_{i}^{H}) + \Big( \frac{1}{2\delta } + \frac{2 - \omega _{i}}{1 - \delta } \Big) (l_{i}^{H} + m_{i}^{H} + n_{i}^{H} + o_{i}^{H}) \\
 + & \frac{1}{2}(k_{i} + o_{i}) - m_{i} + \frac{2}{1 - \delta } m_{i}^{H} + \frac{1}{2\delta }(l_{i} - l_{i}^{H} + n_{i} - n_{i}^{H}) \\
 + & (2 - \omega _{i}) \Big( \frac{1}{1 - \delta } k_{i}^{H} - k_{i} \Big) + \omega _{i} \Big( \frac{1}{1 - \delta } o_{i}^{H} - o_{i} \Big) \\
 & \hspace{-1cm} \geq \Big( \frac{1}{2\delta } - \frac{1}{2} + \omega _{i} \Big) (k_{i} + l_{i} + m_{i} + n_{i}) + \Big( \frac{1}{2\delta } - \frac{1}{2} + 2 - \omega _{i} \Big) (l_{i} + m_{i} + n_{i} + o_{i}) \\
 + & \kappa _{i} + \frac{1}{2\delta }(\psi _{i}^{-} + \psi _{i}^{+}) + (2 - \omega _{i}) \chi _{i}^{-} + \omega _{i} \chi _{i}^{+}.
\end{aligned}
$$
This leads straightforwardly to
$$
\begin{aligned}
 & \hspace{-1.5cm} \Big( \frac{1}{2\delta } + \frac{2}{1 - \delta } \Big) (k_{i}^{H} + l_{i}^{H} + 2m_{i}^{H} + n_{i}^{H} + o_{i}^{H}) \\
 \geq \; & \Big( \frac{1}{2\delta } + 1 \Big) (k_{i} + l_{i} + 2m_{i} + n_{i} + o_{i}) \\
 & + \kappa _{i} + \frac{1}{2\delta }(\psi _{i}^{-} + \psi _{i}^{+}) + (2 - \omega _{i}) \chi _{i}^{-} + \omega _{i} \chi _{i}^{+},
\end{aligned}
$$
i.e.,
$$
\begin{aligned}
 & \hspace{-1.5cm} \Big( \frac{1}{2\delta } + \frac{2}{1 - \delta } \Big) \big( \lambda (H \cap (a_{i} - \alpha _{i}, v_{i})) + \lambda (H \cap (u_{i}, b_{i} + \beta _{i})) \big) \\
 \geq \; & \Big( \frac{1}{2\delta } + 1 \Big) \big( \lambda (a_{i} - \alpha _{i}, v_{i}) + \lambda (u_{i}, b_{i} + \beta _{i}) \big) \\
 & + \kappa _{i} + \frac{1}{2\delta }(\psi _{i}^{-} + \psi _{i}^{+}) + (2 - \omega _{i}) \chi _{i}^{-} + \omega _{i} \chi _{i}^{+}.
\end{aligned}
$$
Recall that $ a_{i} - \alpha _{i} = v_{i-1}, 2 \leq i \leq n, $ and $ b_{i} + \beta _{i} = u_{i+1}, 1 \leq i \leq n - 1 $. Summing these inequalities for $ i = 1, 2, \dots , n $, we obtain
$$
\begin{aligned}
 & \hspace{-1.5cm} \Big( \frac{1}{2\delta } + \frac{2}{1 - \delta } \Big) \big( \lambda (H \cap (a_{1} - \alpha _{1}, q)) + \lambda (H \cap (p, b_{n} + \beta _{n})) \big) \\
 \geq \; & \Big( \frac{1}{2\delta } + 1 \Big) \big( \lambda (a_{1} - \alpha _{1}, q) + \lambda (p, b_{n} + \beta _{n}) \big) \\
 & + \sum_{i=1}^{n} \Big[ \kappa _{i} + \frac{1}{2\delta }(\psi _{i}^{-} + \psi _{i}^{+}) + (2 - \omega _{i}) \chi _{i}^{-} + \omega _{i} \chi _{i}^{+} \Big] .
\end{aligned}
$$

Now, we compute
$$
\begin{aligned}
 & \hspace{-0.75cm} \Big( \frac{1}{2\delta } + 1 \Big) \big( \lambda (a_{1} - \alpha _{1}, p) + \lambda (q, b_{n} + \beta _{n}) \big) + \frac{1}{2\delta }(\psi _{1}^{-} + \psi _{n}^{+}) + (2 - \omega _{1}) \chi _{1}^{-} + \omega _{n} \chi _{n}^{+} \\
 \hspace{0.5cm} \geq \; & \Big( \frac{1}{2\delta } + 1 \Big) \big( \lambda (a_{1} - \alpha _{1}, b_{1} - \beta _{1}) + \lambda (a_{n} + \alpha _{n}, b_{n} + \beta _{n}) \big) + (2 - \omega _{1}) \chi _{1}^{-} + \omega _{n} \chi _{n}^{+} \\
 \hspace{0.5cm} = \; & \Big( \frac{1}{2\delta } + 1 \Big) (k_{1} + o_{n}) + (2 - \omega _{1}) \Big( \frac{1}{1 - \delta } k_{1}^{H} - k_{1} \Big) + \omega _{n} \Big( \frac{1}{1 - \delta } o_{n}^{H} - o_{n} \Big) \\
 \hspace{0.5cm} \geq \; & \Big( \frac{1}{2\delta } + 1 \Big) (k_{1} + o_{n}) - 2k_{1} - 2o_{n} = \Big( \frac{1}{2\delta } - 1 \Big) (k_{1} + o_{n}) \geq 0.
\end{aligned}
$$
Finally,
$$
\begin{aligned}
 & \hspace{-0.75cm} \Big( \frac{1}{2\delta } + \frac{2}{1 - \delta } \Big) 2\lambda H \geq \Big( \frac{1}{2\delta } + \frac{2}{1 - \delta } \Big) \big (\lambda (H \cap (a_{1} - \alpha _{1}, q)) + \lambda (H \cap (p, b_{n} + \beta _{n})) \big) \\
 \geq \; & \Big( \frac{1}{2\delta } + 1 \Big) \big( \lambda (a_{1} - \alpha _{1}, q) + \lambda (p, b_{n} + \beta _{n}) \big) \\
 & + \sum_{i=1}^{n} \Big[ \kappa _{i} + \frac{1}{2\delta }(\psi _{i}^{-} + \psi _{i}^{+}) + (2 - \omega _{i}) \chi _{i}^{-} + \omega _{i} \chi _{i}^{+} \Big] \\
 = \; & \Big( \frac{1}{2\delta } + 1 \Big) 2\lambda (p, q) + \sum_{i=1}^{n} \kappa _{i} + \sum_{i=1}^{n-1} \Big[ \frac{1}{2\delta }(\psi _{i}^{+} + \psi _{i+1}^{-}) + \omega _{i} \chi _{i}^{+} + (2 - \omega _{i+1}) \chi _{i+1}^{-} \Big] \\
 & + \Big( \frac{1}{2\delta } + 1 \Big) \big( \lambda (a_{1} - \alpha _{1}, p) + \lambda (q, b_{n} + \beta _{n}) \big) \\
 & + \frac{1}{2\delta }(\psi _{1}^{-} + \psi _{n}^{+}) + (2 - \omega _{1}) \chi _{1}^{-} + \omega _{n} \chi _{n}^{+} \\
 \geq \; & \Big( \frac{1}{2\delta } + 1 \Big) 2\lambda (p, q) + \sum_{i=1}^{n} \kappa _{i} + \sum_{i=1}^{n-1} \Big[ \frac{1}{2\delta }(\psi _{i}^{+} + \psi _{i+1}^{-}) + \omega _{i} \chi _{i}^{+} + (2 - \omega _{i+1}) \chi _{i+1}^{-} \Big] .
\end{aligned}
$$
\end{proof}

\begin{claim} \label{clx}
{\rm (1)} If $ 1 \leq i \leq n $, then $ \kappa _{i} \geq 0 $.

{\rm (2)} If $ i \in X $, then
$$ \frac{1}{2\delta }(\psi _{i}^{+} + \psi _{i+1}^{-}) + \omega _{i} \chi _{i}^{+} + (2 - \omega _{i+1}) \chi _{i+1}^{-} \geq 0 $$
for any $ \omega _{i} $ and $ \omega _{i+1} $ with $ 0 \leq \omega _{i} \leq 2 $ and $ 0 \leq \omega _{i+1} \leq 2 $.
\end{claim}

\begin{proof}
The part (1) is nothing else than the condition (d). To show (2), we realize first that
$$ \Big( - \frac{1}{2\delta } + \frac{2}{1 - \delta } \Big) o^{H} \geq \Big( - \frac{1}{2\delta } + 2 \Big) o $$
whenever $ 0 \leq o^{H} \leq o $ and $ o^{H} \geq (1 - 2\delta ) o $. This is clear when $ \delta \leq 1/4 $ because then $ -(1/2\delta ) + 2 \leq 0 $, and so
$$ \Big( - \frac{1}{2\delta } + \frac{2}{1 - \delta } \Big) o^{H} \geq \Big( - \frac{1}{2\delta } + 2 \Big) o^{H} \geq \Big( - \frac{1}{2\delta } + 2 \Big) o. $$
When $ \delta > 1/4 $, then, using $ - (1/2\delta ) + 2/(1 - \delta ) \geq - (1/2\delta ) + 2 \geq 0 $, we can compute
$$ \Big( - \frac{1}{2\delta } + \frac{2}{1 - \delta } \Big) o^{H} \geq \Big( - \frac{1}{2\delta } + \frac{2}{1 - \delta } \Big) (1 - 2\delta ) o = \Big( - \frac{1}{2\delta } + 2 + \frac{1 - 3\delta }{1 - \delta } \Big) o \geq \Big( - \frac{1}{2\delta } + 2 \Big) o. $$
Further, if $ 0 \leq x \leq 2 $, then
$$ \frac{1}{2\delta }(o - o^{H}) + x \Big( \frac{1}{1 - \delta } o^{H} - o \Big) \geq 0, $$
as
\begin{eqnarray*}
\Big( - \frac{1}{2\delta } + \frac{x}{1 - \delta } \Big) o^{H} & = & - \Big( 1 - \frac{x}{2} \Big) \frac{1}{2\delta } o^{H} + \frac{x}{2} \Big( - \frac{1}{2\delta } + \frac{2}{1 - \delta } \Big) o^{H} \\
 & \geq & - \Big( 1 - \frac{x}{2} \Big) \frac{1}{2\delta } o + \frac{x}{2} \Big( - \frac{1}{2\delta } + 2 \Big) o \\
 & = & \Big( - \frac{1}{2\delta } + x \Big) o.
\end{eqnarray*}

Let us put
$$ k = \lambda (a_{i+1} - \alpha _{i+1}, b_{i+1} - \beta _{i+1}), \quad k^{H} = \lambda (H \cap (a_{i+1} - \alpha _{i+1}, b_{i+1} - \beta _{i+1})), $$
$$ l = \lambda (a_{i} + \alpha _{i}, b_{i} + \beta _{i}), \quad l^{H} = \lambda (H \cap (a_{i} + \alpha _{i}, b_{i} + \beta _{i})). $$
We have $ b_{i+1} - \beta _{i+1} \leq a_{i} + \alpha _{i} $, as $ i \in X $. We obtain
\begin{eqnarray*}
\psi _{i}^{+} & = & \lambda ((v_{i}, a_{i} + \alpha _{i}) \setminus H) = \lambda ((a_{i+1} - \alpha _{i+1}, a_{i} + \alpha _{i}) \setminus H) \\
 & \geq & \lambda ((a_{i+1} - \alpha _{i+1}, b_{i+1} - \beta _{i+1}) \setminus H) = k - k^{H}, \\
\psi _{i+1}^{-} & = & \lambda ((b_{i+1} - \beta _{i+1}, u_{i+1}) \setminus H) = \lambda ((b_{i+1} - \beta _{i+1}, b_{i} + \beta _{i}) \setminus H) \\
 & \geq & \lambda ((a_{i} + \alpha _{i}, b_{i} + \beta _{i}) \setminus H) = l - l^{H}.
\end{eqnarray*}
By (f), we have $ k^{H} \geq (1 - 2\delta )k $ and $ l^{H} \geq (1 - 2\delta )l $. So, we obtain
\begin{eqnarray*}
\frac{1}{2\delta } \psi _{i}^{+} + (2 - \omega _{i+1}) \chi _{i+1}^{-} & \geq & \frac{1}{2\delta }(k - k^{H}) + (2 - \omega _{i+1}) \Big( \frac{1}{1 - \delta } k^{H} - k \Big) \geq 0, \\
\frac{1}{2\delta } \psi _{i+1}^{-} + \omega _{i} \chi _{i}^{+} & \geq & \frac{1}{2\delta }(l - l^{H}) + \omega _{i} \Big( \frac{1}{1 - \delta } l^{H} - l \Big) \geq 0.
\end{eqnarray*}
\end{proof}

\begin{claim} \label{cly}
Let $ 1 \leq c \leq d \leq n - 1 $ be such that $ \{ c, c + 1, \dots , d \} \subset Y $. Then there are numbers $ \omega _{c}, \omega _{c+1}, \dots , \omega _{d+1} $ with $ 0 \leq \omega _{i} \leq 2 $ such that
$$ \sum_{i=c}^{d+1} \kappa _{i} + \sum_{i=c}^{d} \Big[ \frac{1}{2\delta }(\psi _{i}^{+} + \psi _{i+1}^{-}) + \omega _{i} \chi _{i}^{+} + (2 - \omega _{i+1}) \chi _{i+1}^{-} \Big] \geq 0. $$
\end{claim}

\begin{proof}
We consider two cases.

1) Let $ c = d $. We put
$$ \omega _{c} = 0, \quad \omega _{c+1} = 2, $$
so we need to check that
$$ \kappa _{c} + \kappa _{c+1} + \frac{1}{2\delta }(\psi _{c}^{+} + \psi _{c+1}^{-}) \geq 0. $$
But this is clear, due to Claim \ref{clx}(1).

2) Let $ c < d $. We put
$$ \omega _{c} = \delta , \quad \omega _{d+1} = 2 - \delta , $$
$$ \omega _{i} = 1, \; c + 1 \leq i \leq d. $$
Let us denote
$$ \Theta = \sum_{i=c}^{d+1} \kappa _{i} + \sum_{i=c}^{d} \Big[ \frac{1}{2\delta }(\psi _{i}^{+} + \psi _{i+1}^{-}) + \omega _{i} \chi _{i}^{+} + (2 - \omega _{i+1}) \chi _{i+1}^{-} \Big] , $$
so we have (using Claim \ref{clx}(1))
\begin{eqnarray*}
\Theta & \geq & \sum_{i=c}^{d+1} \kappa _{i} + \sum_{i=c}^{d} \Big[ \omega _{i} \chi _{i}^{+} + (2 - \omega _{i+1}) \chi _{i+1}^{-} \Big] \\
 & = & \sum_{i=c}^{d+1} \kappa _{i} + \omega _{c} \chi _{c}^{+} + \sum_{i=c+1}^{d} \omega _{i} \chi _{i}^{+} + \sum_{i=c+1}^{d} (2 - \omega _{i}) \chi _{i}^{-} + (2 - \omega _{d+1}) \chi _{d+1}^{-} \\
 & = & \kappa _{c} + \kappa _{d+1} + \delta \chi _{c}^{+} + \delta \chi _{d+1}^{-} + \sum_{i=c+1}^{d} \Big[ \kappa _{i} + \chi _{i}^{-} + \chi _{i}^{+} \Big] \\
 & \geq & 2\delta \kappa _{c} + 2\delta \kappa _{d+1} + \delta \chi _{c}^{+} + \delta \chi _{d+1}^{-} + \sum_{i=c+1}^{d} \Big[ \kappa _{i} + \chi _{i}^{-} + \chi _{i}^{+} \Big] .
\end{eqnarray*}

We have $ a_{i} + \alpha _{i} < b_{i+1} - \beta _{i+1} $ when $ c \leq i \leq d $, as $ i \in Y $. We put
$$ k_{i} = \lambda (u_{i}, v_{i}), \quad k_{i}^{H} = \lambda (H \cap (u_{i}, v_{i})), \quad c \leq i \leq d + 1, $$
$$ l_{i} = \lambda (v_{i}, a_{i} + \alpha _{i}), \quad l_{i}^{H} = \lambda (H \cap (v_{i}, a_{i} + \alpha _{i})), \quad c \leq i \leq d, $$
$$ m_{i} = \lambda (a_{i} + \alpha _{i}, b_{i+1} - \beta _{i+1}), \quad m_{i}^{H} = \lambda (H \cap (a_{i} + \alpha _{i}, b_{i+1} - \beta _{i+1})), \quad c \leq i \leq d, $$
$$ n_{i} = \lambda (b_{i+1} - \beta _{i+1}, u_{i+1}), \quad n_{i}^{H} = \lambda (H \cap (b_{i+1} - \beta _{i+1}, u_{i+1})), \quad c \leq i \leq d. $$
Note that
$$ b_{i} - a_{i} = \frac{1}{2} \big( \lambda (a_{i} - \alpha _{i}, b_{i} - \beta _{i}) + \lambda (a_{i} + \alpha _{i}, b_{i} + \beta _{i}) \big) , \quad 1 \leq i \leq n, $$
and recall that $ a_{i} - \alpha _{i} = v_{i-1}, 2 \leq i \leq n, $ and $ b_{i} + \beta _{i} = u_{i+1}, 1 \leq i \leq n - 1 $. We have
$$ \kappa _{i} = \frac{1}{2}(l_{i-1} + m_{i-1} + m_{i} + n_{i}) - k_{i} + \frac{2}{1 - \delta } k_{i}^{H}, \quad c + 1 \leq i \leq d, $$
$$ \kappa _{c} \geq \frac{1}{2}(m_{c} + n_{c}) - k_{c} + \frac{2}{1 - \delta } k_{c}^{H}, $$
$$ \kappa _{d+1} \geq \frac{1}{2}(l_{d} + m_{d}) - k_{d+1} + \frac{2}{1 - \delta } k_{d+1}^{H}, $$
$$ \chi _{i}^{-} = \frac{1}{1 - \delta } (l_{i-1}^{H} + m_{i-1}^{H}) - (l_{i-1} + m_{i-1}), \quad c + 1 \leq i \leq d + 1, $$
$$ \chi _{i}^{+} = \frac{1}{1 - \delta } (m_{i}^{H} + n_{i}^{H}) - (m_{i} + n_{i}), \quad c \leq i \leq d. $$
If $ c + 1 \leq i \leq d $, then, using (c), we obtain
$$ \frac{1}{2} \frac{1}{1 - \delta } \big( \lambda (H \cap I_{\alpha _{i}}(a_{i})) + \lambda (H \cap I_{\beta _{i}}(b_{i})) \big) \geq \frac{1}{2} \big( \lambda I_{\alpha _{i}}(a_{i}) + \lambda I_{\beta _{i}}(b_{i}) \big) , $$
and so
$$
\begin{aligned}
 & \hspace{-1cm} \frac{1}{2} \frac{1}{1 - \delta } (l_{i-1}^{H} + m_{i-1}^{H} + 2n_{i-1}^{H} + 2k_{i}^{H} + 2l_{i}^{H} + m_{i}^{H} + n_{i}^{H}) \\
 & + \kappa _{i} + \chi _{i}^{-} + \chi _{i}^{+} \\
 \geq \; & \frac{1}{2} (l_{i-1} + m_{i-1} + 2n_{i-1} + 2k_{i} + 2l_{i} + m_{i} + n_{i}) \\
 & + \frac{1}{2}(l_{i-1} + m_{i-1} + m_{i} + n_{i}) - k_{i} + \frac{2}{1 - \delta } k_{i}^{H} \\
 & + \frac{1}{1 - \delta } (l_{i-1}^{H} + m_{i-1}^{H}) - (l_{i-1} + m_{i-1}) + \frac{1}{1 - \delta } (m_{i}^{H} + n_{i}^{H}) - (m_{i} + n_{i}).
\end{aligned}
$$
This leads straightforwardly to
$$
\begin{aligned}
 & \hspace{-0.5cm} \kappa _{i} + \chi _{i}^{-} + \chi _{i}^{+} \\
 \geq \; & n_{i-1} + l_{i} + \frac{1}{2} \frac{1}{1 - \delta } (l_{i-1}^{H} + m_{i-1}^{H} + m_{i}^{H} + n_{i}^{H}) - \frac{1}{1 - \delta } (n_{i-1}^{H} + l_{i}^{H}) + \frac{1}{1 - \delta } k_{i}^{H} \\
 \geq \; & n_{i-1} + l_{i} + \frac{1}{2} \frac{1}{1 - \delta } (l_{i-1}^{H} + n_{i}^{H}) - \frac{1}{1 - \delta } (n_{i-1}^{H} + l_{i}^{H}).
\end{aligned}
$$
Summing these inequalities for $ i = c + 1, c + 2, \dots , d $, we obtain
$$
\begin{aligned}
 & \hspace{-0.5cm} \sum_{i=c+1}^{d} \Big[ \kappa _{i} + \chi _{i}^{-} + \chi _{i}^{+} \Big] \\
 \geq \; & \sum_{i=c+1}^{d} \Big[ n_{i-1} + l_{i} + \frac{1}{2} \frac{1}{1 - \delta } (l_{i-1}^{H} + n_{i}^{H}) - \frac{1}{1 - \delta } (n_{i-1}^{H} + l_{i}^{H}) \Big] \\
 = \; & \sum_{i=c+1}^{d-1} \Big[ n_{i} + l_{i} + \frac{1}{2} \frac{1}{1 - \delta } (l_{i}^{H} + n_{i}^{H}) - \frac{1}{1 - \delta } (n_{i}^{H} + l_{i}^{H}) \Big] \\
 & + n_{c} + l_{d} + \frac{1}{2} \frac{1}{1 - \delta } (l_{c}^{H} + n_{d}^{H}) - \frac{1}{1 - \delta } (n_{c}^{H} + l_{d}^{H}) \\
 \geq \; & n_{c} + l_{d} + \frac{1}{2} \frac{1}{1 - \delta } (l_{c}^{H} + n_{d}^{H}) - \frac{1}{1 - \delta } (n_{c}^{H} + l_{d}^{H}),
\end{aligned}
$$
as
$$
\begin{aligned}
 & \hspace{-1cm} \sum_{i=c+1}^{d-1} \Big[ n_{i} + l_{i} + \frac{1}{2} \frac{1}{1 - \delta } (l_{i}^{H} + n_{i}^{H}) - \frac{1}{1 - \delta } (n_{i}^{H} + l_{i}^{H}) \Big] \\
 \geq \; & \sum_{i=c+1}^{d-1} \Big[ n_{i}^{H} + l_{i}^{H} + \frac{1}{2} \frac{1}{1 - \delta } (l_{i}^{H} + n_{i}^{H}) - \frac{1}{1 - \delta } (n_{i}^{H} + l_{i}^{H}) \Big] \\
 = \; & \Big( 1 - \frac{1}{2} \frac{1}{1 - \delta } \Big) \sum_{i=c+1}^{d-1} (l_{i}^{H} + n_{i}^{H}) \geq 0.
\end{aligned}
$$

Since $ c \in Y $, we have $ a_{c} + \alpha _{c} < b_{c+1} - \beta _{c+1} \leq u_{c+1} = b_{c} + \beta _{c} $. In particular, $ I_{\alpha _{c}}(a_{c}) \neq I_{\beta _{c}}(b_{c}) $. Since $ d \in Y $, we have $ a_{d+1} - \alpha _{d+1} = v_{d} \leq a_{d} + \alpha _{d} < b_{d+1} - \beta _{d+1} $. In particular, $ I_{\alpha _{d+1}}(a_{d+1}) \neq I_{\beta _{d+1}}(b_{d+1}) $. Using (e),
$$ k_{c}^{H} + l_{c}^{H} \geq (1 - 2\delta )(k_{c} + l_{c}), $$
$$ n_{d}^{H} + k_{d+1}^{H} \geq (1 - 2\delta )(n_{d} + k_{d+1}). $$
Now, we compute
\begin{eqnarray*}
\Theta & \geq & \Theta + 2\delta (k_{c} + l_{c}) - \frac{2\delta }{1 - 2\delta } (k_{c}^{H} + l_{c}^{H}) + 2\delta (n_{d} + k_{d+1}) - \frac{2\delta }{1 - 2\delta }(n_{d}^{H} + k_{d+1}^{H}) \\
 & \geq & 2\delta \kappa _{c} + 2\delta \kappa _{d+1} + \delta \chi _{c}^{+} + \delta \chi _{d+1}^{-} + \sum_{i=c+1}^{d} \Big[ \kappa _{i} + \chi _{i}^{-} + \chi _{i}^{+} \Big] \\
 & & + 2\delta (k_{c} + l_{c}) - \frac{2\delta }{1 - 2\delta } (k_{c}^{H} + l_{c}^{H}) + 2\delta (n_{d} + k_{d+1}) - \frac{2\delta }{1 - 2\delta }(n_{d}^{H} + k_{d+1}^{H})
\end{eqnarray*}
%divided because of pagebreak
\begin{eqnarray*}
 & \geq & 2\delta \Big( \frac{1}{2}(m_{c} + n_{c}) - k_{c} + \frac{2}{1 - \delta } k_{c}^{H} \Big) + 2\delta \Big( \frac{1}{2}(l_{d} + m_{d}) - k_{d+1} + \frac{2}{1 - \delta } k_{d+1}^{H} \Big) \\
 & & + \delta \Big( \frac{1}{1 - \delta } (m_{c}^{H} + n_{c}^{H}) - (m_{c} + n_{c}) \Big) + \delta \Big( \frac{1}{1 - \delta } (l_{d}^{H} + m_{d}^{H}) - (l_{d} + m_{d}) \Big) \\
 & & + n_{c} + l_{d} + \frac{1}{2} \frac{1}{1 - \delta } (l_{c}^{H} + n_{d}^{H}) - \frac{1}{1 - \delta } (n_{c}^{H} + l_{d}^{H}) \\
 & & + 2\delta (k_{c} + l_{c}) - \frac{2\delta }{1 - 2\delta } (k_{c}^{H} + l_{c}^{H}) + 2\delta (n_{d} + k_{d+1}) - \frac{2\delta }{1 - 2\delta }(n_{d}^{H} + k_{d+1}^{H}) \\
 & = & 2\delta (l_{c} + n_{d}) - \Big( \frac{2\delta }{1 - 2\delta } - \frac{1}{2} \frac{1}{1 - \delta } \Big) (l_{c}^{H} + n_{d}^{H}) + n_{c} - n_{c}^{H} + l_{d} - l_{d}^{H} \\
 & & + 2\delta \Big( \frac{2}{1 - \delta } - \frac{1}{1 - 2\delta } \Big) (k_{c}^{H} + k_{d+1}^{H}) + \frac{\delta }{1 - \delta } (m_{c}^{H} + m_{d}^{H}) \\
 & \geq & 2\delta (l_{c}^{H} + n_{d}^{H}) - \Big( \frac{2\delta }{1 - 2\delta } - \frac{1}{2} \frac{1}{1 - \delta } \Big) (l_{c}^{H} + n_{d}^{H}).
\end{eqnarray*}
We obtain finally $ \Theta \geq 0 $ from the assumption $ \delta < \zeta _{6} $.
\end{proof}

\begin{claim} \label{clxy}
There are numbers $ \omega _{i}, 1 \leq i \leq n, $ with $ 0 \leq \omega _{i} \leq 2 $ such that
$$ \sum_{i=1}^{n} \kappa _{i} + \sum_{i=1}^{n-1} \Big[ \frac{1}{2\delta }(\psi _{i}^{+} + \psi _{i+1}^{-}) + \omega _{i} \chi _{i}^{+} + (2 - \omega _{i+1}) \chi _{i+1}^{-} \Big] \geq 0. $$
\end{claim}

\begin{proof}
We write
$$ Y = \bigcup_{j=1}^{k} \{ c_{j}, c_{j} + 1, \dots , d_{j} \} $$
where $ k $ and $ c_{j}, d_{j}, 1 \leq j \leq k, $ are the uniquely determined natural numbers such that $ c_{j} \leq d_{j}, 1 \leq j \leq k, $ and $ d_{j} + 1 < c_{j+1}, 1 \leq j \leq k - 1 $.
For every $ j $ with $ 1 \leq j \leq k $, we choose, using Claim \ref{cly}, numbers $ \omega _{c_{j}}, \omega _{c_{j}+1}, \dots , \omega _{d_{j}+1} $ with $ 0 \leq \omega _{i} \leq 2 $ such that
$$ \sum_{i=c_{j}}^{d_{j}+1} \kappa _{i} + \sum_{i=c_{j}}^{d_{j}} \Big[ \frac{1}{2\delta }(\psi _{i}^{+} + \psi _{i+1}^{-}) + \omega _{i} \chi _{i}^{+} + (2 - \omega _{i+1}) \chi _{i+1}^{-} \Big] \geq 0. $$
The numbers $ \omega _{i} $ may not be chosen for some $ i $'s yet. For every such an $ i $, we choose $ \omega _{i} $ to be an arbitrary number with $ 0 \leq \omega _{i} \leq 2 $.

Using Claim \ref{clx}, we compute
$$
\begin{aligned}
 & \hspace{-1cm} \sum_{i=1}^{n} \kappa _{i} + \sum_{i=1}^{n-1} \Big[ \frac{1}{2\delta }(\psi _{i}^{+} + \psi _{i+1}^{-}) + \omega _{i} \chi _{i}^{+} + (2 - \omega _{i+1}) \chi _{i+1}^{-} \Big] \\
\geq \; & \sum_{j=1}^{k} \Bigg\{ \sum_{i=c_{j}}^{d_{j}+1} \kappa _{i} + \sum_{i=c_{j}}^{d_{j}} \Big[ \frac{1}{2\delta }(\psi _{i}^{+} + \psi _{i+1}^{-}) + \omega _{i} \chi _{i}^{+} + (2 - \omega _{i+1}) \chi _{i+1}^{-} \Big] \Bigg\} \\
\geq \; & 0.
\end{aligned}
$$
\end{proof}

Now, combining Claims \ref{clc} and \ref{clxy}, we obtain
$$ \Big( \frac{1}{2\delta } + \frac{2}{1 - \delta } \Big) 2\lambda H \geq \Big( \frac{1}{2\delta } + 1 \Big) 2(q - p), $$
which, after a little calculation, completes the proof of Lemma \ref{lemmaxy}.

\begin{proposition} \label{goodsetbound}
Let $ 0 < \delta < \zeta _{2} $ where $ \zeta _{2} $ is as in Table \ref{roots}. If
$$ G = [0, \nu _{1}) \cup (\mu _{2}, \nu _{2}) \cup \dots \cup (\mu _{r-1}, \nu _{r-1}) \cup (\mu _{r}, 1] $$
is a $ \delta $-good set, then
$$ \lambda G \geq \frac{(1 - \delta )(1 + 2\delta )}{1 + 3\delta }. $$
\end{proposition}

Note that the inequality cannot be improved. One may show that, if we take
$$ G = [0, \varphi /2) \cup (\varphi /2 + \psi, \varphi /2 + \psi + \alpha) \cup (1 - \varphi /2 - \psi - \alpha , 1 - \varphi /2 - \psi) \cup (1 - \varphi /2, 1], $$ 
(i.e., $ G = (S_{1} - \varphi /2 ) \cap [0, 1] $), where $ \alpha , \beta , \varphi , \psi $ and $ S_{1} $ are as in the proof of Proposition \ref{ubound}, then $ G $ is (under some assumptions on $ \delta $) a $ \delta $-good set and
$$ \lambda G = \varphi + 2\alpha = \frac{(1 - \delta )(1 + 2\delta )}{1 + 3\delta } . $$

\begin{proof}
Let us put
$$ H = G + \mathbb{Z}. $$
We may assume that $ \lambda G < 1 - \delta $. There is a constant $ R \in \mathbb{N} $ such that $ \varepsilon \leq R $ whenever $ I_{\varepsilon }(c) $ is an interval with $ \lambda (H | I_{\varepsilon }(c)) \geq 1 - \delta $. Let $ p, q \in \mathbb{Z} $ be an arbitrary couple such that $ p < q $. By Lemma \ref{lemmae}(3), $ (p, q) $ is covered by a finite number of intervals $ I $ with $ \lambda (H | I) \geq 1 - \delta $. By Lemma \ref{keyprop}, the condition (C) from Lemma \ref{lemmaxy} is satisfied for $ H $ and for $ p $ and $ q $. Moreover, the intervals $ I_{\alpha }(a) $ and $ I_{\beta }(b) $ intersect $ (p, q) $, and so they are subsets of $ (p - 2R, q + 2R) $. Actually, the condition (C) is satisfied for the set $ H \cap (p - 2R, q + 2R) $. We obtain from Lemma \ref{lemmaxy} that
$$ \lambda (H \cap (p - 2R, q + 2R)) \geq \frac{(1 - \delta)(1 + 2\delta)}{1 + 3\delta} (q - p), $$
i.e.,
$$ (q - p + 4R) \lambda G \geq \frac{(1 - \delta)(1 + 2\delta)}{1 + 3\delta} (q - p). $$
Therefore,
$$ \lambda G \geq \frac{(1 - \delta)(1 + 2\delta)}{1 + 3\delta} \frac{n}{n + 4R} $$
for every $ n \in \mathbb{N} $, and it is sufficient to realize that $ \lim _{n \rightarrow \infty } n/(n + 4R) = 1 $.
\end{proof}

\begin{proposition} \label{lbound}
We have $ \delta _{\mathcal{K}} \geq \zeta _{1} $ where $ \zeta _{1} $ is the only real root of the polynomial $ 8x^{3} + 8x^{2} + x - 1 $. 
\end{proposition}

\begin{proof}
We need to prove the implication
$$ \delta > \delta _{\mathcal{K}} \quad \Rightarrow \quad \delta \geq \zeta _{1}. $$
So, let $ \delta > \delta _{\mathcal{K}} $. We may assume that $ \delta < \zeta _{2} $ where $ \zeta _{2} $ is as in Table \ref{roots} (as $ \zeta _{2} > \zeta _{1} $). Since $ \delta < \zeta _{2} \leq \zeta _{3} $, there is, by Lemma \ref{goodset}, a $ \delta $-good set
$$ G = [0, \nu _{1}) \cup (\mu _{2}, \nu _{2}) \cup \dots \cup (\mu _{r-1}, \nu _{r-1}) \cup (\mu _{r}, 1] $$
such that
$$ \lambda G \leq \frac{4\delta^{2}}{1 - 2\delta}. $$
By Proposition \ref{goodsetbound}, we have
$$ \lambda G \geq \frac{(1 - \delta)(1 + 2\delta)}{1 + 3\delta}. $$
Combining the bounds, we obtain
$$ \frac{(1 - \delta)(1 + 2\delta)}{1 + 3\delta} \leq \lambda G \leq \frac{4\delta^{2}}{1 - 2\delta}. $$
Therefore, the inequality from Table \ref{roots} implied by $ \delta < \zeta _{1} $ is not fulfilled. It follows that $ \delta \geq \zeta _{1} $.
\end{proof}

The author thanks Professor Petr Holick\'y for valuable remarks on preliminary versions of this work.

\end{document}